\documentclass{amsart}

\usepackage{amssymb}
\usepackage{amsmath}
\usepackage{braket}
\usepackage{mathrsfs}
\usepackage[all]{xy}
\usepackage{graphicx}
\usepackage{color}

\newtheorem{theorem}{Theorem}[section]
\newtheorem{remark}[theorem]{Remark}
\newtheorem{fact}[theorem]{Fact}
\newtheorem{caution}[theorem]{Caution}
\newtheorem{condition}[theorem]{Condition}
\newtheorem{claim}[theorem]{Claim}
\theoremstyle{definition}
\newtheorem{definition}[theorem]{Definition}

\newtheorem{corollary}[theorem]{Corollary}
\newtheorem{proposition}[theorem]{Proposition}
\newtheorem{example}[theorem]{Example}
\newtheorem{lemma}[theorem]{Lemma}

\begin{document}

\title[Cotorsion pairs in extriangulated categories.]{Mutation via Hovey twin cotorsion pairs and model structures in extriangulated categories.}

\author{Hiroyuki Nakaoka}
\email{nakaoka@sci.kagoshima-u.ac.jp} 
\address{Research and Education Assembly, Science and Engineering Area, Research Field in Science, Kagoshima University, 1-21-35 Korimoto, Kagoshima, 890-0065 Japan}
\author{Yann Palu}
\email{yann.palu@u-picardie.fr}
\address{LAMFA, Universit\'{e} de Picardie Jules Verne, 33 rue Saint-Leu, 80000 Amiens, France}
%

\thanks{Both authors would like to thank Professor Osamu Iyama and Professor Ivo Dell'Ambrogio for inspiring comments on a previous version of the paper.}
\thanks{The first author wishes to thank Professor Frederik Marks and Professor Jorge Vit\'{o}ria for their interests, and for giving him several opportunities.}
\thanks{The first author is supported by JSPS KAKENHI Grant Numbers 25800022.}

\begin{abstract}
We give a simultaneous generalization of exact categories and triangulated categories, which is suitable for considering cotorsion pairs, and which we call extriangulated categories.  Extension-closed, full subcategories of triangulated categories are examples of extriangulated categories. We give a bijective correspondence between some pairs of cotorsion pairs which we call Hovey twin cotorsion pairs, and admissible model structures. As a consequence, these model structures relate certain localizations with certain ideal quotients, via the homotopy category which can be given a triangulated structure. This gives a natural framework to formulate reduction and mutation of cotorsion pairs, applicable to both exact categories and triangulated categories.
These results can be thought of as arguments towards the view that extriangulated categories are a convenient setup for writing down proofs which apply to both exact categories and (extension-closed subcategories of) triangulated categories.
\end{abstract}

\maketitle

\tableofcontents

\section{Introduction and Preliminaries}

Cotorsion pairs, first introduced in \cite{Sal}, are defined on an exact category or a triangulated category, and are related to several homological structures, such as: $t$-structures \cite{BBD}, cluster tilting subcategories \cite{KR,KZ}, co-$t$-structures~\cite{Pau}, functorially finite rigid subcategories. A careful look reveals that what is necessary to define a cotorsion pair on a category is the existence of an $\mathrm{Ext}^1$ bifunctor with appropriate properties. In this article, we formalize the notion of an {\it extriangulated category} by extracting those properties of $\mathrm{Ext}^1$ on exact categories and on triangulated categories that seem relevant from the point-of-view of cotorsion pairs.

The class of extriangulated categories not only contains exact categories and extension-closed subcategories of triangulated categories as examples, but it is also closed under taking some ideal quotients (Proposition \ref{PropReduction}). This will allow us to construct an extriangulated category which is not exact nor triangulated. Moreover, this axiomatization rams down the problem of the {\it non-existence of a canonical choice of the middle arrow in the axiom (TR3)} to the ambiguity of a representative of realizing sequences (Section~\ref{section_reali}) and the exactness of the associated sequences of natural transformations (Proposition \ref{PropExact1}).

Let us motivate a bit more the use of extriangulated categories. Many results which are homological in nature apply (after suitable adaptation) to both setups: exact categories and triangulated categories. In order to transfer a result known for triangulated categories to a result that applies to exact categories, the usual strategy is the following (non-chronological):
\begin{enumerate}
 \item[(1)] Specify to the case of stable categories of Frobenius exact categories.
 \item[(2)] Lift all definitions and statements from the stable category to the Frobenius category.
 \item[(3)] Adapt the proof so that it applies to any exact categories (with suitable assumptions).
\end{enumerate}
Conversely, a result known for exact categories might have an analog for triangulated categories proven as follows:
\begin{enumerate}
 \item[(1)] Specify to the case of a Frobenius exact category.
 \item[(2)] Descend all definitions and statements to the stable category.
 \item[(3)] Adapt the proof so that it applies to any triangulated categories (with suitable assumptions).
\end{enumerate}
Even though step (2) might be far from trivial, the main difficulty often lies in step (3) for both cases. The use of extriangulated categories somehow removes that difficulty. It is not that this difficulty has vanished into thin air, but that it has already been taken care of in the first results on extriangulated categories obtained in Section~\ref{section_Fund}.

The term ``extriangulated" stands for {\it externally triangulated} by means of a bifunctor. It can also be viewed as the mixing of {\it exact} and {\it triangulated}, or as an abbreviation of {\it $\mathrm{Ext}$-triangulated}. The precise definition will be given in Section~\ref{section_Ex}.
Fundamental properties including several analogs of the octahedron axiom in an extriangulated category, will be given in Section~\ref{section_Fund}.

On an extriangulated category, we can define the notion of a cotorsion pair, which generalizes that on exact categories \cite{Ho1,Ho2,Liu,S} and on triangulated categories \cite{AN,Na1}. Basic properties will be stated in Section~\ref{section_Cot}.

In Section~\ref{section: model}, we give a bijective correspondence between {\it Hovey twin cotorsion pairs} and {\it admissible model structures}. This result is inspired from~\cite{Ho1,Ho2,G}, where the case of exact categories is studied in more details. We note that an analog of Hovey's result in~\cite{Ho1,Ho2} has been proven for triangulated categories in~\cite{Y}. The use of extriangulated categories allows for a uniform proof.

As a result, we can realize the associated homotopy category by a certain ideal quotient, on which we can give a triangulated structure as in Section~\ref{section_Tria}. This triangulation can be regarded as a simultaneous generalization of those given by Happel's theorem on stable categories of Frobenius exact categories and of Iyama-Yoshino reductions, and gives a link to the one given by the Verdier quotient.
As a consequence, the homotopy category of any exact model structure on a weakly idempotent complete exact category is triangulated. This result was previously known in the case of hereditary exact model structures \cite[Proposition 5.2]{G}.

With this view, in Section~\ref{section_Mut}, we propose a natural framework to formulate reduction and mutation of cotorsion pairs, applicable to both exact categories and triangulated categories. Indeed, we establish a bijective correspondence between the class of {\it mutable cotorsion pairs} associated with a Hovey twin cotorsion pair and the class of all cotorsion pairs on the triangulated homotopy category.

\section{Extriangulated category}\label{section_Ex}
In this section, we abstract the properties of extension-closed subcategory of triangulated or exact category, to formulate it in an internal way by means of an $\mathrm{Ext}^1$ functor. This gives a simultaneous generalization of triangulated categories and exact categories, suitable for dealing with cotorsion pairs.

\subsection{$\mathbb{E}$-extensions}
Throughout this paper, let $\mathscr{C}$ be an additive category.
\begin{definition}\label{DefExtension}
Suppose $\mathscr{C}$ is equipped with a biadditive functor $\mathbb{E}\colon\mathscr{C}^\mathrm{op}\times\mathscr{C}\to\mathit{Ab}$. For any pair of objects $A,C\in\mathscr{C}$, an element $\delta\in\mathbb{E}(C,A)$ is called an {\it $\mathbb{E}$-extension}. Thus formally, an $\mathbb{E}$-extension is a triplet $(A,\delta,C)$.
\end{definition}

\begin{remark}
Let $(A,\delta,C)$ be any $\mathbb{E}$-extension. Since $\mathbb{E}$ is a bifunctor, for any $a\in\mathscr{C}(A,A^{\prime})$ and $c\in\mathscr{C}(C^{\prime},C)$, we have $\mathbb{E}$-extensions
\[ \mathbb{E}(C,a)(\delta)\in\mathbb{E}(C,A^{\prime})\ \ \text{and}\ \ \mathbb{E}(c,A)(\delta)\in\mathbb{E}(C^{\prime},A). \]
We abbreviately denote them by $a_{\ast}\delta$ and $c^{\ast}\delta$.
In this terminology, we have
\[ \mathbb{E}(c,a)(\delta)=c^{\ast} a_{\ast}\delta=a_{\ast} c^{\ast}\delta \]
in $\mathbb{E}(C^{\prime},A^{\prime})$.
\end{remark}

\begin{definition}\label{DefMorphExt}
Let $(A,\delta,C),(A^{\prime},\delta^{\prime},C^{\prime})$ be any pair of $\mathbb{E}$-extensions. A {\it morphism} $(a,c)\colon(A,\delta,C)\to(A^{\prime},\delta^{\prime},C^{\prime})$ of $\mathbb{E}$-extensions is a pair of morphisms $a\in\mathscr{C}(A,A^{\prime})$ and $c\in\mathscr{C}(C,C^{\prime})$ in $\mathscr{C}$, satisfying the equality
\[ a_{\ast}\delta=c^{\ast}\delta^{\prime}. \]
Simply we denote it as $(a,c)\colon\delta\to\delta^{\prime}$.

We obtain the category $\mathbb{E}\text{-}\mathrm{Ext}(\mathscr{C})$ of $\mathbb{E}$-extensions, with composition and identities naturally induced from those in $\mathscr{C}$.
\end{definition}

\begin{remark}\label{RemMorphExt}
Let $(A,\delta,C)$ be any $\mathbb{E}$-extension. We have the following.
\begin{enumerate}
\item[(1)] Any morphism $a\in\mathscr{C}(A,A^{\prime})$ gives rise to a morphism of $\mathbb{E}$-extensions
\[ (a,\mathrm{id}_C)\colon\delta\to a_{\ast}\delta. \]
\item[(2)] Any morphism $c\in\mathscr{C}(C^{\prime},C)$ gives rise to a morphism of $\mathbb{E}$-extensions
\[ (\mathrm{id}_A,c)\colon c^{\ast}\delta\to\delta. \]
\end{enumerate}
\end{remark}

\begin{definition}\label{DefSplitExtension}
For any $A,C\in\mathscr{C}$, the zero element $0\in\mathbb{E}(C,A)$ is called the {\it split $\mathbb{E}$-extension}.
\end{definition}

\begin{definition}\label{DefSumExtension}
Let $\delta=(A,\delta,C),\delta^{\prime}=(A^{\prime},\delta^{\prime},C^{\prime})$ be any pair of $\mathbb{E}$-extensions. Let
\[ C\overset{\iota_C}{\longrightarrow}C\oplus C^{\prime}\overset{\iota_{C^{\prime}}}{\longleftarrow}C^{\prime} \]
and 
\[ A\overset{p_A}{\longleftarrow}A\oplus A^{\prime}\overset{p_{A^{\prime}}}{\longrightarrow}A^{\prime} \]
be coproduct and product in $\mathscr{C}$, respectively. Remark that, by the biadditivity of $\mathbb{E}$, we have a natural isomorphism
\[ \mathbb{E}(C\oplus C^{\prime},A\oplus A^{\prime})\cong \mathbb{E}(C,A)\oplus\mathbb{E}(C,A^{\prime})\oplus\mathbb{E}(C^{\prime},A)\oplus\mathbb{E}(C^{\prime},A^{\prime}). \]

Let $\delta\oplus\delta^{\prime}\in\mathbb{E}(C\oplus C^{\prime},A\oplus A^{\prime})$ be the element corresponding to $(\delta,0,0,\delta^{\prime})$ through this isomorphism. This is the unique element which satisfies
\begin{eqnarray*}
\mathbb{E}(\iota_C,p_A)(\delta\oplus\delta^{\prime})=\delta&,&\mathbb{E}(\iota_C,p_{A^{\prime}})(\delta\oplus\delta^{\prime})=0,\\
\mathbb{E}(\iota_{C^{\prime}},p_A)(\delta\oplus\delta^{\prime})=0&,&\mathbb{E}(\iota_{C^{\prime}},p_{A^{\prime}})(\delta\oplus\delta^{\prime})=\delta^{\prime}.
\end{eqnarray*}

If $A=A^{\prime}$ and $C=C^{\prime}$, then the sum $\delta+\delta^{\prime}\in\mathbb{E}(C,A)$ of $\delta,\delta^{\prime}\in\mathbb{E}(C,A)$ is obtained by
\[ \delta+\delta^{\prime}=\mathbb{E}(\Delta_C,\nabla_A)(\delta\oplus\delta^{\prime}), \]
where $\Delta_C=\Big[\raise1ex\hbox{\leavevmode\vtop{\baselineskip-8ex \lineskip1ex \ialign{#\crcr{$1$}\crcr{$1$}\crcr}}}\Big]\colon C\to C\oplus C$, $\nabla_A=[1\ 1]\colon A\oplus A\to A$.
\end{definition}

\subsection{Realization of $\mathbb{E}$-extensions}\label{section_reali}
Let $\mathscr{C},\mathbb{E}$ be as before.
\begin{definition}\label{DefSqEquiv}
Let $A,C\in\mathscr{C}$ be any pair of objects. Sequences of morphisms in $\mathscr{C}$
\[ A\overset{x}{\longrightarrow}B\overset{y}{\longrightarrow}C\ \ \text{and}\ \ A\overset{x^{\prime}}{\longrightarrow}B^{\prime}\overset{y^{\prime}}{\longrightarrow}C \]
are said to be {\it equivalent} if there exists an isomorphism $b\in\mathscr{C}(B,B^{\prime})$ which makes the following diagram commutative.
\[
\xy
(-16,0)*+{A}="0";
(3,0)*+{}="1";
(0,8)*+{B}="2";
(0,-8)*+{B^{\prime}}="4";
(-3,0)*+{}="5";
(16,0)*+{C}="6";
{\ar^{x} "0";"2"};
{\ar^{y} "2";"6"};
{\ar_{x^{\prime}} "0";"4"};
{\ar_{y^{\prime}} "4";"6"};
{\ar^{b}_{\cong} "2";"4"};
{\ar@{}|\circlearrowright "0";"1"};
{\ar@{}|\circlearrowright "5";"6"};
\endxy
\]

We denote the equivalence class of $A\overset{x}{\longrightarrow}B\overset{y}{\longrightarrow}C$ by $[A\overset{x}{\longrightarrow}B\overset{y}{\longrightarrow}C]$.
\end{definition}

\begin{definition}\label{DefAddSeq}
$\ \ $
\begin{enumerate}
\item[(1)] For any $A,C\in\mathscr{C}$, we denote as
\[ 0=[A\overset{\Big[\raise1ex\hbox{\leavevmode\vtop{\baselineskip-8ex \lineskip1ex \ialign{#\crcr{$\scriptstyle{1}$}\crcr{$\scriptstyle{0}$}\crcr}}}\Big]}{\longrightarrow}A\oplus C\overset{[0\ 1]}{\longrightarrow}C]. \]

\item[(2)] For any $[A\overset{x}{\longrightarrow}B\overset{y}{\longrightarrow}C]$ and $[A^{\prime}\overset{x^{\prime}}{\longrightarrow}B^{\prime}\overset{y^{\prime}}{\longrightarrow}C^{\prime}]$, we denote as
\[ [A\overset{x}{\longrightarrow}B\overset{y}{\longrightarrow}C]\oplus [A^{\prime}\overset{x^{\prime}}{\longrightarrow}B^{\prime}\overset{y^{\prime}}{\longrightarrow}C^{\prime}]=[A\oplus A^{\prime}\overset{x\oplus x^{\prime}}{\longrightarrow}B\oplus B^{\prime}\overset{y\oplus y^{\prime}}{\longrightarrow}C\oplus C^{\prime}]. \]
\end{enumerate}
\end{definition}

\begin{definition}\label{DefRealization}
Let $\mathfrak{s}$ be a correspondence which associates an equivalence class $\mathfrak{s}(\delta)=[A\overset{x}{\longrightarrow}B\overset{y}{\longrightarrow}C]$ to any $\mathbb{E}$-extension $\delta\in\mathbb{E}(C,A)$. This $\mathfrak{s}$ is called a {\it realization} of $\mathbb{E}$, if it satisfies the following condition $(\ast)$. In this case, we say that sequence $A\overset{x}{\longrightarrow}B\overset{y}{\longrightarrow}C$ {\it realizes} $\delta$, whenever it satisfies $\mathfrak{s}(\delta)=[A\overset{x}{\longrightarrow}B\overset{y}{\longrightarrow}C]$.
\begin{itemize}
\item[$(\ast)$] Let $\delta\in\mathbb{E}(C,A)$ and $\delta^{\prime}\in\mathbb{E}(C^{\prime},A^{\prime})$ be any pair of $\mathbb{E}$-extensions, with $\mathfrak{s}(\delta)=[A\overset{x}{\longrightarrow}B\overset{y}{\longrightarrow}C],\, \mathfrak{s}(\delta^{\prime})=[A^{\prime}\overset{x^{\prime}}{\longrightarrow}B^{\prime}\overset{y^{\prime}}{\longrightarrow}C^{\prime}]$.
Then, for any morphism $(a,c)\in\mathbb{E}\text{-}\mathrm{Ext}(\mathscr{C})(\delta,\delta^{\prime})$, there exists $b\in\mathscr{C}(B,B^{\prime})$ which makes the following diagram commutative.
\begin{equation}\label{MorphRealize}
\xy
(-12,6)*+{A}="0";
(0,6)*+{B}="2";
(12,6)*+{C}="4";
(-12,-6)*+{A^{\prime}}="10";
(0,-6)*+{B^{\prime}}="12";
(12,-6)*+{C^{\prime}}="14";
{\ar^{x} "0";"2"};
{\ar^{y} "2";"4"};
{\ar_{a} "0";"10"};
{\ar^{b} "2";"12"};
{\ar^{c} "4";"14"};
{\ar_{x^{\prime}} "10";"12"};
{\ar_{y^{\prime}} "12";"14"};
{\ar@{}|\circlearrowright "0";"12"};
{\ar@{}|\circlearrowright "2";"14"};
\endxy
\end{equation}
\end{itemize}
Remark that this condition does not depend on the choices of the representatives of the equivalence classes. In the above situation, we say that $(\ref{MorphRealize})$ (or the triplet $(a,b,c)$) {\it realizes} $(a,c)$.
\end{definition}

\begin{definition}\label{DefAdditiveRealization}
Let $\mathscr{C},\mathbb{E}$ be as above. A realization of $\mathbb{E}$ is said to be {\it additive}, if it satisfies the following conditions.
\begin{itemize}
\item[{\rm (i)}] For any $A,C\in\mathscr{C}$, the split $\mathbb{E}$-extension $0\in\mathbb{E}(C,A)$ satisfies
\[ \mathfrak{s}(0)=0. \]
\item[{\rm (ii)}] For any pair of $\mathbb{E}$-extensions $\delta=(A,\delta,C)$ and $\delta^{\prime}=(A^{\prime},\delta^{\prime},C^{\prime})$,
\[ \mathfrak{s}(\delta\oplus\delta^{\prime})=\mathfrak{s}(\delta)\oplus\mathfrak{s}(\delta^{\prime}) \]
holds.
\end{itemize}
\end{definition}

\begin{remark}\label{RemSplit}
If $\mathfrak{s}$ is an additive realization of $\mathbb{E}$, then the following holds.
\begin{enumerate}
\item[(1)] For any $A,C\in\mathscr{C}$, if $0\in\mathbb{E}(C,A)$ is realized by $A\overset{x}{\longrightarrow}B\overset{y}{\longrightarrow}C$, then there exist a retraction $r\in\mathscr{C}(B,A)$ of $x$ and a section $s\in\mathscr{C}(C,B)$ of $y$ which give an isomorphism $\Big[\raise1ex\hbox{\leavevmode\vtop{\baselineskip-8ex \lineskip1ex \ialign{#\crcr{$r$}\crcr{$y$}\crcr}}}\Big]\colon B\overset{\cong}{\longrightarrow}A\oplus C$.
\item[(2)] For any $f\in\mathscr{C}(A,B)$, the sequence
\[ A\overset{\Big[\raise1ex\hbox{\leavevmode\vtop{\baselineskip-8ex \lineskip1ex \ialign{#\crcr{$\,\,\scriptstyle{1}\,$}\crcr{$\scriptstyle{-f}$}\crcr}}}\Big]}{\longrightarrow}A\oplus B\overset{[f\ 1]}{\longrightarrow}B \]
realizes the split $\mathbb{E}$-extension $0\in\mathbb{E}(B,A)$.
\end{enumerate}
\end{remark}

\subsection{Definition of extriangulated category}

\begin{definition}\label{DefExtCat}
We call the pair $(\mathbb{E},\mathfrak{s})$ an {\it external triangulation} of $\mathscr{C}$ if it satisfies the following conditions. In this case, we call $\mathfrak{s}$ an $\mathbb{E}$-{\it triangulation of }$\mathscr{C}$, and call the triplet $(\mathscr{C},\mathbb{E},\mathfrak{s})$ an {\it externally triangulated category}, or for short, {\it extriangulated category}.
\begin{itemize}
\item[{\rm (ET1)}] $\mathbb{E}\colon\mathscr{C}^{\mathrm{op}}\times\mathscr{C}\to\mathit{Ab}$ is a biadditive functor.
\item[{\rm (ET2)}] $\mathfrak{s}$ is an additive realization of $\mathbb{E}$.
\item[{\rm (ET3)}] Let $\delta\in\mathbb{E}(C,A)$ and $\delta^{\prime}\in\mathbb{E}(C^{\prime},A^{\prime})$ be any pair of $\mathbb{E}$-extensions, realized as
\[ \mathfrak{s}(\delta)=[A\overset{x}{\longrightarrow}B\overset{y}{\longrightarrow}C],\ \ \mathfrak{s}(\delta^{\prime})=[A^{\prime}\overset{x^{\prime}}{\longrightarrow}B^{\prime}\overset{y^{\prime}}{\longrightarrow}C^{\prime}]. \]
For any commutative square
\begin{equation}\label{SquareForET3}
\xy
(-12,6)*+{A}="0";
(0,6)*+{B}="2";
(12,6)*+{C}="4";
(-12,-6)*+{A^{\prime}}="10";
(0,-6)*+{B^{\prime}}="12";
(12,-6)*+{C^{\prime}}="14";
{\ar^{x} "0";"2"};
{\ar^{y} "2";"4"};
{\ar_{a} "0";"10"};
{\ar^{b} "2";"12"};
{\ar_{x^{\prime}} "10";"12"};
{\ar_{y^{\prime}} "12";"14"};
{\ar@{}|\circlearrowright "0";"12"};
\endxy
\end{equation}
in $\mathscr{C}$, there exists a morphism $(a,c)\colon\delta\to\delta^{\prime}$ which is realized by $(a,b,c)$.
\item[{\rm (ET3)$^{\mathrm{op}}$}] Let $\delta\in\mathbb{E}(C,A)$ and $\delta^{\prime}\in\mathbb{E}(C^{\prime},A^{\prime})$ be any pair of $\mathbb{E}$-extensions, realized by
\[ A\overset{x}{\longrightarrow}B\overset{y}{\longrightarrow}C\ \ \text{and}\ \ A^{\prime}\overset{x^{\prime}}{\longrightarrow}B^{\prime}\overset{y^{\prime}}{\longrightarrow}C^{\prime} \]
respectively.
For any commutative square
\[
\xy
(-12,6)*+{A}="0";
(0,6)*+{B}="2";
(12,6)*+{C}="4";
(-12,-6)*+{A^{\prime}}="10";
(0,-6)*+{B^{\prime}}="12";
(12,-6)*+{C^{\prime}}="14";
{\ar^{x} "0";"2"};
{\ar^{y} "2";"4"};
{\ar_{b} "2";"12"};
{\ar^{c} "4";"14"};
{\ar_{x^{\prime}} "10";"12"};
{\ar_{y^{\prime}} "12";"14"};
{\ar@{}|\circlearrowright "2";"14"};
\endxy
\]
in $\mathscr{C}$, there exists a morphism $(a,c)\colon\delta\to\delta^{\prime}$ which is realized by $(a,b,c)$.
\item[{\rm (ET4)}] Let $(A,\delta,D)$ and $(B,\delta^{\prime},F)$ be $\mathbb{E}$-extensions realized by
\[ A\overset{f}{\longrightarrow}B\overset{f^{\prime}}{\longrightarrow}D\ \ \text{and}\ \ B\overset{g}{\longrightarrow}C\overset{g^{\prime}}{\longrightarrow}F \]
respectively. Then there exist an object $E\in\mathscr{C}$, a commutative diagram
\begin{equation}\label{DiagET4}
\xy
(-21,7)*+{A}="0";
(-7,7)*+{B}="2";
(7,7)*+{D}="4";
(-21,-7)*+{A}="10";
(-7,-7)*+{C}="12";
(7,-7)*+{E}="14";
(-7,-21)*+{F}="22";
(7,-21)*+{F}="24";
{\ar^{f} "0";"2"};
{\ar^{f^{\prime}} "2";"4"};
{\ar@{=} "0";"10"};
{\ar_{g} "2";"12"};
{\ar^{d} "4";"14"};
{\ar_{h} "10";"12"};
{\ar_{h^{\prime}} "12";"14"};
{\ar_{g^{\prime}} "12";"22"};
{\ar^{e} "14";"24"};
{\ar@{=} "22";"24"};
{\ar@{}|\circlearrowright "0";"12"};
{\ar@{}|\circlearrowright "2";"14"};
{\ar@{}|\circlearrowright "12";"24"};
\endxy
\end{equation}
in $\mathscr{C}$, and an $\mathbb{E}$-extension $\delta^{\prime\prime}\in\mathbb{E}(E,A)$ realized by $A\overset{h}{\longrightarrow}C\overset{h^{\prime}}{\longrightarrow}E$, which satisfy the following compatibilities.
\begin{itemize}
\item[{\rm (i)}] $D\overset{d}{\longrightarrow}E\overset{e}{\longrightarrow}F$ realizes $\mathbb{E}(F,f^{\prime})(\delta^{\prime})$,
\item[{\rm (ii)}] $\mathbb{E}(d,A)(\delta^{\prime\prime})=\delta$,

\item[{\rm (iii)}] $\mathbb{E}(E,f)(\delta^{\prime\prime})=\mathbb{E}(e,B)(\delta^{\prime})$. 
\end{itemize}
By {\rm (iii)}, $(f,e)\colon\delta^{\prime\prime}\to\delta^{\prime}$ is a morphism of $\mathbb{E}$-extensions, realized by
\[ (f,\mathrm{id}_C,e)\colon [A\overset{h}{\longrightarrow}C\overset{h^{\prime}}{\longrightarrow}E]\to [B\overset{g}{\longrightarrow}C\overset{g^{\prime}}{\longrightarrow}F]. \]

\item[{\rm (ET4)$^{\mathrm{op}}$}] Dual of {\rm (ET4)} (see Remark~\ref{RemET4op}).
\end{itemize}
\end{definition}

\begin{example}\label{Example1}
Exact categories (with a condition concerning the smallness) and triangulated categories are examples of extriangulated categories. See also Remark~\ref{RemETrExtClosed}.

We briefly show how an exact category can be viewed as an extriangulated category. As for triangulated categories, see the construction in Proposition~\ref{PropTriaExt}.

For the definition and basic properties of an exact category, see \cite{Bu} and \cite{Ke}.
Let $A,C\in\mathscr{C}$ be any pair of objects. Remark that, as shown in \cite[Corollary~3.2]{Bu}, for any morphism of short exact sequences (={\it conflations} in \cite{Ke}) of the form
\[
\xy
(-12,6)*+{A}="0";
(0,6)*+{B}="2";
(12,6)*+{C}="4";
(-12,-6)*+{A}="10";
(0,-6)*+{B^{\prime}}="12";
(12,-6)*+{C}="14";
{\ar^{x} "0";"2"};
{\ar^{y} "2";"4"};
{\ar@{=} "0";"10"};
{\ar^{b} "2";"12"};
{\ar@{=} "4";"14"};
{\ar_{x^{\prime}} "10";"12"};
{\ar_{y^{\prime}} "12";"14"};
{\ar@{}|\circlearrowright "0";"12"};
{\ar@{}|\circlearrowright "2";"14"};
\endxy,
\]
the morphism $b$ in the middle automatically becomes an isomorphism. Consider the same equivalence relation as in Definition~\ref{DefSqEquiv}, and define $\mathrm{Ext}^1(C,A)$ to be the collection of all equivalence classes of short exact sequences of the form $A\overset{x}{\longrightarrow}B\overset{y}{\longrightarrow}C$. We denote the equivalence class by $[A\overset{x}{\longrightarrow}B\overset{y}{\longrightarrow}C]$ as before.

This becomes a small set, for example in the following cases.
\begin{itemize}
\item $\mathscr{C}$ is skeletally small.
\item $\mathscr{C}$ has enough projectives or injectives.
\end{itemize}
In such a case, we obtain a biadditive functor $\mathrm{Ext}^1\colon\mathscr{C}^{\mathrm{op}}\times\mathscr{C}\to\mathit{Ab}$, as stated in \cite[Definitions~5.1]{S}. Its structure is given as follows.
\begin{itemize}
\item[-] For any $\delta=[A\overset{x}{\longrightarrow}B\overset{y}{\longrightarrow}C]\in\mathrm{Ext}^1(C,A)$ and any $a\in\mathscr{C}(A,A^{\prime})$, take a pushout in $\mathscr{C}$, to obtain a morphism of short exact sequences
\[
\xy
(-12,6)*+{A}="0";
(0,6)*+{B}="2";
(12,6)*+{C}="4";
(-6,0)*+{\mathrm{PO}}="6";
(-12,-6)*+{A^{\prime}}="10";
(0,-6)*+{M}="12";
(12,-6)*+{C}="14";
{\ar^{x} "0";"2"};
{\ar^{y} "2";"4"};
{\ar_{a} "0";"10"};
{\ar^{} "2";"12"};
{\ar@{=} "4";"14"};
{\ar_{m} "10";"12"};
{\ar_{e} "12";"14"};
%
{\ar@{}|\circlearrowright "2";"14"};
\endxy.
\]
This gives $\mathrm{Ext}^1(C,a)(\delta)=a_{\ast}\delta=[A^{\prime}\overset{m}{\longrightarrow}M\overset{e}{\longrightarrow}C]$.
\item[-] Dually, for any $c\in\mathscr{C}(C^{\prime},C)$, the map $\mathrm{Ext}^1(c,A)=c^{\ast}\colon \mathrm{Ext}^1(C,A)\to\mathrm{Ext}^1(C^{\prime},A)$ is defined by using pullbacks.
\item[-] The zero element in $\mathrm{Ext}^1(C,A)$ is given by the split short exact sequence \[ 0=[A\overset{\Big[\raise1ex\hbox{\leavevmode\vtop{\baselineskip-8ex \lineskip1ex \ialign{#\crcr{$\scriptstyle{1}$}\crcr{$\scriptstyle{0}$}\crcr}}}\Big]}{\longrightarrow}A\oplus C\overset{[0\ 1]}{\longrightarrow}C]. \]
For any pair $\delta_1=[A\overset{x_1}{\longrightarrow}B_1\overset{y_1}{\longrightarrow}C],\delta_2=[A\overset{x_2}{\longrightarrow}B_2\overset{y_2}{\longrightarrow}C]\in\mathrm{Ext}^1(C,A)$, its sum $\delta_1+\delta_2$ is given by the Baer sum
\[ \Delta_C^{\ast}(\nabla_A)_{\ast}(\delta_1\oplus\delta_2)=\Delta_C^{\ast}(\nabla_A)_{\ast}([A\oplus A\overset{x_1\oplus x_2}{\longrightarrow}B_1\oplus B_2\overset{y_1\oplus y_2}{\longrightarrow}C\oplus C]). \]
\end{itemize}
This shows {\rm (ET1)}. Define the realization $\mathfrak{s}(\delta)$ of $\delta=[A\overset{x}{\longrightarrow}B\overset{y}{\longrightarrow}C]$ to be $\delta$ itself. Then {\rm (ET2)} is trivially satisfied. For {\rm (ET3)} and {\rm (ET4)}, the following fact is useful.
\begin{fact}(\cite[Proposition~3.1]{Bu})\label{FactBu}
For any morphism of short exact sequences
\[
\xy
(-12,6)*+{A}="0";
(0,6)*+{B}="2";
(12,6)*+{C}="4";
(-12,-6)*+{A^{\prime}}="10";
(0,-6)*+{B^{\prime}}="12";
(12,-6)*+{C^{\prime}}="14";
{\ar^{x} "0";"2"};
{\ar^{y} "2";"4"};
{\ar_{a} "0";"10"};
{\ar^{b} "2";"12"};
{\ar^{c} "4";"14"};
{\ar_{x^{\prime}} "10";"12"};
{\ar_{y^{\prime}} "12";"14"};
{\ar@{}|\circlearrowright "0";"12"};
{\ar@{}|\circlearrowright "2";"14"};
\endxy
\]
in $\mathscr{C}$, there exists a commutative diagram
\[
\xy
(-14,12)*+{A}="2";
(0,12)*+{B}="4";
(14,12)*+{C}="6";
(-14,0)*+{A^{\prime}}="12";
(0,0)*+{{}^{\exists}M}="14";
(14,0)*+{C}="16";
(-14,-12)*+{A^{\prime}}="22";
(0,-12)*+{B^{\prime}}="24";
(14,-12)*+{C^{\prime}}="26";
(-7,6)*+{\mathrm{PO}}="32";
(7,-6)*+{\mathrm{PB}}="34";
{\ar^{x} "2";"4"};
{\ar^{y} "4";"6"};
{\ar_{a} "2";"12"};
{\ar "4";"14"};
{\ar@{=} "6";"16"};
{\ar_{m} "12";"14"};
{\ar_{e} "14";"16"};
{\ar@{=} "12";"22"};
{\ar "14";"24"};
{\ar^{c} "16";"26"};
{\ar_{x^{\prime}} "22";"24"};
{\ar_{y^{\prime}} "24";"26"};
{\ar@{}|\circlearrowright "4";"16"};
{\ar@{}|\circlearrowright "12";"24"};
\endxy
\]
whose middle row is also a short exact sequence, the upper-left square is a pushout, and the down-right square is a pullback.
Remark that this means $a_{\ast} [A\overset{x}{\longrightarrow}B\overset{y}{\longrightarrow}C]=c^{\ast} [A^{\prime}\overset{x^{\prime}}{\longrightarrow}B^{\prime}\overset{y^{\prime}}{\longrightarrow}C^{\prime}]$.
\end{fact}
 By Fact~\ref{FactBu}, {\rm (ET3)} follows immediately from the universality of cokernel. Similarly, {\rm (ET4)} follows from \cite[Lemma~3.5]{Bu}. Dually for {\rm (ET3)$^{\mathrm{op}}$} and {\rm (ET4)$^{\mathrm{op}}$}.
\end{example}

\subsection{Terminology in an extriangulated category}
To allow an argument with familiar terms, we introduce terminology from both exact categories and triangulated categories (cf. \cite{Ke,Bu,Ne}).
\begin{definition}\label{DefTermExact1}
Let $(\mathscr{C},\mathbb{E},\mathfrak{s})$ be a triplet satisfying {\rm (ET1)} and {\rm (ET2)}.
\begin{enumerate}
\item[(1)] A sequence $A\overset{x}{\longrightarrow}B\overset{y}{\longrightarrow}C$ is called a {\it conflation} if it realizes some $\mathbb{E}$-extension $\delta\in\mathbb{E}(C,A)$. For the ambiguity of such an $\mathbb{E}$-extension, see Corollary~\ref{CorModif}.
\item[(2)] A morphism $f\in\mathscr{C}(A,B)$ is called an {\it inflation} if it admits some conflation $A\overset{f}{\longrightarrow}B\to C$. For the ambiguity of such a conflation, see Remark~\ref{RemConeCocone}.
\item[(3)] A morphism $f\in\mathscr{C}(A,B)$ is called a {\it deflation} if it admits some conflation $K\to A\overset{f}{\longrightarrow}B$.
\end{enumerate}
\end{definition}

\begin{remark}\label{RemInfInf}
Condition {\rm (ET4)} implies that inflations are closed under composition. Dually, {\rm (ET4)$^{\mathrm{op}}$} implies the composition-closedness of deflations.
\end{remark}

\begin{definition}\label{DefExtClosed}
Let $\mathcal{D}\subseteq\mathscr{C}$ be a full additive subcategory, closed under isomorphisms. We say $\mathcal{D}$ is {\it extension-closed} if it satisfies the following condition.
\begin{itemize}
\item If a conflation $A\to B\to C$ satisfies $A,C\in\mathcal{D}$, then $B\in\mathcal{D}$.
\end{itemize}
\end{definition}

The following can be checked in a straightforward way.
\begin{remark}\label{RemETrExtClosed}
Let $(\mathscr{C},\mathbb{E},\mathfrak{s})$ be an extriangulated category, and let $\mathcal{D}\subseteq\mathscr{C}$ be an extension-closed subcategory. If we define $\mathbb{E}_{\mathcal{D}}$ to be the restriction of $\mathbb{E}$ onto $\mathcal{D}^{\mathrm{op}}\times\mathcal{D}$, and define $\mathfrak{s}_{\mathcal{D}}$ by restricting $\mathfrak{s}$, then $(\mathcal{D},\mathbb{E}_{\mathcal{D}},\mathfrak{s}_{\mathcal{D}})$ becomes an extriangulated category.
\end{remark}

\begin{definition}\label{DefTermExact2}
Let $(\mathscr{C},\mathbb{E},\mathfrak{s})$ be a triplet satisfying {\rm (ET1)} and {\rm (ET2)}.
\begin{enumerate}
\item[(1)] If a conflation $A\overset{x}{\longrightarrow}B\overset{y}{\longrightarrow}C$ realizes $\delta\in\mathbb{E}(C,A)$, we call the pair $(A\overset{x}{\longrightarrow}B\overset{y}{\longrightarrow}C,\delta)$ an {\it $\mathbb{E}$-triangle}, and write it in the following way.
\begin{equation}\label{Etriangle}
A\overset{x}{\longrightarrow}B\overset{y}{\longrightarrow}C\overset{\delta}{\dashrightarrow}
\end{equation}
\item[(2)] Let $A\overset{x}{\longrightarrow}B\overset{y}{\longrightarrow}C\overset{\delta}{\dashrightarrow}$ and $A^{\prime}\overset{x^{\prime}}{\longrightarrow}B^{\prime}\overset{y^{\prime}}{\longrightarrow}C^{\prime}\overset{\delta^{\prime}}{\dashrightarrow}$ be any pair of $\mathbb{E}$-triangles. If a triplet $(a,b,c)$ realizes $(a,c)\colon\delta\to\delta^{\prime}$ as in $(\ref{MorphRealize})$, then we write it as
\[
\xy
(-12,6)*+{A}="0";
(0,6)*+{B}="2";
(12,6)*+{C}="4";
(24,6)*+{}="6";
(-12,-6)*+{A^{\prime}}="10";
(0,-6)*+{B^{\prime}}="12";
(12,-6)*+{C^{\prime}}="14";
(24,-6)*+{}="16";
{\ar^{x} "0";"2"};
{\ar^{y} "2";"4"};
{\ar@{-->}^{\delta} "4";"6"};
{\ar_{a} "0";"10"};
{\ar^{b} "2";"12"};
{\ar^{c} "4";"14"};
{\ar_{x^{\prime}} "10";"12"};
{\ar_{y^{\prime}} "12";"14"};
{\ar@{-->}_{\delta^{\prime}} "14";"16"};
{\ar@{}|\circlearrowright "0";"12"};
{\ar@{}|\circlearrowright "2";"14"};
\endxy
\]
and call $(a,b,c)$ a {\it morphism of $\mathbb{E}$-triangles}.
\end{enumerate}
\end{definition}

\begin{caution}
Although the abbreviated expression $(\ref{Etriangle})$ looks superficially asymmetric, we remark that the definition of an extriangulated category is completely self-dual.
\end{caution}

\begin{remark}
{\rm (ET3)} means that any commutative square $(\ref{SquareForET3})$ bridging $\mathbb{E}$-triangles can be completed into a morphism of $\mathbb{E}$-triangles. Dually for {\rm (ET3)$^{\mathrm{op}}$}.

Condition $(\ast)$ in Definition~\ref{DefRealization} means that any morphism of $\mathbb{E}$-extensions can be realized by a morphism of $\mathbb{E}$-triangles.
\end{remark}

In the above terminology, condition {\rm (ET4)$^{\mathrm{op}}$} can be stated as follows.
\begin{remark}\label{RemET4op}
$($Paraphrase of {\rm (ET4)$^{\mathrm{op}}$}$)$
Let $D\overset{f^{\prime}}{\longrightarrow}A\overset{f}{\longrightarrow}B\overset{\delta}{\dashrightarrow}$ and $F\overset{g^{\prime}}{\longrightarrow}B\overset{g}{\longrightarrow}C\overset{\delta^{\prime}}{\dashrightarrow}$ be $\mathbb{E}$-triangles. Then there exist an $\mathbb{E}$-triangle $E\overset{h^{\prime}}{\longrightarrow}A\overset{h}{\longrightarrow}C\overset{\delta^{\prime\prime}}{\dashrightarrow}$ and a commutative diagram
\[
\xy
(-21,7)*+{D}="0";
(-7,7)*+{E}="2";
(7,7)*+{F}="4";
(-21,-7)*+{D}="10";
(-7,-7)*+{A}="12";
(7,-7)*+{B}="14";
(-7,-21)*+{C}="22";
(7,-21)*+{C}="24";
{\ar^{d} "0";"2"};
{\ar^{e} "2";"4"};
{\ar@{=} "0";"10"};
{\ar_{h^{\prime}} "2";"12"};
{\ar^{g^{\prime}} "4";"14"};
{\ar_{f^{\prime}} "10";"12"};
{\ar_{f} "12";"14"};
{\ar_{h} "12";"22"};
{\ar^{g} "14";"24"};
{\ar@{=} "22";"24"};
{\ar@{}|\circlearrowright "0";"12"};
{\ar@{}|\circlearrowright "2";"14"};
{\ar@{}|\circlearrowright "12";"24"};
\endxy
\]
in $\mathscr{C}$, satisfying the following compatibilities.
\begin{itemize}
\item[{\rm (i)}] $D\overset{d}{\longrightarrow}E\overset{e}{\longrightarrow}F\overset{g^{\prime\ast}\delta}{\dashrightarrow}$ is an $\mathbb{E}$-triangle,
\item[{\rm (ii)}] $\delta^{\prime}=e_{\ast}\delta^{\prime\prime}$,
\item[{\rm (iii)}] $d_{\ast}\delta=g^{\ast}\delta^{\prime\prime}$.
\end{itemize}
\end{remark}

\section{Fundamental properties}\label{section_Fund}

\subsection{Associated exact sequence}

In this section, we will associate exact sequences of natural transformations
\[ \mathscr{C}(C,-)\!\overset{\mathscr{C}(y,-)}{\Longrightarrow}\!\mathscr{C}(B,-)\!\overset{\mathscr{C}(x,-)}{\Longrightarrow}\!\mathscr{C}(A,-)\!\overset{\delta^\sharp}{\Longrightarrow}\!\mathbb{E}(C,-)\!\overset{\mathbb{E}(y,-)}{\Longrightarrow}\!\mathbb{E}(B,-)\!\overset{\mathbb{E}(x,-)}{\Longrightarrow}\!\mathbb{E}(A,-), \]
\[ \mathscr{C}(-,A)\!\overset{\mathscr{C}(-,x)}{\Longrightarrow}\!\mathscr{C}(-,B)\!\overset{\mathscr{C}(-,y)}{\Longrightarrow}\!\mathscr{C}(-,C)\!\overset{\delta_\sharp}{\Longrightarrow}\!\mathbb{E}(-,A)\!\overset{\mathbb{E}(-,x)}{\Longrightarrow}\!\mathbb{E}(-,B)\!\overset{\mathbb{E}(-,y)}{\Longrightarrow}\!\mathbb{E}(-,C) \]to any $\mathbb{E}$-triangle $A\overset{x}{\longrightarrow}B\overset{y}{\longrightarrow}C\overset{\delta}{\dashrightarrow}$ in an extriangulated category $(\mathscr{C},\mathbb{E},\mathfrak{s})$ (Corollary~\ref{ExactToShow}).

Here, $\delta_\sharp$ and $\delta^\sharp$ are defined in the following.
\begin{definition}\label{DefYoneda}
Assume $\mathscr{C}$ and $\mathbb{E}$ satisfy {\rm (ET1)}.
By Yoneda's lemma, any $\mathbb{E}$-extension $\delta\in\mathbb{E}(C,A)$ induces natural transformations
\[ \delta_\sharp\colon\mathscr{C}(-,C)\Rightarrow\mathbb{E}(-,A)\ \ \text{and}\ \ \delta^\sharp\colon\mathscr{C}(A,-)\Rightarrow\mathbb{E}(C,-). \]
For any $X\in\mathscr{C}$, these $(\delta_\sharp)_X$ and $\delta^\sharp_X$ are given as follows.
\begin{enumerate}
\item[(1)] $(\delta_\sharp)_X\colon\mathscr{C}(X,C)\to\mathbb{E}(X,A)\ ;\ f\mapsto f^{\ast}\delta$.
\item[(2)] $\delta^\sharp_X\colon\mathscr{C}(A,X)\to\mathbb{E}(C,X)\ ;\ g\mapsto g_{\ast}\delta$.
\end{enumerate}
We abbreviately denote $(\delta_\sharp)_X(f)$ and $\delta^\sharp_X(g)$ by $\delta_\sharp f$ and $\delta^\sharp g$, when there is no confusion.
\end{definition}

\begin{lemma}\label{LemZero}
Assume $(\mathscr{C},\mathbb{E},\mathfrak{s})$ satisfies {\rm (ET1),(ET2),(ET3),(ET3)$^{\mathrm{op}}$}. Then for any $\mathbb{E}$-triangle $A\overset{x}{\longrightarrow}B\overset{y}{\longrightarrow}C\overset{\delta}{\dashrightarrow}$, the following hold.
\begin{enumerate}
\item[(1)] $y\circ x=0$.
\item[(2)] $x_{\ast}\delta(=\delta^\sharp x)=0$.
\item[(3)] $y^{\ast}\delta(=\delta_\sharp y)=0$.
\end{enumerate}
\end{lemma}
\begin{proof}
{\rm (1)} By {\rm (ET2)}, the conflation $A\overset{\mathrm{id}}{\longrightarrow}A\to 0$ realizes $0\in\mathbb{E}(0,A)$. Applying {\rm (ET3)} to 
\[
\xy
(-12,6)*+{A}="0";
(0,6)*+{A}="2";
(12,6)*+{0}="4";
(-12,-6)*+{A}="10";
(0,-6)*+{B}="12";
(12,-6)*+{C}="14";
{\ar^{\mathrm{id}_A} "0";"2"};
{\ar^{} "2";"4"};
{\ar_{\mathrm{id}_A} "0";"10"};
{\ar^{x} "2";"12"};
{\ar_{x} "10";"12"};
{\ar_{y} "12";"14"};
{\ar@{}|\circlearrowright "0";"12"};
\endxy,
\]
we obtain a morphism of $\mathbb{E}$-triangles $(\mathrm{id}_A,x,0)$. Especially we have $y\circ x=0$.

{\rm (2)} Similarly, applying {\rm (ET3)} to
\[
\xy
(-12,6)*+{A}="0";
(0,6)*+{B}="2";
(12,6)*+{C}="4";
(-12,-6)*+{B}="10";
(0,-6)*+{B}="12";
(12,-6)*+{0}="14";
{\ar^{x} "0";"2"};
{\ar^{y} "2";"4"};
{\ar_{x} "0";"10"};
{\ar^{\mathrm{id}_B} "2";"12"};
{\ar_{\mathrm{id}_B} "10";"12"};
{\ar "12";"14"};
{\ar@{}|\circlearrowright "0";"12"};
\endxy,
\]
we obtain a morphism of $\mathbb{E}$-extensions $(x,0)\colon\delta\to0$. Especially we have $x_{\ast}\delta=0$.
{\rm (3)} is dual to {\rm (2)}.
\end{proof}

\begin{proposition}\label{PropExact1}
Assume $(\mathscr{C},\mathbb{E},\mathfrak{s})$ satisfies {\rm (ET1),(ET2)}. Then the following are equivalent.
\begin{enumerate}
\item[(1)] $(\mathscr{C},\mathbb{E},\mathfrak{s})$ satisfies {\rm (ET3)} and {\rm (ET3)$^{\mathrm{op}}$}.
\item[(2)] For any $\mathbb{E}$-triangle $A\overset{x}{\longrightarrow}B\overset{y}{\longrightarrow}C\overset{\delta}{\dashrightarrow}$, the following sequences of natural transformations are exact.
\begin{itemize}
\item[{\rm (i)}] $\mathscr{C}(C,-)\overset{\mathscr{C}(y,-)}{\Longrightarrow}\mathscr{C}(B,-)\overset{\mathscr{C}(x,-)}{\Longrightarrow}\mathscr{C}(A,-)\overset{\delta^\sharp}{\Longrightarrow}\mathbb{E}(C,-)\overset{\mathbb{E}(y,-)}{\Longrightarrow}\mathbb{E}(B,-)$\\ in $\mathrm{Mod}(\mathscr{C})$. Here $\mathrm{Mod}(\mathscr{C})$ denotes the abelian category of additive functors from $\mathscr{C}$ to $\mathit{Ab}$.
\item[{\rm (ii)}] $\mathscr{C}(-,A)\overset{\mathscr{C}(-,x)}{\Longrightarrow}\mathscr{C}(-,B)\overset{\mathscr{C}(-,y)}{\Longrightarrow}\mathscr{C}(-,C)\overset{\delta_\sharp}{\Longrightarrow}\mathbb{E}(-,A)\overset{\mathbb{E}(-,x)}{\Longrightarrow}\mathbb{E}(-,B)$\\ in $\mathrm{Mod}(\mathscr{C}^{\mathrm{op}})$.
\end{itemize}
\end{enumerate}
\end{proposition}

\begin{remark}
In the above {\rm (i)}, the category $\mathrm{Mod}(\mathscr{C})$ is not locally small in general. The ``exactness" of the sequence in {\rm (i)} simply means that
\[ \mathscr{C}(C,X)\overset{\mathscr{C}(y,X)}{\longrightarrow}\mathscr{C}(B,X)\overset{\mathscr{C}(x,X)}{\longrightarrow}\mathscr{C}(A,X)\overset{\delta^\sharp_X}{\longrightarrow}\mathbb{E}(C,X)\overset{\mathbb{E}(y,X)}{\longrightarrow}\mathbb{E}(B,X) \]
is exact in $\mathit{Ab}$ for any $X\in\mathscr{C}$. Similarly for {\rm (ii)}.
\end{remark}

\begin{proof}[Proof of Proposition~\ref{PropExact1}]
First we assume {\rm (1)}. We only show the exactness of {\rm (i)}, since {\rm (ii)} can be shown dually. By Lemma~\ref{LemZero}, composition of any consecutive morphisms in {\rm (i)} is equal to $0$. Let us show the exactness of
\[ \mathscr{C}(C,X)\overset{\mathscr{C}(y,X)}{\longrightarrow}\mathscr{C}(B,X)\overset{\mathscr{C}(x,X)}{\longrightarrow}\mathscr{C}(A,X)\overset{\delta^\sharp_X}{\longrightarrow}\mathbb{E}(C,X)\overset{\mathbb{E}(y,X)}{\longrightarrow}\mathbb{E}(B,X) \]
for any $X\in\mathscr{C}$.


\noindent\underline{Exactness at $\mathscr{C}(B,X)$}

Let $b\in\mathscr{C}(B,X)$ be any morphism satisfying $\mathscr{C}(x,X)(b)=b\circ x=0$. Applying {\rm (ET3)} to
\[
\xy
(-12,6)*+{A}="0";
(0,6)*+{B}="2";
(12,6)*+{C}="4";
(-12,-6)*+{0}="10";
(0,-6)*+{X}="12";
(12,-6)*+{X}="14";
{\ar^{x} "0";"2"};
{\ar^{y} "2";"4"};
{\ar_{0} "0";"10"};
{\ar^{b} "2";"12"};
{\ar_{} "10";"12"};
{\ar_{\mathrm{id}_X} "12";"14"};
{\ar@{}|\circlearrowright "0";"12"};
\endxy,
\]
we obtain a morphism $c\in\mathscr{C}(C,X)$ satisfying $b=c\circ y=\mathscr{C}(y,X)(c)$.


\noindent\underline{Exactness at $\mathscr{C}(A,X)$}

Let $a\in\mathscr{C}(A,X)$ be any morphism satisfying $\delta^\sharp_X(a)=a_{\ast}\delta=0$. This means that $(a,0)\colon\delta\to0$ is a morphism of $\mathbb{E}$-extensions.
Since $\mathfrak{s}$ realizes $\mathbb{E}$, there exists $b\in\mathscr{C}(B,X)$ which gives the following morphism of $\mathbb{E}$-triangles.
\[
\xy
(-12,6)*+{A}="0";
(0,6)*+{B}="2";
(12,6)*+{C}="4";
(24,6)*+{}="6";
(-12,-6)*+{X}="10";
(0,-6)*+{X}="12";
(12,-6)*+{0}="14";
(24,-6)*+{}="16";
{\ar^{x} "0";"2"};
{\ar^{y} "2";"4"};
{\ar@{-->}^{\delta} "4";"6"};
{\ar_{a} "0";"10"};
{\ar^{b} "2";"12"};
{\ar "4";"14"};
{\ar_{\mathrm{id}_X} "10";"12"};
{\ar "12";"14"};
{\ar@{-->}_{0} "14";"16"};
{\ar@{}|\circlearrowright "0";"12"};
{\ar@{}|\circlearrowright "2";"14"};
\endxy
\]
Especially we have $a=b\circ x=\mathscr{C}(x,X)(b)$.


\noindent\underline{Exactness at $\mathbb{E}(C,X)$}

Let $\theta\in\mathbb{E}(C,X)$ be any $\mathbb{E}$-extension satisfying $\mathbb{E}(y,X)(\theta)=y^{\ast}\theta=0$. Realize them as $\mathbb{E}$-triangles
\[ X\overset{f}{\longrightarrow}Y\overset{g}{\longrightarrow}C\overset{\theta}{\dashrightarrow}\ \ \text{and}\ \ X\overset{m}{\longrightarrow}Z\overset{e}{\longrightarrow}B\overset{y^{\ast}\theta}{\dashrightarrow}. \]
Then the morphism $(\mathrm{id}_X,y)\colon y^{\ast}\theta\to\theta$ can be realized by
\[
\xy
(-12,6)*+{X}="0";
(0,6)*+{Z}="2";
(12,6)*+{B}="4";
(24,6)*+{}="6";
(-12,-6)*+{X}="10";
(0,-6)*+{Y}="12";
(12,-6)*+{C}="14";
(24,-6)*+{}="16";
{\ar^{m} "0";"2"};
{\ar^{e} "2";"4"};
{\ar@{-->}^{y^{\ast}\theta} "4";"6"};
{\ar@{=} "0";"10"};
{\ar^{e^{\prime}} "2";"12"};
{\ar^{y} "4";"14"};
{\ar_{f} "10";"12"};
{\ar_{g} "12";"14"};
{\ar@{-->}_{\theta} "14";"16"};
{\ar@{}|\circlearrowright "0";"12"};
{\ar@{}|\circlearrowright "2";"14"};
\endxy
\]
with some $e^{\prime}\in\mathscr{C}(Z,Y)$. Since $y^{\ast}\theta$ splits by assumption, $e$ has a section $s$. Applying {\rm (ET3)$^{\mathrm{op}}$} to
\[
\xy
(-12,6)*+{A}="0";
(0,6)*+{B}="2";
(12,6)*+{C}="4";
(24,6)*+{}="6";
(-12,-6)*+{X}="10";
(0,-6)*+{Y}="12";
(12,-6)*+{C}="14";
(24,-6)*+{}="16";
{\ar^{x} "0";"2"};
{\ar^{y} "2";"4"};
{\ar@{-->}^{\delta} "4";"6"};
{\ar_{e^{\prime}\circ s} "2";"12"};
{\ar@{=} "4";"14"};
{\ar_{f} "10";"12"};
{\ar_{g} "12";"14"};
{\ar@{-->}_{\theta} "14";"16"};
{\ar@{}|\circlearrowright "2";"14"};
\endxy,
\]
we obtain $a\in\mathscr{C}(A,X)$ which gives a morphism $(a,\mathrm{id}_C)\colon\delta\to\theta$. This means $\theta=a_{\ast}\delta=\delta^\sharp a$.


Conversely, let us assume {\rm (2)} and show {\rm (ET3)}. Let $A\overset{x}{\longrightarrow}B\overset{y}{\longrightarrow}C\overset{\delta}{\dashrightarrow}$ and $A^{\prime}\overset{x^{\prime}}{\longrightarrow}B^{\prime}\overset{y^{\prime}}{\longrightarrow}C^{\prime}\overset{\delta^{\prime}}{\dashrightarrow}$ be any pair of $\mathbb{E}$-triangles. Suppose that we are given a commutative diagram
\[
\xy
(-12,6)*+{A}="0";
(0,6)*+{B}="2";
(12,6)*+{C}="4";
(-12,-6)*+{A^{\prime}}="10";
(0,-6)*+{B^{\prime}}="12";
(12,-6)*+{C^{\prime}}="14";
{\ar^{x} "0";"2"};
{\ar^{y} "2";"4"};
{\ar_{a} "0";"10"};
{\ar^{b} "2";"12"};
{\ar_{x^{\prime}} "10";"12"};
{\ar_{y^{\prime}} "12";"14"};
{\ar@{}|\circlearrowright "0";"12"};
\endxy
\]
in $\mathscr{C}$. Remark that $\mathbb{E}(-,x)\circ\delta_\sharp=0$ is equivalent to $x_{\ast}\delta=0$ by Yoneda's lemma. Similarly, we have $y^{\prime\ast}\delta^{\prime}=0$.

By the exactness of
\[ \mathscr{C}(C,C^{\prime})\overset{(\delta^{\prime}_\sharp)_C}{\longrightarrow}\mathbb{E}(C,A^{\prime})\overset{\mathbb{E}(C,x^{\prime})}{\longrightarrow}\mathbb{E}(C,B^{\prime}) \]
and the equality
\[ \mathbb{E}(C,x^{\prime})(a_{\ast}\delta)=x^{\prime}_{\ast} a_{\ast}\delta=b_{\ast} x_{\ast}\delta=0, \]
there is $c^{\prime}\in\mathscr{C}(C,C^{\prime})$ satisfying $a_{\ast}\delta=\delta^{\prime}_\sharp c^{\prime}=c^{\prime\ast}\delta^{\prime}$.
Thus $(a,c)\colon\delta\to\delta^{\prime}$ is a morphism of $\mathbb{E}$-extensions. Take its realization as follows.
\[
\xy
(-12,6)*+{A}="0";
(0,6)*+{B}="2";
(12,6)*+{C}="4";
(24,6)*+{}="6";
(-12,-6)*+{A^{\prime}}="10";
(0,-6)*+{B^{\prime}}="12";
(12,-6)*+{C^{\prime}}="14";
(24,-6)*+{}="16";
{\ar^{x} "0";"2"};
{\ar^{y} "2";"4"};
{\ar@{-->}^{\delta} "4";"6"};
{\ar_{a} "0";"10"};
{\ar^{b^{\prime}} "2";"12"};
{\ar^{c^{\prime}} "4";"14"};
{\ar@{-->}_{\delta^{\prime}} "14";"16"};
{\ar_{x^{\prime}} "10";"12"};
{\ar_{y^{\prime}} "12";"14"};
{\ar@{}|\circlearrowright "0";"12"};
{\ar@{}|\circlearrowright "2";"14"};
\endxy
\]
Then by the exactness of
\[ \mathscr{C}(C,B^{\prime})\overset{\mathscr{C}(y,B^{\prime})}{\longrightarrow}\mathscr{C}(B,B^{\prime})\overset{\mathscr{C}(x,B^{\prime})}{\longrightarrow}\mathscr{C}(A,B^{\prime}) \]
and the equality
\[ (b-b^{\prime})\circ x=x^{\prime}\circ a-x^{\prime}\circ a=0, \]
there exists $c^{\prime\prime}\in\mathscr{C}(C,B^{\prime})$ satisfying $c^{\prime\prime}\circ y=b-b^{\prime}$. If we put $c=c^{\prime}+y^{\prime}\circ c^{\prime\prime}$, this satisfies
\begin{eqnarray*}
&c\circ y=c^{\prime}\circ y+y^{\prime}\circ c^{\prime\prime}\circ y=y^{\prime}\circ b^{\prime}+(y^{\prime}\circ b-y^{\prime}\circ b^{\prime})=y^{\prime}\circ b,&\\
&c^{\ast}\delta^{\prime}=c^{\prime\ast}\delta^{\prime}+c^{\prime\prime\ast}y^{\prime\ast}\delta^{\prime}=a_{\ast}\delta.&
\end{eqnarray*}
Dually, {\rm (2)} implies {\rm (ET3)$^{\mathrm{op}}$}.
\end{proof}

\begin{corollary}\label{CorExact0}
Assume $(\mathscr{C},\mathbb{E},\mathfrak{s})$ satisfies {\rm (ET1),(ET2),(ET3),(ET3)$^{\mathrm{op}}$}. Let
\begin{equation}\label{Diag_CorEx}
\xy
(-12,6)*+{A}="0";
(0,6)*+{B}="2";
(12,6)*+{C}="4";
(24,6)*+{}="6";
(-12,-6)*+{A^{\prime}}="10";
(0,-6)*+{B^{\prime}}="12";
(12,-6)*+{C^{\prime}}="14";
(24,-6)*+{}="16";
{\ar^{x} "0";"2"};
{\ar^{y} "2";"4"};
{\ar@{-->}^{\delta} "4";"6"};
{\ar_{a} "0";"10"};
{\ar^{b} "2";"12"};
{\ar^{c} "4";"14"};
{\ar_{x^{\prime}} "10";"12"};
{\ar_{y^{\prime}} "12";"14"};
{\ar@{-->}_{\delta^{\prime}} "14";"16"};
{\ar@{}|\circlearrowright "0";"12"};
{\ar@{}|\circlearrowright "2";"14"};
\endxy
\end{equation}
be any morphism of $\mathbb{E}$-triangles. Then the following are equivalent.
\begin{enumerate}
\item[(1)] $a$ factors through $x$.
\item[(2)] $a_{\ast}\delta=c^{\ast}\delta^{\prime}=0$.
\item[(3)] $c$ factors through $y^{\prime}$.
\end{enumerate}
In particular, in the case $\delta=\delta^{\prime}$ and $(a,b,c)=(\mathrm{id},\mathrm{id},\mathrm{id})$, we obtain
\[ x\ \text{has a retraction}\ \Leftrightarrow\ \delta\ \text{splits}\ \Leftrightarrow\ y\ \text{has a section}. \]
\end{corollary}
\begin{proof}
By the definition of $\delta^\sharp$, it satisfies $\delta^\sharp_{A^{\prime}}(a)=a_{\ast}\delta$. Thus $(1)\Leftrightarrow(2)$ follows from the exactness of
\[ \mathscr{C}(B,A^{\prime})\overset{\mathscr{C}(x,A^{\prime})}{\longrightarrow}\mathscr{C}(A,A^{\prime})\overset{\delta^\sharp_{A^{\prime}}}{\longrightarrow}\mathbb{E}(C,A^{\prime}). \]
$(2)\Leftrightarrow(3)$ can be shown dually.
\end{proof}

\begin{corollary}\label{CorExact1}
Assume $(\mathscr{C},\mathbb{E},\mathfrak{s})$ satisfies {\rm (ET1),(ET2),(ET3),(ET3)$^{\mathrm{op}}$}. Then the following holds for any morphism of $\mathbb{E}$-triangles.
\begin{enumerate}
\item[(1)] If $a$ and $c$ are isomorphisms in $\mathscr{C}$ $($equivalently, if $(a,c)$ is an isomorphism in $\mathbb{E}\text{-}\mathrm{Ext}(\mathscr{C})$ in Definition~\ref{DefMorphExt}$)$, then so is $b$.
\item[(2)] If $a$ and $b$ are isomorphisms in $\mathscr{C}$, then so is $c$.
\item[(3)] If $b$ and $c$ are isomorphisms in $\mathscr{C}$, then so is $a$.
\end{enumerate}
\end{corollary}
\begin{proof}
{\rm (1)} By Proposition~\ref{PropExact1}, we have morphisms of exact sequences of abelian groups
\begin{equation}\label{Diag_CorExact1A}
\xy
(-40,7)*+{\mathscr{C}(B^{\prime},A)}="0";
(-14,7)*+{\mathscr{C}(B^{\prime},B)}="2";
(12,7)*+{\mathscr{C}(B^{\prime},C)}="4";
(36,7)*+{\mathbb{E}(B^{\prime},A)}="6";
(-40,-7)*+{\mathscr{C}(B^{\prime},A^{\prime})}="10";
(-14,-7)*+{\mathscr{C}(B^{\prime},B^{\prime})}="12";
(12,-7)*+{\mathscr{C}(B^{\prime},C^{\prime})}="14";
(36,-7)*+{\mathbb{E}(B^{\prime},A^{\prime})}="16";
{\ar^{\mathscr{C}(B^{\prime},x)} "0";"2"};
{\ar^{\mathscr{C}(B^{\prime},y)} "2";"4"};
{\ar^{(\delta_\sharp)_{B^{\prime}}} "4";"6"};
{\ar^{\cong}_{\mathscr{C}(B^{\prime},a)} "0";"10"};
{\ar^{\mathscr{C}(B^{\prime},b)} "2";"12"};
{\ar_{\cong}^{\mathscr{C}(B^{\prime},c)} "4";"14"};
{\ar_{\cong}^{\mathbb{E}(B^{\prime},a)} "6";"16"};
{\ar_{\mathscr{C}(B^{\prime},x^{\prime})} "10";"12"};
{\ar_{\mathscr{C}(B^{\prime},y^{\prime})} "12";"14"};
{\ar_{(\delta^{\prime}_{\sharp})_{B^{\prime}}} "14";"16"};
{\ar@{}|\circlearrowright "0";"12"};
{\ar@{}|\circlearrowright "2";"14"};
{\ar@{}|\circlearrowright "4";"16"};
\endxy
\end{equation}
and
\begin{equation}\label{Diag_CorExact1B}
\xy
(-40,7)*+{\mathscr{C}(C^{\prime},B)}="0";
(-14,7)*+{\mathscr{C}(B^{\prime},B)}="2";
(12,7)*+{\mathscr{C}(A^{\prime},B)}="4";
(36,7)*+{\mathbb{E}(C^{\prime},B)}="6";
(-40,-7)*+{\mathscr{C}(C,B)}="10";
(-14,-7)*+{\mathscr{C}(B,B)}="12";
(12,-7)*+{\mathscr{C}(A,B)}="14";
(36,-7)*+{\mathbb{E}(C,B)}="16";
{\ar^{\mathscr{C}(y^{\prime},B)} "0";"2"};
{\ar^{\mathscr{C}(x^{\prime},B)} "2";"4"};
{\ar^{\delta^{\prime\sharp}_B} "4";"6"};
{\ar^{\cong}_{\mathscr{C}(c,B)} "0";"10"};
{\ar^{\mathscr{C}(b,B)} "2";"12"};
{\ar_{\cong}^{\mathscr{C}(a,B)} "4";"14"};
{\ar_{\cong}^{\mathbb{E}(c,B)} "6";"16"};
{\ar_{\mathscr{C}(y,B)} "10";"12"};
{\ar_{\mathscr{C}(x,B)} "12";"14"};
{\ar_{\delta^\sharp_B} "14";"16"};
{\ar@{}|\circlearrowright "0";"12"};
{\ar@{}|\circlearrowright "2";"14"};
{\ar@{}|\circlearrowright "4";"16"};
\endxy.
\end{equation}
Diagram $(\ref{Diag_CorExact1A})$ shows the surjectivity of $\mathscr{C}(B^{\prime},b)$. Thus $b$ has a right inverse in $\mathscr{C}$. Dually, $(\ref{Diag_CorExact1B})$ shows the surjectivity of $\mathscr{C}(b,B)$, and thus $b$ also has a left inverse.


{\rm (2)} By Yoneda's lemma, this follows from the five lemma applied to
\[
\xy
(-44,7)*+{\mathscr{C}(-,A)}="0";
(-22,7)*+{\mathscr{C}(-,B)}="2";
(0,7)*+{\mathscr{C}(-,C)}="4";
(22,7)*+{\mathbb{E}(-,A)}="6";
(44,7)*+{\mathbb{E}(-,B)}="8";
(-44,-7)*+{\mathscr{C}(-,A^{\prime})}="10";
(-22,-7)*+{\mathscr{C}(-,B^{\prime})}="12";
(0,-7)*+{\mathscr{C}(-,C^{\prime})}="14";
(22,-7)*+{\mathbb{E}(-,A^{\prime})}="16";
(44,-7)*+{\mathbb{E}(-,B^{\prime})}="18";
{\ar@{=>} "0";"2"};
{\ar@{=>} "2";"4"};
{\ar@{=>} "4";"6"};
{\ar@{=>} "6";"8"};
{\ar@{=>}_{\cong} "0";"10"};
{\ar@{=>}_{\cong} "2";"12"};
{\ar@{=>}^{\mathscr{C}(-,c)} "4";"14"};
{\ar@{=>}_{\cong} "6";"16"};
{\ar@{=>}^{\cong} "8";"18"};
{\ar@{=>} "10";"12"};
{\ar@{=>} "12";"14"};
{\ar@{=>} "14";"16"};
{\ar@{=>} "16";"18"};
{\ar@{}|\circlearrowright "0";"12"};
{\ar@{}|\circlearrowright "2";"14"};
{\ar@{}|\circlearrowright "4";"16"};
{\ar@{}|\circlearrowright "6";"18"};
\endxy.
\]
{\rm (3)} is dual to {\rm (2)}.
\end{proof}

\begin{proposition}\label{PropModif}
Assume $(\mathscr{C},\mathbb{E},\mathfrak{s})$ satisfies {\rm (ET1),(ET2),(ET3),(ET3)$^{\mathrm{op}}$}. Let $A\overset{x}{\longrightarrow}B\overset{y}{\longrightarrow}C\overset{\delta}{\dashrightarrow}$ be any $\mathbb{E}$-triangle. If $a\in\mathscr{C}(A,A^{\prime})$ and $c\in\mathscr{C}(C^{\prime},C)$ are isomorphisms, then
\[ A^{\prime}\overset{x\circ a^{-1}}{\longrightarrow}B\overset{c^{-1}\circ y}{\longrightarrow}C^{\prime}\overset{a_{\ast} c^{\ast}\delta}{\dashrightarrow} \]
becomes again an $\mathbb{E}$-triangle.
\end{proposition}
\begin{proof}
Put $\mathfrak{s}(a_{\ast} c^{\ast}\delta)=[A^{\prime}\overset{x^{\prime}}{\longrightarrow}B^{\prime}\overset{y^{\prime}}{\longrightarrow}C^{\prime}]$.
Since
\[ (c^{-1})^{\ast}\big(a_{\ast} c^{\ast}\delta\big)=(c\circ c^{-1})^{\ast} a_{\ast}\delta=a_{\ast}\delta, \]
we see that $(a,c^{-1})\colon\delta\to c^{\ast} a_{\ast} \delta$ is a morphism of $\mathbb{E}$-extensions.
Take a morphism of $\mathbb{E}$-triangles $(a,b,c^{-1})$ realizing $(a,c^{-1})$. Then $b$ is an isomorphism by Corollary~\ref{CorExact1}. Since
\[
\xy
(-16,0)*+{A^{\prime}}="0";
(3,0)*+{}="1";
(0,8)*+{B}="2";
(0,-8)*+{B^{\prime}}="4";
(-3,0)*+{}="5";
(16,0)*+{C^{\prime}}="6";
{\ar^{x\circ a^{-1}} "0";"2"};
{\ar^{c^{-1}\circ y} "2";"6"};
{\ar_{x^{\prime}} "0";"4"};
{\ar_{y^{\prime}} "4";"6"};
{\ar^{b}_{\cong} "2";"4"};
{\ar@{}|\circlearrowright "0";"1"};
{\ar@{}|\circlearrowright "5";"6"};
\endxy
\]
is commutative, it follows that $[A^{\prime}\overset{x^{\prime}}{\longrightarrow}B^{\prime}\overset{y^{\prime}}{\longrightarrow}C^{\prime}]=[A^{\prime}\overset{x\circ a^{-1}}{\longrightarrow}B\overset{c^{-1}\circ y}{\longrightarrow}C^{\prime}]$.
\end{proof}

\begin{corollary}\label{CorModif}
Assume $(\mathscr{C},\mathbb{E},\mathfrak{s})$ satisfies {\rm (ET1),(ET2),(ET3),(ET3)$^{\mathrm{op}}$}. Let $A\overset{x}{\longrightarrow}B\overset{y}{\longrightarrow}C\overset{\delta}{\dashrightarrow}$ be any $\mathbb{E}$-triangle. Then for any $\delta^{\prime}\in\mathbb{E}(C,A)$, the following are equivalent.
\begin{enumerate}
\item[(1)] $\mathfrak{s}(\delta)=\mathfrak{s}(\delta^{\prime})$.
\item[(2)] $\delta^{\prime}=a_{\ast}\delta$ for some automorphism $a\in\mathscr{C}(A,A)$ satisfying $x\circ a=x$.
\item[(3)] $\delta^{\prime}=c^{\ast}\delta$ for some automorphism $c\in\mathscr{C}(C,C)$ satisfying $c\circ y=y$.
\item[(4)] $\delta^{\prime}=a_{\ast} c^{\ast}\delta$ for some pair of automorphisms $a\in\mathscr{C}(A,A),\, c\in\mathscr{C}(C,C)$ satisfying $x\circ a=x$ and $c\circ y=y$.
\end{enumerate}
\end{corollary}
\begin{proof}
$(2)\Rightarrow(4)$ is trivial. Similarly for $(3)\Rightarrow(4)$. Proposition~\ref{PropModif} shows $(4)\Rightarrow(1)$.

Let us show $(1)\Rightarrow(2)$. Suppose $\delta^{\prime}$ satisfies $\mathfrak{s}(\delta^{\prime})=\mathfrak{s}(\delta)=[A\overset{x}{\longrightarrow}B\overset{y}{\longrightarrow}C]$. Applying {\rm (ET3)$^{\mathrm{op}}$} to
\[
\xy
(-12,6)*+{A}="0";
(0,6)*+{B}="2";
(12,6)*+{C}="4";
(24,6)*+{}="6";
(-12,-6)*+{A}="10";
(0,-6)*+{B}="12";
(12,-6)*+{C}="14";
(24,-6)*+{}="16";
{\ar^{x} "0";"2"};
{\ar^{y} "2";"4"};
{\ar@{-->}^{\delta} "4";"6"};
{\ar@{=} "2";"12"};
{\ar@{=} "4";"14"};
{\ar_{x} "10";"12"};
{\ar_{y} "12";"14"};
{\ar@{-->}_{\delta^{\prime}} "14";"16"};
{\ar@{}|\circlearrowright "2";"14"};
\endxy,
\]
we obtain $a\in\mathscr{C}(A,A)$ with which $(a,\mathrm{id}_B,\mathrm{id}_C)$ gives a morphism between the above $\mathbb{E}$-triangles. By Corollary~\ref{CorExact1}, this $a$ is an isomorphism.
$(1)\Rightarrow(3)$ can be shown in a similar way.
\end{proof}

For simplicity, we use the following notations.
\begin{definition}\label{DefConeCocone}
Assume $(\mathscr{C},\mathbb{E},\mathfrak{s})$ satisfies {\rm (ET1),(ET2),(ET3),(ET3)$^{\mathrm{op}}$}.
\begin{enumerate}
\item[(1)] For an inflation $f\in\mathscr{C}(A,B)$, take a conflation $A\overset{f}{\longrightarrow}B\to C$, and denote this $C$ by $\mathrm{Cone}(f)$.
\item[(2)] For a deflation $f\in\mathscr{C}(A,B)$, take a conflation $K\to A\overset{f}{\longrightarrow}B$. We denote this $K$ by $\mathrm{CoCone}(f)$.
\end{enumerate}
\end{definition}

$\mathrm{Cone}(f)$ is determined uniquely up to isomorphism by the following remark. They are not functorial in general, as the case of triangulated category suggests. Dually for $\mathrm{CoCone}(f)$.
\begin{remark}\label{RemConeCocone}
Let $f\in\mathscr{C}(A,B)$ be an inflation, and suppose
\[ A\overset{f}{\longrightarrow}B\overset{g}{\longrightarrow}C\overset{\delta}{\dashrightarrow},\quad A\overset{f}{\longrightarrow}B\overset{g^{\prime}}{\longrightarrow}C^{\prime}\overset{\delta^{\prime}}{\dashrightarrow} \]
are $\mathbb{E}$-triangles. Then by {\rm (ET3)} applied to
\[
\xy
(-12,6)*+{A}="0";
(0,6)*+{B}="2";
(12,6)*+{C}="4";
(24,6)*+{}="6";
(-12,-6)*+{A}="10";
(0,-6)*+{B}="12";
(12,-6)*+{C^{\prime}}="14";
(24,-6)*+{}="16";
{\ar^{f} "0";"2"};
{\ar^{g} "2";"4"};
{\ar@{-->}^{\delta} "4";"6"};
{\ar@{=} "0";"10"};
{\ar@{=} "2";"12"};
{\ar_{f} "10";"12"};
{\ar_{g^{\prime}} "12";"14"};
{\ar@{-->}_{\delta^{\prime}} "14";"16"};
{\ar@{}|\circlearrowright "0";"12"};
\endxy,
\]
there exists $c\in\mathscr{C}(C,C^{\prime})$ which gives a morphism of $\mathbb{E}$-triangles $(\mathrm{id}_A,\mathrm{id}_B,c)$. By Corollary~\ref{CorExact1}, this $c$ is an isomorphism.
\end{remark}

\begin{proposition}\label{PropExact2}
Assume $(\mathscr{C},\mathbb{E},\mathfrak{s})$ satisfies {\rm (ET1),(ET2),(ET3),(ET3)$^\mathrm{op}$}.
Let $A\overset{x}{\longrightarrow}B\overset{y}{\longrightarrow}C\overset{\delta}{\dashrightarrow}$ be any $\mathbb{E}$-triangle. Then, we have the following.
\begin{enumerate}
\item[(1)] If $(\mathscr{C},\mathbb{E},\mathfrak{s})$ satisfies {\rm (ET4)}, then
\[ \mathbb{E}(-,A)\overset{\mathbb{E}(-,x)}{\Longrightarrow}\mathbb{E}(-,B)\overset{\mathbb{E}(-,y)}{\Longrightarrow}\mathbb{E}(-,C) \]
is exact.
\item[(2)] If $(\mathscr{C},\mathbb{E},\mathfrak{s})$ satisfies {\rm (ET4)$^\mathrm{op}$}, then
\[ \mathbb{E}(C,-)\overset{\mathbb{E}(y,-)}{\Longrightarrow}\mathbb{E}(B,-)\overset{\mathbb{E}(x,-)}{\Longrightarrow}\mathbb{E}(A,-) \]
is exact.
\end{enumerate}
\end{proposition}
\begin{proof}
{\rm (1)} $\mathbb{E}(-,y)\circ\mathbb{E}(-,x)=0$ follows from Lemma~\ref{LemZero}. Let $X\in\mathscr{C}$ be any object. Let $\theta\in\mathbb{E}(X,B)$ be any $\mathbb{E}$-extension, realized by an $\mathbb{E}$-triangle $B\overset{f}{\longrightarrow}Y\overset{g}{\longrightarrow}X\overset{\theta}{\dashrightarrow}$. By {\rm (ET4)}, there exist $E\in\mathscr{C},\theta^{\prime}\in\mathbb{E}(E,A)$ and a commutative diagram
\[
\xy
(-21,7)*+{A}="0";
(-7,7)*+{B}="2";
(7,7)*+{C}="4";
(-21,-7)*+{A}="10";
(-7,-7)*+{Y}="12";
(7,-7)*+{E}="14";
(-7,-21)*+{X}="22";
(7,-21)*+{X}="24";
{\ar^{x} "0";"2"};
{\ar^{y} "2";"4"};
{\ar@{=} "0";"10"};
{\ar_{f} "2";"12"};
{\ar^{d} "4";"14"};
{\ar_{h} "10";"12"};
{\ar_{h^{\prime}} "12";"14"};
{\ar_{g} "12";"22"};
{\ar^{e} "14";"24"};
{\ar@{=} "22";"24"};
{\ar@{}|\circlearrowright "0";"12"};
{\ar@{}|\circlearrowright "2";"14"};
{\ar@{}|\circlearrowright "12";"24"};
\endxy
\]
which satisfy
\begin{eqnarray*}
&\mathfrak{s}(y_{\ast}\theta)=[C\overset{d}{\longrightarrow}E\overset{e}{\longrightarrow}X],&\\
&\mathfrak{s}(\theta^{\prime})=[A\overset{h}{\longrightarrow}Y\overset{h^{\prime}}{\longrightarrow}E],&\\
&x_{\ast}\theta^{\prime}=e^{\ast}\theta.&
\end{eqnarray*}
Thus if $\mathbb{E}(X,y)(\theta)=y_{\ast}\theta=0$, then $e$ has a section $s\in\mathscr{C}(X,E)$. If we put $\rho=s^{\ast}\theta^{\prime}$, then this satisfies
\[ \mathbb{E}(x,X)(\rho)=x_{\ast} s^{\ast}\theta^{\prime}=s^{\ast} x_{\ast}\theta^{\prime}=s^{\ast} e^{\ast}\theta=\theta. \]

{\rm (2)} is dual to {\rm (1)}.
\end{proof}

\begin{corollary}\label{ExactToShow}
Let $(\mathscr{C},\mathbb{E},\mathfrak{s})$ be an extriangulated category. For any $\mathbb{E}$-triangle $A\overset{x}{\longrightarrow}B\overset{y}{\longrightarrow}C\overset{\delta}{\dashrightarrow}$, the following sequences of natural transformations are exact.
\[ \mathscr{C}(C,-)\!\overset{\mathscr{C}(y,-)}{\Longrightarrow}\!\mathscr{C}(B,-)\!\overset{\mathscr{C}(x,-)}{\Longrightarrow}\!\mathscr{C}(A,-)\!\overset{\delta^\sharp}{\Longrightarrow}\!\mathbb{E}(C,-)\!\overset{\mathbb{E}(y,-)}{\Longrightarrow}\!\mathbb{E}(B,-)\!\overset{\mathbb{E}(x,-)}{\Longrightarrow}\!\mathbb{E}(A,-), \]
\[ \mathscr{C}(-,A)\!\overset{\mathscr{C}(-,x)}{\Longrightarrow}\!\mathscr{C}(-,B)\!\overset{\mathscr{C}(-,y)}{\Longrightarrow}\!\mathscr{C}(-,C)\!\overset{\delta_\sharp}{\Longrightarrow}\!\mathbb{E}(-,A)\!\overset{\mathbb{E}(-,x)}{\Longrightarrow}\!\mathbb{E}(-,B)\!\overset{\mathbb{E}(-,y)}{\Longrightarrow}\!\mathbb{E}(-,C). \]
\end{corollary}
\begin{proof}
This immediately follows from Propositions~\ref{PropExact1} and \ref{PropExact2}.
\end{proof}

The following lemma shows that the upper-right square of Diagram $(\ref{DiagET4})$ obtained by {\rm (ET4)} is a weak pushout.
\begin{lemma}\label{LemHomotPush}
Let $(\ref{DiagET4})$ be a commutative diagram in $\mathscr{C}$, where
\begin{eqnarray*}
A\overset{f}{\longrightarrow}B\overset{f^{\prime}}{\longrightarrow}D\overset{d^{\ast}\delta^{\prime\prime}}{\dashrightarrow}&,&\ \ B\overset{g}{\longrightarrow}C\overset{g^{\prime}}{\longrightarrow}F\overset{\delta^{\prime}}{\dashrightarrow}\\
A\overset{h}{\longrightarrow}C\overset{h^{\prime}}{\longrightarrow}E\overset{\delta^{\prime\prime}}{\dashrightarrow}&,&\ \ D\overset{d}{\longrightarrow}E\overset{e}{\longrightarrow}F\overset{f^{\prime}_{\ast}\delta^{\prime}}{\dashrightarrow}
\end{eqnarray*}
are $\mathbb{E}$-triangles, which satisfy $e^{\ast}\delta^{\prime}=f_{\ast}\delta^{\prime\prime}$.

Suppose we are given a commutative square
\[
\xy
(-7,6)*+{B}="0";
(7,6)*+{D}="2";
(-7,-6)*+{C}="4";
(7,-6)*+{Y}="6";
{\ar^{f^{\prime}} "0";"2"};
{\ar_{g} "0";"4"};
{\ar^{y} "2";"6"};
{\ar_{x} "4";"6"};
{\ar@{}|\circlearrowright "0";"6"};
\endxy
\]
in $\mathscr{C}$. Then there exists $z\in\mathscr{C}(E,Y)$ which makes the following diagram commutative.
\[
\xy
(-7,7)*+{B}="0";
(7,7)*+{D}="2";
(-7,-7)*+{C}="4";
(3,-15)*+{}="5";
(7,-7)*+{E}="6";
(15,-3)*+{}="7";
(18,-18)*+{Y}="8";
{\ar^{f^{\prime}} "0";"2"};
{\ar_{g} "0";"4"};
{\ar^{d} "2";"6"};
{\ar_{h^{\prime}} "4";"6"};
{\ar_{z} "6";"8"};
{\ar@/^0.60pc/^{y} "2";"8"};
{\ar@/_0.60pc/_{x} "4";"8"};
{\ar@{}|\circlearrowright "0";"6"};
{\ar@{}|\circlearrowright "5";"6"};
{\ar@{}|\circlearrowright "6";"7"};
\endxy
\]
\end{lemma}
\begin{proof}

By $(f^{\prime}_{\ast}\delta^{\prime})^\sharp(y)=y_{\ast} f^{\prime}_{\ast}\delta^{\prime}=x_{\ast} g_{\ast}\delta^{\prime}=0$ and the exactness of
\[ \mathscr{C}(E,Y)\overset{\mathscr{C}(d,Y)}{\longrightarrow}\mathscr{C}(D,Y)\overset{(f^{\prime}_{\ast}\delta^{\prime})^\sharp_Y}{\longrightarrow}\mathbb{E}(F,Y)\to0, \]
there exists $z_1\in\mathscr{C}(E,Y)$ satisfying $z_1\circ d=y.$ Then by
\[ (x-z_1\circ h^{\prime})\circ g=y\circ f^{\prime}-z_1\circ d\circ f^{\prime}=0 \]
and the exactness of
\[ \mathscr{C}(F,Y)\overset{\mathscr{C}(g^{\prime},Y)}{\longrightarrow}\mathscr{C}(C,Y)\overset{\mathscr{C}(g,Y)}{\longrightarrow}\mathscr{C}(B,Y), \]
there exists $z_2\in\mathscr{C}(F,Y)$ satisfying $z_2\circ g^{\prime}=x-z_1\circ h^{\prime}$.
If we put $z=z_1+z_2\circ e$, this satisfies the desired commutativity.
\end{proof}

\subsection{Shifted octahedrons}
Condition {\rm (ET4)} in Definition~\ref{DefExtCat} is an analog of the octahedron axiom {\rm (TR4)} for a triangulated category. As in the case of a triangulated category, we can make it slightly more rigid as follows.
\begin{lemma}\label{LemOctaRigid}
Let $(\mathscr{C},\mathbb{E},\mathfrak{s})$ be an extriangulated category. Let
\begin{eqnarray*}
&A\overset{f}{\longrightarrow}B\overset{f^{\prime}}{\longrightarrow}D\overset{\delta_f}{\dashrightarrow},&\\
&B\overset{g}{\longrightarrow}C\overset{g^{\prime}}{\longrightarrow}F\overset{\delta_g}{\dashrightarrow},&\\
&A\overset{h}{\longrightarrow}C\overset{h_0}{\longrightarrow}E_0\overset{\delta_h}{\dashrightarrow}&
\end{eqnarray*}
be any triplet of $\mathbb{E}$-triangles satisfying $h=g\circ f$.
Then there are morphisms $d_0,e_0$ in $\mathscr{C}$ which make the diagram
\begin{equation}\label{DiagET4Rigid}
\xy
(-21,7)*+{A}="0";
(-7,7)*+{B}="2";
(7,7)*+{D}="4";
(-21,-7)*+{A}="10";
(-7,-7)*+{C}="12";
(7,-7)*+{E_0}="14";
(-7,-21)*+{F}="22";
(7,-21)*+{F}="24";
{\ar^{f} "0";"2"};
{\ar^{f^{\prime}} "2";"4"};
{\ar@{=} "0";"10"};
{\ar_{g} "2";"12"};
{\ar^{d_0} "4";"14"};
{\ar_{h} "10";"12"};
{\ar_{h_0} "12";"14"};
{\ar_{g^{\prime}} "12";"22"};
{\ar^{e_0} "14";"24"};
{\ar@{=} "22";"24"};
{\ar@{}|\circlearrowright "0";"12"};
{\ar@{}|\circlearrowright "2";"14"};
{\ar@{}|\circlearrowright "12";"24"};
\endxy
\end{equation}
commutative, and satisfy the following compatibilities.
\begin{itemize}
\item[{\rm (i)}] $D\overset{d_0}{\longrightarrow}E_0\overset{e_0}{\longrightarrow}F\overset{f^{\prime}_{\ast}(\delta_g)}{\dashrightarrow}$ is an $\mathbb{E}$-triangle,
\item[{\rm (ii)}] $d_0^{\ast}(\delta_h)=\delta_f$,

\item[{\rm (iii)}] $f_{\ast}(\delta_h)=e_0^{\ast}(\delta_g)$. 
\end{itemize}
\end{lemma}
\begin{proof}
By {\rm (ET4)}, there exist an object $E\in\mathscr{C}$, a commutative diagram $(\ref{DiagET4})$ in $\mathscr{C}$, and an $\mathbb{E}$-triangle $A\overset{h}{\longrightarrow}C\overset{h^{\prime}}{\longrightarrow}E\overset{\delta^{\prime\prime}}{\dashrightarrow}$, which satisfy the following compatibilities.
\begin{itemize}
\item[{\rm (i$^{\prime}$)}] $D\overset{d}{\longrightarrow}E\overset{e}{\longrightarrow}F\overset{f^{\prime}_{\ast}(\delta_g)}{\dashrightarrow}$ is an $\mathbb{E}$-triangle,
\item[{\rm (ii$^{\prime}$)}] $d^{\ast}(\delta^{\prime\prime})=\delta_f$,
\item[{\rm (iii$^{\prime}$)}] $f_{\ast}(\delta^{\prime\prime})=e^{\ast}(\delta_g)$. 
\end{itemize}
By Remark~\ref{RemConeCocone}, we obtain a morphism of $\mathbb{E}$-triangles
\[
\xy
(-12,6)*+{A}="0";
(0,6)*+{C}="2";
(12,6)*+{E}="4";
(24,6)*+{}="6";
(-12,-6)*+{A}="10";
(0,-6)*+{C}="12";
(12,-6)*+{E_0}="14";
(24,-6)*+{}="16";
{\ar^{h} "0";"2"};
{\ar^{h^{\prime}} "2";"4"};
{\ar@{-->}^{\delta^{\prime\prime}} "4";"6"};
{\ar@{=} "0";"10"};
{\ar@{=} "2";"12"};
{\ar^{u} "4";"14"};
{\ar_{h} "10";"12"};
{\ar_{h_0} "12";"14"};
{\ar@{-->}_{\delta_h} "14";"16"};
{\ar@{}|\circlearrowright "0";"12"};
{\ar@{}|\circlearrowright "2";"14"};
\endxy
\]
in which $u$ is an isomorphism. In particular we have $\delta^{\prime\prime}=u^{\ast}(\delta_h)$.
If we put $d_0=u\circ d$ and $e_0=e\circ u^{-1}$, then the commutativity of $(\ref{DiagET4Rigid})$ follows from that of $(\ref{DiagET4})$.
By the definition of the equivalence relation, we have $[D\overset{d_0}{\longrightarrow}E_0\overset{e_0}{\longrightarrow}F]=[D\overset{d}{\longrightarrow}E\overset{e}{\longrightarrow}F]$. It is straightforward to check that {\rm (i$^{\prime}$),(ii$^{\prime}$),(iii$^{\prime}$)} imply {\rm (i),(ii),(iii)}.
\end{proof}

\begin{proposition}\label{PropBaer}
Let $(\mathscr{C},\mathbb{E},\mathfrak{s})$ be an extriangulated category. Then the following holds.
\begin{enumerate}
\item[(1)] Let $C$ be any object, and let
\[ A_1\overset{x_1}{\longrightarrow}B_1\overset{y_1}{\longrightarrow}C\overset{\delta_1}{\dashrightarrow},\quad A_2\overset{x_2}{\longrightarrow}B_2\overset{y_2}{\longrightarrow}C\overset{\delta_2}{\dashrightarrow} \]
be any pair of $\mathbb{E}$-triangles. Then there is a commutative diagram in $\mathscr{C}$
\begin{equation}\label{Diag_PropBaer}
\xy
(-7,21)*+{A_2}="-12";
(7,21)*+{A_2}="-14";
(-21,7)*+{A_1}="0";
(-7,7)*+{M}="2";
(7,7)*+{B_2}="4";
(-21,-7)*+{A_1}="10";
(-7,-7)*+{B_1}="12";
(7,-7)*+{C}="14";
{\ar@{=} "-12";"-14"};
{\ar_{m_2} "-12";"2"};
{\ar^{x_2} "-14";"4"};
{\ar^{m_1} "0";"2"};
{\ar^{e_1} "2";"4"};
{\ar@{=} "0";"10"};
{\ar_{e_2} "2";"12"};
{\ar^{y_2} "4";"14"};
{\ar_{x_1} "10";"12"};
{\ar_{y_1} "12";"14"};
{\ar@{}|\circlearrowright "-12";"4"};
{\ar@{}|\circlearrowright "0";"12"};
{\ar@{}|\circlearrowright "2";"14"};
\endxy
\end{equation}
which satisfies
\begin{eqnarray*}
&\mathfrak{s}(y_2^{\ast}\delta_1)=[A_1\overset{m_1}{\longrightarrow}M\overset{e_1}{\longrightarrow}B_2],&\\
&\mathfrak{s}(y_1^{\ast}\delta_2)=[A_2\overset{m_2}{\longrightarrow}M\overset{e_2}{\longrightarrow}B_1],&\\
&m_{1\ast}\delta_1+m_{2\ast}\delta_2=0.&
\end{eqnarray*}
\item[(2)] Dual of {\rm (1)}.
\end{enumerate}
\end{proposition}
\begin{proof}
By the additivity of $\mathfrak{s}$, we have
\[ \mathfrak{s}(\delta_1\oplus\delta_2)=[A_1\oplus A_2\overset{x_1\oplus x_2}{\longrightarrow}B_1\oplus B_2\overset{y_1\oplus y_2}{\longrightarrow}C\oplus C]. \]
Let
\[
\xy
(-18,0)*+{A_1}="-2";
(-15.5,1)*+{}="-10";
(-15.5,-1)*+{}="-11";
(-6.5,1)*+{}="-0";
(-6.5,-1)*+{}="-1";
(0,0)*+{A_1\oplus A_2}="0";
(15.5,1)*+{}="10";
(15.5,-1)*+{}="11";
(6.5,1)*+{}="+0";
(6.5,-1)*+{}="1";
(18,0)*+{A_2}="2";
{\ar^{\iota_1} "-10";"-0"};
{\ar^{p_1} "-1";"-11"};
{\ar@{<-}^{\iota_2} "+0";"10"};
{\ar@{<-}^{p_2} "11";"1"};
\endxy
\]
be a biproduct in $\mathscr{C}$.
Put $\mu=(\Delta_C)^{\ast}(\delta_1\oplus\delta_2)$ and take $\mathfrak{s}(\mu)=[A_1\oplus A_2\overset{j}{\longrightarrow}M\overset{k}{\longrightarrow}C]$. Then $\mu$ satisfies
\begin{equation}\label{Eq_PropBaer}
p_{1\ast}\mu=\delta_1\ \ \text{and}\ \ p_{2\ast}\mu=\delta_2.
\end{equation}
Applying {\rm (ET4)} to $\mathfrak{s}(0)=[A_1\overset{\iota_1}{\longrightarrow}A_1\oplus A_2\overset{p_2}{\longrightarrow}A_2]$ and $\mathfrak{s}(\mu)=[A_1\oplus A_2\overset{j}{\longrightarrow}M\overset{k}{\longrightarrow}C]$, we obtain $B_2^{\prime}\in\mathscr{C}$, $\theta_1\in\mathbb{E}(B_2^{\prime},A_1)$ and a commutative diagram
\[
\xy
(-24,8)*+{A_1}="0";
(-8,8)*+{A_1\oplus A_2}="2";
(8,8)*+{A_2}="4";
(-24,-7)*+{A_1}="10";
(-8,-7)*+{M}="12";
(8,-7)*+{B_2^{\prime}}="14";
(-8,-22)*+{C}="22";
(8,-22)*+{C}="24";
{\ar^(0.4){\iota_1} "0";"2"};
{\ar^(0.6){p_2} "2";"4"};
{\ar@{=} "0";"10"};
{\ar_{j} "2";"12"};
{\ar^{x_2^{\prime}} "4";"14"};
{\ar_{m_1} "10";"12"};
{\ar_{e_1^{\prime}} "12";"14"};
{\ar_{k} "12";"22"};
{\ar^{y_2^{\prime}} "14";"24"};
{\ar@{=} "22";"24"};
{\ar@{}|\circlearrowright "0";"12"};
{\ar@{}|\circlearrowright "2";"14"};
{\ar@{}|\circlearrowright "12";"24"};
\endxy
\]
which satisfy
\begin{itemize}
\item[{\rm (i)}] $\mathfrak{s}(p_{2\ast}\mu)=[A_2\overset{x_2^{\prime}}{\longrightarrow}B_2\overset{y_2^{\prime}}{\longrightarrow}C]$,
\item[{\rm (ii)}] $x_2^{\ast}\theta_1=0$,
\item[{\rm (iii)}] $\mathfrak{s}(\theta_1)=[A_1\overset{m_1}{\longrightarrow}M\overset{e_1^{\prime}}{\longrightarrow}B_2^{\prime}]$,
\item[{\rm (iv)}] $(\iota_1,\mathrm{id}_M,y_2^{\prime})$ is a morphism of $\mathbb{E}$-triangles. Especially, we have $y_2^{\prime\ast}\mu=\iota_{1\ast}\theta_1$.
\end{itemize}

In particular, we have
\[ {[}A_2\overset{x_2^{\prime}}{\longrightarrow}B_2\overset{y_2^{\prime}}{\longrightarrow}C]=\mathfrak{s}(\delta_2)=[A_2\overset{x_2}{\longrightarrow}B_2\overset{y_2}{\longrightarrow}C]. \]
Thus there is an isomorphism $b_2\in\mathscr{C}(B_2,B_2^{\prime})$ satisfying $b_2\circ x_2=x_2^{\prime}$ and $y_2^{\prime}\circ b_2=y_2$. If we put $e_1=b_2^{-1}\circ e_1^{\prime}$, then
\[ \mathfrak{s}(b_2^{\ast}\theta_1)=[A_1\overset{m_1}{\longrightarrow}M\overset{e_1}{\longrightarrow}B_2] \]
by Proposition~\ref{PropModif}. Thus we obtain a commutative diagram
\[
\xy
(-24,8)*+{A_1}="0";
(-8,8)*+{A_1\oplus A_2}="2";
(8,8)*+{A_2}="4";
(-24,-7)*+{A_1}="10";
(-8,-7)*+{M}="12";
(8,-7)*+{B_2}="14";
(-8,-22)*+{C}="22";
(8,-22)*+{C}="24";
{\ar^(0.4){\iota_1} "0";"2"};
{\ar^(0.6){p_2} "2";"4"};
{\ar@{=} "0";"10"};
{\ar_{j} "2";"12"};
{\ar^{x_2} "4";"14"};
{\ar_{m_1} "10";"12"};
{\ar_{e_1} "12";"14"};
{\ar_{k} "12";"22"};
{\ar^{y_2} "14";"24"};
{\ar@{=} "22";"24"};
{\ar@{}|\circlearrowright "0";"12"};
{\ar@{}|\circlearrowright "2";"14"};
{\ar@{}|\circlearrowright "12";"24"};
\endxy
\]
which satisfies
\[ y_2^{\ast}\delta_1=b_2^{\ast} y_2^{\prime\ast}p_{1\ast}\mu=b_2^{\ast} p_{1\ast}y_2^{\prime\ast}\mu=b_2^{\ast} p_{1\ast}\iota_{1\ast}\theta_1=b_2^{\ast}\theta_1. \]
Thus we obtain $\mathfrak{s}(y_2^{\ast}\delta_1)=[A_1\overset{m_1}{\longrightarrow}M\overset{e_1}{\longrightarrow}B_2]$. Similarly, from $\mathfrak{s}(0)=[A_2\overset{\iota_2}{\longrightarrow}A_1\oplus A_2\overset{p_1}{\longrightarrow}A_1]$ and $\mathfrak{s}(\mu)=[A_1\oplus A_2\overset{j}{\longrightarrow}M\overset{k}{\longrightarrow}C]$, we obtain a commutative diagram
\[
\xy
(-24,8)*+{A_2}="0";
(-8,8)*+{A_1\oplus A_2}="2";
(8,8)*+{A_1}="4";
(-24,-7)*+{A_2}="10";
(-8,-7)*+{M}="12";
(8,-7)*+{B_1}="14";
(-8,-22)*+{C}="22";
(8,-22)*+{C}="24";
{\ar^(0.4){\iota_2} "0";"2"};
{\ar^(0.6){p_1} "2";"4"};
{\ar@{=} "0";"10"};
{\ar_{j} "2";"12"};
{\ar^{x_1} "4";"14"};
{\ar_{m_2} "10";"12"};
{\ar_{e_2} "12";"14"};
{\ar_{k} "12";"22"};
{\ar^{y_1} "14";"24"};
{\ar@{=} "22";"24"};
{\ar@{}|\circlearrowright "0";"12"};
{\ar@{}|\circlearrowright "2";"14"};
{\ar@{}|\circlearrowright "12";"24"};
\endxy
\]
which satisfies $\mathfrak{s}(y_1^{\ast}\delta_2)=[A_2\overset{m_2}{\longrightarrow}M\overset{e_2}{\longrightarrow}B_1]$. Since
\begin{eqnarray*}
&e_2\circ m_1=e_2\circ j\circ\iota_1=x_1\circ p_1\circ\iota_1=x_1,&\\
&e_1\circ m_2=e_1\circ j\circ\iota_2=x_2\circ p_2\circ\iota_2=x_2,&\\
&y_2\circ e_1=k=y_1\circ e_2,&
\end{eqnarray*}
diagram $(\ref{Diag_PropBaer})$ is commutative. Moreover, we have
\[ m_{1\ast}\delta_1+m_{2\ast}\delta_2=j_{\ast}\big(\iota_{1\ast}\delta_1+\iota_{2\ast}\delta_2\big)=j_{\ast}\big((\iota_1\circ p_1)_{\ast}+(\iota_2\circ p_2)_{\ast}\big)(\mu)=j_{\ast}\mu=0 \]
by $(\ref{Eq_PropBaer})$ and Lemma~\ref{LemZero}.
\end{proof}

\begin{corollary}\label{CorTrivWIC}
Let $x\in\mathscr{C}(A,B),\, f\in\mathscr{C}(A,D)$ be any pair of morphisms. If $x$ is an inflation, then so is $\Big[\raise1ex\hbox{\leavevmode\vtop{\baselineskip-8ex \lineskip1ex \ialign{#\crcr{$f$}\crcr{$x$}\crcr}}}\Big]\in\mathscr{C}(A,D\oplus B)$. Dually for deflations.
\end{corollary}
\begin{proof}
Let $A\overset{x}{\longrightarrow}B\overset{y}{\longrightarrow}C\overset{\delta}{\dashrightarrow}$ be an $\mathbb{E}$-triangle. Realize $f_{\ast}\delta$ by an $\mathbb{E}$-triangle
\[ D\overset{d}{\longrightarrow}E\overset{e}{\longrightarrow}C\overset{f_{\ast}\delta}{\dashrightarrow}. \]
By Proposition~\ref{PropBaer} {\rm (1)}, we obtain a commutative diagram made of $\mathbb{E}$-triangles
\[
\xy
(-7,7)*+{A}="2";
(7,7)*+{A}="4";
(-21,-7)*+{D}="10";
(-7,-7)*+{M}="12";
(7,-7)*+{B}="14";
(21,-7)*+{}="16";
(-21,-21)*+{D}="20";
(-7,-21)*+{E}="22";
(7,-21)*+{C}="24";
(21,-21)*+{}="26";
(-7,-35)*+{}="32";
(7,-35)*+{}="34";
{\ar@{=} "2";"4"};
{\ar_{m} "2";"12"};
{\ar^{x} "4";"14"};
{\ar^{k} "10";"12"};
{\ar_{\ell} "12";"14"};
{\ar@{-->}^{y^{\ast} f_{\ast}\delta} "14";"16"};
{\ar@{=} "10";"20"};
{\ar^{e} "12";"22"};
{\ar^{y} "14";"24"};
{\ar_{d} "20";"22"};
{\ar_{e} "22";"24"};
{\ar@{-->}_{f_{\ast}\delta} "24";"26"};
{\ar@{-->}_{e^{\ast}\delta} "22";"32"};
{\ar@{-->}^{\delta} "24";"34"};
{\ar@{}|\circlearrowright "2";"14"};
{\ar@{}|\circlearrowright "10";"22"};
{\ar@{}|\circlearrowright "12";"24"};
\endxy
\]
satisfying $m_{\ast}\delta+k_{\ast} f_{\ast}\delta=0$. Since $y^{\ast} f_{\ast}\delta=f_{\ast} y^{\ast}\delta=0$, we may assume $M=D\oplus B,\ k=\Big[\raise1ex\hbox{\leavevmode\vtop{\baselineskip-8ex \lineskip1ex \ialign{#\crcr{$1$}\crcr{$0$}\crcr}}}\Big],\ \ell=[0\ 1]$, and take $p\in\mathscr{C}(M,D),\, i\in\mathscr{C}(B,M)$ which make
\[
\xy
(-16,0)*+{D}="-2";
(-13.5,1)*+{}="-10";
(-13.5,-1)*+{}="-11";
(-3,1)*+{}="-0";
(-3,-1)*+{}="-1";
(0,0)*+{M}="0";
(13.5,1)*+{}="10";
(13.5,-1)*+{}="11";
(3,1)*+{}="+0";
(3,-1)*+{}="1";
(16,0)*+{B}="2";
{\ar^{k} "-10";"-0"};
{\ar^{p} "-1";"-11"};
{\ar@{<-}^{i} "+0";"10"};
{\ar@{<-}^{\ell} "11";"1"};
\endxy
\]
a biproduct. By the exactness of
\[ \mathscr{C}(B,M)\overset{\mathscr{C}(x,M)}{\longrightarrow}\mathscr{C}(A,M)\overset{\delta^\sharp}{\longrightarrow}\mathbb{E}(C,M) \]
and the equality $\delta^\sharp(m+k\circ f)=m_{\ast}\delta+k_{\ast} f_{\ast}\delta=0$, there exists $b\in\mathscr{C}(B,M)$ satisfying $b\circ x=m+k\circ f$.

Modifying $A\overset{m}{\longrightarrow}M\overset{e}{\longrightarrow}E$ by the automorphism
\[ n=\left[\begin{array}{cc}-1&p\circ b\\0&1\end{array}\right]=\left[\begin{array}{cc}-1&0\\0&1\end{array}\right]\circ(\mathrm{id}_M-k\circ p\circ b\circ \ell)\colon M\overset{\cong}{\longrightarrow}M, \]
we obtain a conflation
\[ A\overset{n\circ m}{\longrightarrow}D\oplus B\overset{e\circ n^{-1}}{\longrightarrow}E. \]
Then, since
\begin{eqnarray*}
p\circ(n\circ m)&=&-p\circ(\mathrm{id}_M-k\circ p\circ b\circ \ell)\circ m\ =\ p\circ k\circ p\circ b\circ \ell\circ m-p\circ m\\
&=&p\circ b\circ x-p\circ m\ =\ p\circ k\circ f\ =\ f
\end{eqnarray*}
and
\[ \ell\circ(n\circ m)=\ell\circ(\mathrm{id}_M-k\circ p\circ b\circ \ell)\circ m=\ell\circ m=x, \]
we have $n\circ m=\Big[\raise1ex\hbox{\leavevmode\vtop{\baselineskip-8ex \lineskip1ex \ialign{#\crcr{$f$}\crcr{$x$}\crcr}}}\Big]$.
\end{proof}

\begin{proposition}\label{PropShiftOctahedron}
Suppose we are given $\mathbb{E}$-triangles
\begin{eqnarray*}
&D\overset{f}{\longrightarrow}A\overset{f^{\prime}}{\longrightarrow}C\overset{\delta_f}{\dashrightarrow},&\\
&A\overset{g}{\longrightarrow}B\overset{g^{\prime}}{\longrightarrow}F\overset{\delta_g}{\dashrightarrow},&\\
&E\overset{h}{\longrightarrow}B\overset{h^{\prime}}{\longrightarrow}C\overset{\delta_h}{\dashrightarrow}&
\end{eqnarray*}
satisfying $h^{\prime}\circ g=f^{\prime}$. Then there is an $\mathbb{E}$-triangle
\[ D\overset{d}{\longrightarrow}E\overset{e}{\longrightarrow}F\overset{\theta}{\dashrightarrow} \]
which makes
\begin{equation}\label{OctaShift}
\xy
(-7,7)*+{D}="2";
(7,7)*+{A}="4";
(21,7)*+{C}="6";
(-7,-7)*+{E}="12";
(7,-7)*+{B}="14";
(21,-7)*+{C}="16";
(-7,-21)*+{F}="22";
(7,-21)*+{F}="24";
{\ar^{f} "2";"4"};
{\ar^{f^{\prime}} "4";"6"};
{\ar_{d} "2";"12"};
{\ar^{g} "4";"14"};
{\ar@{=} "6";"16"};
{\ar^{h} "12";"14"};
{\ar^{h^{\prime}} "14";"16"};
{\ar_{e} "12";"22"};
{\ar^{g^{\prime}} "14";"24"};
{\ar@{=} "22";"24"};
{\ar@{}|\circlearrowright "2";"14"};
{\ar@{}|\circlearrowright "4";"16"};
{\ar@{}|\circlearrowright "12";"24"};
\endxy
\end{equation}
commutative in $\mathscr{C}$, and satisfy the following equalities.
\begin{itemize}
\item[{\rm (i)}] $d_{\ast}(\delta_f)=\delta_h$.
\item[{\rm (ii)}] $f_{\ast}(\theta)=\delta_g$.
\item[{\rm (iii)}] $g^{\prime\ast}(\theta)+h^{\prime\ast}(\delta_f)=0$.
\end{itemize}
\end{proposition}
\begin{proof}
By {\rm (ET4)}, we have $\mathbb{E}$-triangles
\[ D\overset{g\circ f}{\longrightarrow}B\overset{a}{\longrightarrow}G\overset{\mu}{\dashrightarrow}\ \ \text{and}\ \ C\overset{b}{\longrightarrow}G\overset{c}{\longrightarrow}F\overset{\nu}{\dashrightarrow} \]
which make the following diagram commutative in $\mathscr{C}$,
\begin{equation}\label{Comm_PropShiftOctahedron}
\xy
(-21,7)*+{D}="0";
(-7,7)*+{A}="2";
(7,7)*+{C}="4";
(-21,-7)*+{D}="10";
(-7,-7)*+{B}="12";
(7,-7)*+{G}="14";
(-7,-21)*+{F}="22";
(7,-21)*+{F}="24";
{\ar^{f} "0";"2"};
{\ar^{f^{\prime}} "2";"4"};
{\ar@{=} "0";"10"};
{\ar_{g} "2";"12"};
{\ar^{b} "4";"14"};
{\ar_{g\circ f} "10";"12"};
{\ar_{a} "12";"14"};
{\ar_{g^{\prime}} "12";"22"};
{\ar^{c} "14";"24"};
{\ar@{=} "22";"24"};
{\ar@{}|\circlearrowright "0";"12"};
{\ar@{}|\circlearrowright "2";"14"};
{\ar@{}|\circlearrowright "12";"24"};
\endxy
\end{equation}
and satisfy
\[ f^{\prime}_{\ast}(\delta_g)=\nu,\ \ b^{\ast}\mu=\delta_f,\ \ c^{\ast}(\delta_g)=f_{\ast}\mu. \]
It follows from Lemma~\ref{LemZero} that $\nu=f^{\prime}_{\ast}(\delta_g)=h^{\prime}_{\ast} g_{\ast}(\delta_g)=0$. Thus, up to equivalence, we may assume
\[ G=C\oplus F,\ b=\Big[\raise1ex\hbox{\leavevmode\vtop{\baselineskip-8ex \lineskip1ex \ialign{#\crcr{$1$}\crcr{$0$}\crcr}}}\Big],\ c=[0\ 1] \]
from the start. Then $a=\Big[\raise1ex\hbox{\leavevmode\vtop{\baselineskip-8ex \lineskip1ex \ialign{#\crcr{$a_1$}\crcr{$a_2$}\crcr}}}\Big]\colon B\to G=C\oplus F$ should satisfy
\[ a_1\circ g=f^{\prime},\ a_2=g^{\prime} \]
by the commutativity of $(\ref{Comm_PropShiftOctahedron})$. Since $h^{\prime}-a_1\in\mathscr{C}(B,C)$ satisfies $(h^{\prime}-a_1)\circ g=f^{\prime}-f^{\prime}=0$, there exists $z\in\mathscr{C}(F,C)$ satisfying $z\circ g^{\prime}=h^{\prime}-a_1$.

Put $z^{\prime}=\Big[\raise1ex\hbox{\leavevmode\vtop{\baselineskip-8ex \lineskip1ex \ialign{#\crcr{$-z$}\crcr{$\,\,1\,$}\crcr}}}\Big]$. Applying the dual of Lemma~\ref{LemOctaRigid} to the following diagram made of $\mathbb{E}$-triangles,
\[
\xy
(-7,7)*+{E}="2";
(7,7)*+{F}="4";
(-21,-7)*+{D}="10";
(-7,-7)*+{B}="12";
(7,-7)*+{G}="14";
(21,-7)*+{}="16";
(-7,-21)*+{C}="22";
(7,-21)*+{C}="24";
(-7,-34)*+{}="32";
(7,-34)*+{}="34";
{\ar_{h} "2";"12"};
{\ar^{z^{\prime}} "4";"14"};
{\ar_{g\circ f} "10";"12"};
{\ar^{a} "12";"14"};
{\ar@{-->}^{\mu} "14";"16"};
{\ar_{h^{\prime}} "12";"22"};
{\ar^{[1\ z]} "14";"24"};
{\ar@{=} "22";"24"};
{\ar@{-->}_{\delta_h} "22";"32"};
{\ar@{-->}^{0} "24";"34"};
{\ar@{}|\circlearrowright "12";"24"};
\endxy
\]
we obtain an $\mathbb{E}$-triangle
\[ D\overset{d}{\longrightarrow}E\overset{e}{\longrightarrow}F\overset{\theta}{\dashrightarrow} \]
which makes the following diagram commutative,
\[
\xy
(-21,7)*+{D}="0";
(-7,7)*+{E}="2";
(7,7)*+{F}="4";
(-21,-7)*+{D}="10";
(-7,-7)*+{B}="12";
(7,-7)*+{G}="14";
(-7,-21)*+{C}="22";
(7,-21)*+{C}="24";
{\ar^{d} "0";"2"};
{\ar^{e} "2";"4"};
{\ar@{=} "0";"10"};
{\ar_{h} "2";"12"};
{\ar^{z^{\prime}} "4";"14"};
{\ar_{g\circ f} "10";"12"};
{\ar_{a} "12";"14"};
{\ar_{h^{\prime}} "12";"22"};
{\ar^{[1\ z]} "14";"24"};
{\ar@{=} "22";"24"};
{\ar@{}|\circlearrowright "0";"12"};
{\ar@{}|\circlearrowright "2";"14"};
{\ar@{}|\circlearrowright "12";"24"};
\endxy
\]
and satisfies
\[ \theta=z^{\prime\ast}\mu,\ \ d_{\ast}\mu=[1\ z]^{\ast}(\delta_h). \]
Then the commutativity of $(\ref{OctaShift})$ can be checked in a straightforward way. Let us show the equalities {\rm (i),(ii),(iii)}.

{\rm (i)} follows from
\[ d_{\ast}(\delta_f)=d_{\ast} b^{\ast}\mu=b^{\ast} [1\ z]^{\ast}(\delta_h)=([1\ z]\circ b)^{\ast}(\delta_h)=\delta_h. \]
{\rm (ii)} follows from the injectivity of $\mathbb{E}(c,A)=c^{\ast}$ and
\begin{eqnarray*}
c^{\ast} f_{\ast}(\theta)&=&c^{\ast} f_{\ast} z^{\prime\ast}\mu\ =\ f_{\ast} (z^{\prime}\circ c)^{\ast}\mu\\
&=&f_{\ast} \left[\begin{array}{cc}0&-z\\ 0&1\end{array}\right]^{\ast}\mu%
\ =\ f_{\ast} (1-b\circ [1\ z])^{\ast}\mu\\
&=&f_{\ast}\mu-f_{\ast} [1\ z]^{\ast}(\delta_f)\ =\ f_{\ast}\mu-[1\ z]^{\ast} f_{\ast} (\delta_f)\ =\ f_{\ast}\mu\ =\ c^{\ast}(\delta_g).
\end{eqnarray*}
{\rm (iii)} follows from
\begin{eqnarray*}
g^{\prime\ast}(\theta)+h^{\prime\ast}(\delta_f)&=&g^{\prime\ast}z^{\prime\ast}\mu+h^{\prime\ast} b^{\ast}\mu\ =\ \Big(\Big[\raise1.3ex\hbox{\leavevmode\vtop{\baselineskip-8ex \lineskip0.4ex \ialign{#\crcr{$-z\circ g^{\prime}$}\crcr{$\ g^{\prime}\ $}\crcr}}}\Big]+\Big[\raise1ex\hbox{\leavevmode\vtop{\baselineskip-8ex \lineskip1ex \ialign{#\crcr{$h^{\prime}$}\crcr{$0$}\crcr}}}\Big]\Big)^{\ast}\mu\\
&=&\Big[\raise1.3ex\hbox{\leavevmode\vtop{\baselineskip-8ex \lineskip0.4ex \ialign{#\crcr{$a_1$}\crcr{$g^{\prime}$}\crcr}}}\Big]^{\ast}\mu\ =\ a^{\ast}\mu\ =\ 0.
\end{eqnarray*}
\end{proof}

As in Example~\ref{Example1}, an exact category (with some smallness assumption) can be regarded as an extriangulated category, whose inflations are monomorphic and whose deflations are epimorphic. Conversely, the following holds.
\begin{corollary}\label{CorExTraExact}
Let $(\mathscr{C},\mathbb{E},\mathfrak{s})$ be an extriangulated category, in which any inflation is monomorphic, and any deflation is epimorphic. If we let $\mathscr{S}$ be the class of conflations given by the $\mathbb{E}$-triangles (see Definition~\ref{DefTermExact1}), then $(\mathscr{C},\mathscr{S})$ is an exact category in the sense of \cite{Bu}.
\end{corollary}
\begin{proof}
By the exact sequences obtained in Proposition~\ref{PropExact1}, for any conflation $A\overset{x}{\longrightarrow}B\overset{y}{\longrightarrow}C$, the pair $(A,x)$ gives a weak kernel of $y$. Since $x$ is monomorphic by assumption, it is a kernel of $y$. Dually $(C,y)$ gives a cokernel of $x$, and $A\overset{x}{\longrightarrow}B\overset{y}{\longrightarrow}C$ becomes a kernel-cokernel pair. 

Thus $\mathscr{S}$ consists of some kernel-cokernel pairs. Moreover, it is closed under isomorphisms. Indeed, let $A\overset{x}{\longrightarrow}B\overset{y}{\longrightarrow}C\overset{\delta}{\dashrightarrow}$ be any $\mathbb{E}$-triangle, let $A^{\prime}\overset{x^{\prime}}{\longrightarrow}B^{\prime}\overset{y^{\prime}}{\longrightarrow}C^{\prime}$ be a kernel-cokernel pair, and suppose there are isomorphisms $a\in\mathscr{C}(A,A^{\prime}),b\in\mathscr{C}(B,B^{\prime}),c\in\mathscr{C}(C,C^{\prime})$ satisfying $x^{\prime}\circ a=b\circ x$ and $y^{\prime}\circ b=c\circ y$.
By Proposition~\ref{PropModif}, we obtain an $\mathbb{E}$-triangle $A^{\prime}\overset{x\circ a^{-1}}{\longrightarrow}B\overset{c\circ y}{\longrightarrow}C^{\prime}\overset{(c^{-1})^{\ast} a_{\ast}\delta}{\dashrightarrow}$.
Since
\[
\xy
(-16,0)*+{A^{\prime}}="0";
(3,0)*+{}="1";
(0,8)*+{B}="2";
(0,-8)*+{B^{\prime}}="4";
(-3,0)*+{}="5";
(16,0)*+{C^{\prime}}="6";
{\ar^{x\circ a^{-1}} "0";"2"};
{\ar^{c\circ y} "2";"6"};
{\ar_{x^{\prime}} "0";"4"};
{\ar_{y^{\prime}} "4";"6"};
{\ar^{b}_{\cong} "2";"4"};
{\ar@{}|\circlearrowright "0";"1"};
{\ar@{}|\circlearrowright "5";"6"};
\endxy
\]
is commutative in $\mathscr{C}$, this gives $\mathfrak{s}(\delta)=[A^{\prime}\overset{x\circ a^{-1}}{\longrightarrow}B\overset{c\circ y}{\longrightarrow}C^{\prime}]=[A^{\prime}\overset{x^{\prime}}{\longrightarrow}B^{\prime}\overset{y^{\prime}}{\longrightarrow}C^{\prime}]$, which means that $A^{\prime}\overset{x^{\prime}}{\longrightarrow}B^{\prime}\overset{y^{\prime}}{\longrightarrow}C^{\prime}$ belongs to $\mathscr{S}$.

Let us confirm conditions {\rm [E0],[E1],[E2]} in \cite[Definition 2.1]{Bu}. Since our assumptions are self-dual, the other conditions {\rm [E0$^{\mathrm{op}}$],[E1$^{\mathrm{op}}$],[E2$^{\mathrm{op}}$]} can be shown dually.


{\rm [E0]} For any object $A\in\mathscr{C}$,  the split sequence $A\overset{\mathrm{id}_A}{\longrightarrow}A\to0$ belongs to $\mathscr{S}$ by {\rm (ET2)}.


{\rm [E1]} The class of inflations (= admissible monics) is closed under composition by {\rm (ET4)}, as in Remark~\ref{RemInfInf}.


{\rm [E2]} Let $A\overset{x}{\longrightarrow}B\overset{y}{\longrightarrow}C\overset{\delta}{\dashrightarrow}$ be any $\mathbb{E}$-triangle, and let $a\in\mathscr{C}(A,A^{\prime})$ be any morphism. By Corollary~\ref{CorTrivWIC}, there exists a conflation
\[ A\overset{s}{\longrightarrow}B\oplus A^{\prime}\overset{{}^{\exists}[b\ x^{\prime}]}{\longrightarrow}{}^{\exists}B^{\prime}, \]
where $s=\Big[\raise1ex\hbox{\leavevmode\vtop{\baselineskip-8ex \lineskip1ex \ialign{#\crcr{$\,\,x\,$}\crcr{$-a$}\crcr}}}\Big]$. Since it becomes a kernel-cokernel pair by the above argument, it follows that
\[
\xy
(-6,6)*+{A}="0";
(6,6)*+{B}="2";
(-6,-6)*+{A^{\prime}}="4";
(6,-6)*+{B^{\prime}}="6";
%
%
{\ar^{x} "0";"2"};
{\ar_{a} "0";"4"};
{\ar^{b} "2";"6"};
{\ar_{x^{\prime}} "4";"6"};
{\ar@{}|\circlearrowright "0";"6"};
\endxy
\]
is a pushout square. By the dual of Proposition~\ref{PropShiftOctahedron}, we obtain the following commutative diagram made of conflations,
\[
\xy
(-9,24)*+{A^{\prime}}="2";
(9,24)*+{A^{\prime}}="4";
(-27,8)*+{A}="10";
(-9,8)*+{B\oplus A^{\prime}}="12";
(9,8)*+{B^{\prime}}="14";
%
(-27,-8)*+{A}="20";
(-9,-8)*+{B}="22";
(9,-8)*+{C}="24";
%
%
{\ar@{=} "2";"4"};
(-11.6,16)*+{\Big[\ \Big]}="-1";
(-11.6,17.6)*+{\scriptstyle{0}};
(-11.6,14.4)*+{\scriptstyle{1}};
%
{\ar "2";"12"};
{\ar^{x^{\prime}} "4";"14"};
{\ar^(0.4){s} "10";"12"};
{\ar_(0.56){[b\ x^{\prime}]} "12";"14"};
%
{\ar@{=} "10";"20"};
{\ar^{[1\ 0]} "12";"22"};
{\ar^{{}^{\exists}y^{\prime}} "14";"24"};
{\ar_{x} "20";"22"};
{\ar_{y} "22";"24"};
%
%
{\ar@{}|\circlearrowright "2";"14"};
{\ar@{}|\circlearrowright "10";"22"};
{\ar@{}|\circlearrowright "12";"24"};
\endxy
\]
which shows that $x^{\prime}$ is an inflation.
\end{proof}

\subsection{Relation with triangulated categories}

In this section, let $\mathscr{C}$ be an additive category equipped with an equivalence $[1]\colon\mathscr{C}\overset{\simeq}{\longrightarrow}\mathscr{C}$, and let $\mathbf{E}^1\colon\mathscr{C}^{\mathrm{op}}\times\mathscr{C}\to\mathit{Ab}$ be the bifunctor defined by $\mathbf{E}^1=\mathrm{Ext}^1(-,-)=\mathscr{C}(-,-[1])$.
\begin{remark}
As usual, we use notations like $X[1]$ and $f[1]$ for objects $X$ and morphisms $f$ in $\mathscr{C}$. The $n$-times composition of $[1]$ is denoted by $[n]$.
\end{remark}

We will show that, to give a triangulation of $\mathscr{C}$ with shift functor $[1]$, is equivalent to give an $\mathbf{E}^1$-triangulation of $\mathscr{C}$ (Proposition~\ref{PropTriaExt}).

\begin{remark}\label{RenTriaExt}
Let $\mathscr{C},[1],\mathbf{E}^1$ be as above. Then for any $\delta\in\mathbf{E}^1(C,A)=\mathscr{C}(C,A[1])$, we have the following.
\begin{enumerate}
\item[(1)] $\delta_\sharp=\mathscr{C}(-,\delta)\colon\mathscr{C}(-,C)\Rightarrow\mathscr{C}(-,A[1])$.
\item[(2)] $\delta^\sharp$ is given by
\[ \delta^\sharp_X\colon\mathscr{C}(A,X)\to\mathscr{C}(C,X[1])\ ;\ f\mapsto (f[1])\circ\delta \]
for any $X\in\mathscr{C}$.
\end{enumerate}
\end{remark}

\begin{lemma}\label{LemTriaExt}
Let $\mathscr{C},[1],\mathbf{E}^1$ be as above. Suppose that $\mathfrak{s}$ is an $\mathbf{E}^1$-triangulation of $\mathscr{C}$.
Then for any $A\in\mathscr{C}$, the $\mathbf{E}^1$-extension $\mathbf{1}=\mathrm{id}_{A[1]}\in\mathbf{E}^1(A[1],A)=\mathscr{C}(A[1],A[1])$ can be realized as
\[ \mathfrak{s}(\mathbf{1})=[A\to0\to A[1]]. \]
Namely, $A\to0\to A[1]\overset{\mathbf{1}}{\dashrightarrow}$ is an $\mathbf{E}^1$-triangle.
\end{lemma}
\begin{proof}
Put $\mathfrak{s}(\mathbf{1})=[A\overset{x}{\longrightarrow}X\overset{y}{\longrightarrow}A[1]]$. By Proposition~\ref{PropExact1},
\[ \mathscr{C}(-,A)\overset{\mathscr{C}(-,x)}{\Longrightarrow}\mathscr{C}(-,X)\overset{\mathscr{C}(-,y)}{\Longrightarrow}\mathscr{C}(-,A[1])\overset{\mathbf{1}_\sharp=\mathrm{id}}{\Longrightarrow}\mathscr{C}(-,A[1])\overset{\mathscr{C}(-,x[1])}{\Longrightarrow}\mathscr{C}(-,X[1]) \]
is exact. In particular $\mathrm{id}_X\in\mathscr{C}(X,X)$ satisfies $y=(\mathbf{1}_\sharp)_X\circ\mathscr{C}(X,y)(\mathrm{id}_X)=0$. Similarly $x[1]=0$ implies $x=0$. Thus
\[ 0\Rightarrow\mathscr{C}(-,X)\Rightarrow0 \]
becomes exact, which shows $X=0$.
\end{proof}

\begin{proposition}\label{PropTriaExt}
As before, let $\mathscr{C}$ be an additive category equipped with an auto-equivalence $[1]$, and put $\mathbf{E}^1=\mathscr{C}(-,-[1])$. Then we have the following.
\begin{enumerate}
\item[(1)] Suppose $\mathscr{C}$ is a triangulated category with shift functor $[1]$. For any $\delta\in\mathbf{E}^1(C,A)=\mathscr{C}(C,A[1])$, take a distinguished triangle
\[ A\overset{x}{\longrightarrow}B\overset{y}{\longrightarrow}C\overset{\delta}{\longrightarrow}A[1] \]
and define as $\mathfrak{s}(\delta)=[A\overset{x}{\longrightarrow}B\overset{y}{\longrightarrow}C]$. Remark that this $\mathfrak{s}(\delta)$ does not depend on the choice of the distinguished triangle above. With this definition, $(\mathscr{C},\mathbf{E}^1,\mathfrak{s})$ becomes an extriangulated category.
\item[(2)] Suppose we are given an $\mathbf{E}^1$-triangulation $\mathfrak{s}$ of $\mathscr{C}$. Define that $A\overset{x}{\longrightarrow}B\overset{y}{\longrightarrow}C\overset{\delta}{\longrightarrow}A[1]$ is a distinguished triangle if and only if it satisfies $\mathfrak{s}(\delta)=[A\overset{x}{\longrightarrow}B\overset{y}{\longrightarrow}C]$. With this class of distinguished triangles, $\mathscr{C}$ becomes a triangulated category.
\end{enumerate}
By construction, distinguished triangles correspond to $\mathbf{E}^1$-triangles by the above {\rm (1)} and {\rm (2)}.
\end{proposition}
\begin{proof}
{\rm (1)} is straightforward. For {\rm (2)}, all the axioms except for {\rm (TR2)} are easily confirmed.
Let us show {\rm (TR2)}. 

Let $A\overset{x}{\longrightarrow}B\overset{y}{\longrightarrow}C\overset{\delta}{\dashrightarrow}$ be any $\mathbf{E}^1$-triangle. Applying Proposition~\ref{PropBaer} {\rm (2)} to $A\to0\to A[1]\overset{\mathbf{1}}{\dashrightarrow}$ and $\delta$, we obtain
\[
\xy
(-7,7)*+{A}="2";
(7,7)*+{B}="4";
(21,7)*+{C}="6";
(-7,-7)*+{0}="12";
(7,-7)*+{{}^{\exists}M}="14";
(21,-7)*+{C}="16";
(-7,-21)*+{A[1]}="22";
(7,-21)*+{A[1]}="24";
{\ar^{x} "2";"4"};
{\ar^{y} "4";"6"};
{\ar_{} "2";"12"};
{\ar^{m^{\prime}} "4";"14"};
{\ar@{=} "6";"16"};
{\ar^{} "12";"14"};
{\ar^{e} "14";"16"};
{\ar_{} "12";"22"};
{\ar^{e^{\prime}} "14";"24"};
{\ar@{=} "22";"24"};
{\ar@{}|\circlearrowright "2";"14"};
{\ar@{}|\circlearrowright "4";"16"};
{\ar@{}|\circlearrowright "12";"24"};
\endxy
\]
which satisfies
\begin{itemize}
\item[{\rm (i)}] $[0\to M\overset{e}{\longrightarrow}C]=0_{\ast}\delta=0$,
\item[{\rm (ii)}] $\mathfrak{s}(x[1])=\mathfrak{s}(x_{\ast}\mathbf{1})=[B\overset{m^{\prime}}{\longrightarrow}M\overset{e^{\prime}}{\longrightarrow}A[1]]$,
\item[{\rm (iii)}] $e^{\ast}\delta+e^{\prime\ast}\mathbf{1}=0$.
\end{itemize}
{\rm (i)} shows that $e$ is an isomorphism, by Remark~\ref{RemSplit} {\rm (1)}. {\rm (iii)} means $\delta\circ e+e^{\prime}=0$ in $\mathscr{C}(M,A[1])$, namely $e^{\prime}\circ e^{-1}=-\delta$. Thus we have
\[ \mathfrak{s}(x[1])=[B\overset{y}{\longrightarrow}C\overset{-\delta}{\longrightarrow}A[1]] \]
by {\rm (ii)}, which means $B\overset{y}{\longrightarrow}C\overset{-\delta}{\longrightarrow}A[1]\overset{x[1]}{\longrightarrow}B[1]$ is a distinguished triangle. This is isomorphic to $B\overset{y}{\longrightarrow}C\overset{\delta}{\longrightarrow}A[1]\overset{-x[1]}{\longrightarrow}B[1]$.
\end{proof}

\subsection{Projectives and injectives}
If $(\mathscr{C},\mathbb{E},\mathfrak{s})$ has enough ``{\it projectives}", then the bifunctor $\mathbb{E}$ can be described in terms of them. Throughout this section, let $(\mathscr{C},\mathbb{E},\mathfrak{s})$ be an extriangulated category. Duals of the results in this section hold true for ``{\it injectives}".
\begin{definition}\label{DefProj}
An object $P\in\mathscr{C}$ is called {\it projective} if it satisfies the following condition.
\begin{itemize}
\item For any $\mathbb{E}$-triangle $A\overset{x}{\longrightarrow}B\overset{y}{\longrightarrow}C\overset{\delta}{\dashrightarrow}$ and any morphism $c\in\mathscr{C}(P,C)$, there exists $b\in\mathscr{C}(P,B)$ satisfying $y\circ b=c$.
\end{itemize}
We denote the full subcategory of projective objects in $\mathscr{C}$ by $\mathrm{Proj}(\mathscr{C})$. Dually, the full subcategory of injective objects in $\mathscr{C}$ is denoted by $\mathrm{Inj}(\mathscr{C})$.
\end{definition}

\begin{proposition}\label{PropProj}
An object $P\in\mathscr{C}$ is projective if and only if it satisfies $\mathbb{E}(P,A)=0$ for any $A\in\mathscr{C}$.
\end{proposition}
\begin{proof}
Suppose $P$ satisfies $\mathbb{E}(P,A)=0$ for any $A\in\mathscr{C}$. Then for any $\mathbb{E}$-triangle $A\overset{x}{\longrightarrow}B\overset{y}{\longrightarrow}C\overset{\delta}{\dashrightarrow}$,
\[ \mathscr{C}(P,B)\overset{\mathscr{C}(P,y)}{\longrightarrow}\mathscr{C}(P,C)\to0 \]
becomes exact by Proposition~\ref{PropExact1}. Thus $P$ is projective.

Conversely, suppose $P$ is projective. Let $A\in\mathscr{C}$ be any object, and let $\delta\in\mathbb{E}(P,A)$ be any element, with $\mathfrak{s}(\delta)=[A\overset{x}{\longrightarrow}M\overset{y}{\longrightarrow}P]$.
Since $P$ is projective, there exists $m\in\mathscr{C}(P,M)$ which makes the following diagram commutative.
\[
\xy
(-12,6)*+{0}="0";
(0,6)*+{P}="2";
(12,6)*+{P}="4";
(-12,-6)*+{A}="10";
(0,-6)*+{M}="12";
(12,-6)*+{P}="14";
{\ar "0";"2"};
{\ar^{\mathrm{id}_P} "2";"4"};
{\ar_{m} "2";"12"};
{\ar^{\mathrm{id}_P} "4";"14"};
{\ar_{x} "10";"12"};
{\ar_{y} "12";"14"};
{\ar@{}|\circlearrowright "2";"14"};
\endxy
\]
By {\rm (ET3)$^{\mathrm{op}}$}, the triplet $(0,m,\mathrm{id}_P)$ realizes the morphism $(0,\mathrm{id}_P)\colon0\to\delta$. Especially we have $\delta=\mathbb{E}(P,0)(0)=0$.
\end{proof}

\begin{definition}\label{DefEnoughProj}
Let $(\mathscr{C},\mathbb{E},\mathfrak{s})$ be an extriangulated category, as before. We say it {\it has enough projectives}, if it satisfies the following condition.
\begin{itemize}
\item For any object $C\in\mathscr{C}$, there exists an $\mathbb{E}$-triangle
\[ A\overset{x}{\longrightarrow}P\overset{y}{\longrightarrow}C\overset{\delta}{\dashrightarrow} \]
satisfying $P\in\mathrm{Proj}(\mathscr{C})$.
\end{itemize}
\end{definition}

\begin{example}
\begin{enumerate}
\item[(1)] If $(\mathscr{C},\mathbb{E},\mathfrak{s})$ is an exact category, then these agree with the usual definitions.
\item[(2)] If $(\mathscr{C},\mathbf{E}^1,\mathfrak{s})$ is a triangulated category as in the previous section, then $\mathrm{Proj}(\mathscr{C})$ consists of zero objects. Moreover it always has enough projectives.
\item[(3)] If $(\mathscr{C},\mathbf{E}^1,\mathfrak{s})$ is a triangulated category with a rigid subcategory $\mathscr{R}$ (i.e. for all $R_1,R_2\in\mathscr{R}, \mathrm{Ext}^1(R_1,R_2) = 0$), let $\mathscr{D}$ be its full subcategory whose objects are those objects $X$ that satisfy $\mathrm{Ext}^1(R,X)=0$
for all $R\in\mathscr{R}$. Then $\mathscr{D}$ is an additive and extension-closed subcategory of $\mathscr{C}$, which is thus extriangulated by Remark~\ref{RemETrExtClosed}. We then have:
\begin{itemize}
 \item[(a)] $\mathscr{R} \subseteq \mathrm{Proj}(\mathscr{D})$;
 \item[(b)] $\mathrm{Proj}(\mathscr{D})=\mathscr{R}$ and $\mathscr{D}$ has enough projectives if and only if $\mathscr{R}$
is contravariantly finite. 
\end{itemize}
\end{enumerate}
\end{example}

\begin{corollary}\label{CorExtIsom}
Assume that $(\mathscr{C},\mathbb{E},\mathfrak{s})$ has enough projectives.
For any object $C\in\mathscr{C}$ and any $\mathbb{E}$-triangle
\[ A\overset{x}{\longrightarrow}P\overset{y}{\longrightarrow}C\overset{\delta}{\dashrightarrow}\quad(P\in\mathrm{Proj}(\mathscr{C})), \]
the sequence
\[ \mathscr{C}(P,-)\overset{\mathscr{C}(x,-)}{\Longrightarrow}\mathscr{C}(A,-)\overset{\delta^\sharp}{\Longrightarrow}\mathbb{E}(C,-)\Rightarrow0 \]
is exact. Namely, we have a natural isomorphism
\begin{equation}\label{Isom_ExtCok}
\mathbb{E}(C,-)\cong\mathrm{Cok}\big(\mathscr{C}(x,-)\big).
\end{equation}
\end{corollary}
\begin{proof}
This immediately follows from Propositions~\ref{PropExact1} and \ref{PropProj}.
\end{proof}

The isomorphism $(\ref{Isom_ExtCok})$ in Corollary~\ref{CorExtIsom} is natural in $C$, in the following sense.
\begin{remark}\label{RemExtIsom}
Assume that $(\mathscr{C},\mathbb{E},\mathfrak{s})$ has enough projectives. Let $c\in\mathscr{C}(C,C^{\prime})$ be any morphism, and let
\[ A\overset{x}{\longrightarrow}P\overset{y}{\longrightarrow}C\overset{\delta}{\dashrightarrow},\quad A^{\prime}\overset{x^{\prime}}{\longrightarrow}P^{\prime}\overset{y^{\prime}}{\longrightarrow}C^{\prime}\overset{\delta^{\prime}}{\dashrightarrow} \]
be any pair of $\mathbb{E}$-triangles satisfying $P,P^{\prime}\in\mathrm{Proj}(\mathscr{C})$. By the projectivity of $P$ and {\rm (ET3)$^{\mathrm{op}}$}, we obtain a morphism of $\mathbb{E}$-triangles
\[
\xy
(-12,6)*+{A}="0";
(0,6)*+{P}="2";
(12,6)*+{C}="4";
(24,6)*+{}="6";
(-12,-6)*+{A^{\prime}}="10";
(0,-6)*+{P^{\prime}}="12";
(12,-6)*+{C^{\prime}}="14";
(24,-6)*+{}="16";
{\ar^{x} "0";"2"};
{\ar^{y} "2";"4"};
{\ar@{-->}^{\delta} "4";"6"};
{\ar_{a} "0";"10"};
{\ar^{p} "2";"12"};
{\ar^{c} "4";"14"};
{\ar_{x^{\prime}} "10";"12"};
{\ar_{y^{\prime}} "12";"14"};
{\ar@{-->}_{\delta^{\prime}} "14";"16"};
{\ar@{}|\circlearrowright "0";"12"};
{\ar@{}|\circlearrowright "2";"14"};
\endxy.
\]
This gives the following morphism of exact sequences.
\[
\xy
(-36,7)*+{\mathscr{C}(P,-)}="0";
(-12,7)*+{\mathscr{C}(A,-)}="2";
(11,7)*+{\mathbb{E}(C,-)}="4";
(27,7)*+{0}="6";
(-36,-7)*+{\mathscr{C}(P^{\prime},-)}="10";
(-12,-7)*+{\mathscr{C}(A^{\prime},-)}="12";
(11,-7)*+{\mathbb{E}(C^{\prime},-)}="14";
(27,-7)*+{0}="16";
{\ar@{=>}^{\mathscr{C}(x,-)} "0";"2"};
{\ar@{=>}^{\delta^\sharp} "2";"4"};
{\ar@{=>} "4";"6"};
{\ar@{=>}^{\mathscr{C}(p,-)} "10";"0"};
{\ar@{=>}^{\mathscr{C}(a,-)} "12";"2"};
{\ar@{=>}_{\mathbb{E}(c,-)} "14";"4"};
{\ar@{=>}_{\mathscr{C}(x^{\prime},-)} "10";"12"};
{\ar@{=>}_{\delta^{^{\prime}\sharp}} "12";"14"};
{\ar@{=>} "14";"16"};
{\ar@{}|\circlearrowright "0";"12"};
{\ar@{}|\circlearrowright "2";"14"};
\endxy
\]
\end{remark}

\begin{lemma}\label{LemReplaceI}
Let $A\overset{x}{\longrightarrow}B\overset{y}{\longrightarrow}C\overset{\delta}{\dashrightarrow}$ be any $\mathbb{E}$-triangle, and let $i\in\mathscr{C}(A,I)$ be any morphism to $I\in\mathrm{Inj}(\mathscr{C})$. Then for the projection $p_C\colon C\oplus I\to C$, the $\mathbb{E}$-extension $p_C^{\ast}\delta$ is realized by an $\mathbb{E}$-triangle of the form
\begin{equation}\label{Tria_xIyI}
A\overset{x_I}{\longrightarrow}B\oplus I\overset{y_I}{\longrightarrow}C\oplus I\overset{p_C^{\ast}\delta}{\dashrightarrow},%
\quad(x_I=\Big[\raise1ex\hbox{\leavevmode\vtop{\baselineskip-8ex \lineskip1ex \ialign{#\crcr{$x$}\crcr{$i$}\crcr}}}\Big],\ y_I=\left[\begin{array}{cc} y&\ast\\ \ast&\ast\end{array}\right]).
\end{equation}
\end{lemma}
\begin{proof}
By Corollary~\ref{CorTrivWIC}, we have an $\mathbb{E}$-triangle $A\overset{x_I}{\longrightarrow}B\oplus I\overset{d}{\longrightarrow}D\overset{\nu}{\dashrightarrow}$. By the dual of Proposition~\ref{PropShiftOctahedron}, we obtain the following commutative diagram made of $\mathbb{E}$-triangles
\[
\xy
(-9,21)*+{I}="2";
(9,21)*+{I}="4";
(-27,7)*+{A}="10";
(-9,7)*+{B\oplus I}="12";
(9,7)*+{D}="14";
(25,7)*+{}="16";
(-27,-7)*+{A}="20";
(-9,-7)*+{B}="22";
(9,-7)*+{C}="24";
(25,-7)*+{}="26";
(-9,-21)*+{}="32";
(9,-21)*+{}="34";
{\ar@{=} "2";"4"};
{\ar "2";"12"};
(-11.6,14)*+{\Big[\ \Big]}="-1";
(-11.6,15.6)*+{\scriptstyle{0}};
(-11.6,12.4)*+{\scriptstyle{1}};
{\ar^{e} "4";"14"};
{\ar^(0.4){x_I} "10";"12"};
{\ar_(0.56){d} "12";"14"};
{\ar@{-->}^{\nu} "14";"16"};
{\ar@{=} "10";"20"};
{\ar^{[1\ 0]} "12";"22"};
{\ar^{f} "14";"24"};
{\ar_{x} "20";"22"};
{\ar_{y} "22";"24"};
{\ar@{-->}_{\delta} "24";"26"};
{\ar@{-->}_{0} "22";"32"};
{\ar@{-->}^{{}^{\exists}\theta} "24";"34"};
{\ar@{}|\circlearrowright "2";"14"};
{\ar@{}|\circlearrowright "10";"22"};
{\ar@{}|\circlearrowright "12";"24"};
\endxy
\]
satisfying $f^{\ast}\delta=\nu$. Since $I\in\mathrm{Inj}(\mathscr{C})$, we have $\theta=0$. Thus there is some isomorphism $n\colon C\oplus I\overset{\cong}{\longrightarrow}D$ satisfying $n\circ\Big[\raise1ex\hbox{\leavevmode\vtop{\baselineskip-8ex \lineskip1ex \ialign{#\crcr{$0$}\crcr{$1$}\crcr}}}\Big]=e$ and $f\circ n=[1\ 0]$. Then for $p_C=[1\ 0]\colon C\oplus I\to C$,
\[
\xy
(-18,6)*+{A}="0";
(0,6)*+{B\oplus I}="2";
(20,6)*+{C\oplus I}="4";
(36,6)*+{}="6";
(-18,-6)*+{A}="10";
(0,-6)*+{B}="12";
(20,-6)*+{C}="14";
(36,-6)*+{}="16";
{\ar^{x_I} "0";"2"};
{\ar^{n^{-1}\circ d} "2";"4"};
{\ar@{-->}^{n^{\ast}\nu} "4";"6"};
{\ar@{=} "0";"10"};
{\ar^{[1\ 0]} "2";"12"};
{\ar^{[1\ 0]=p_C} "4";"14"};
{\ar_{x} "10";"12"};
{\ar_{y} "12";"14"};
{\ar@{-->}_{\delta} "14";"16"};
{\ar@{}|\circlearrowright "0";"12"};
{\ar@{}|\circlearrowright "2";"14"};
\endxy
\]
is a morphism of $\mathbb{E}$-triangles. Then $n^{-1}\circ d$ satisfies $p_C\circ n^{-1}\circ d\circ\Big[\raise1ex\hbox{\leavevmode\vtop{\baselineskip-8ex \lineskip1ex \ialign{#\crcr{$1$}\crcr{$0$}\crcr}}}\Big]=y\circ [1\ 0]\circ\Big[\raise1ex\hbox{\leavevmode\vtop{\baselineskip-8ex \lineskip1ex \ialign{#\crcr{$1$}\crcr{$0$}\crcr}}}\Big]=y$, and thus is of the form $\left[\begin{array}{cc} y&\ast\\ \ast&\ast\end{array}\right]$.
\end{proof}

The following construction gives extriangulated categories which are not exact nor triangulated in general. 
\begin{proposition}\label{PropReduction}
Let $\mathcal{I}\subseteq\mathscr{C}$ be a full additive subcategory, closed under isomorphisms. If $\mathcal{I}$ satisfies $\mathcal{I}\subseteq\mathrm{Proj}(\mathscr{C})\cap\mathrm{Inj}(\mathscr{C})$, then the ideal quotient $\mathscr{C}/\mathcal{I}$ has the structure of an extriangulated category, induced from that of $\mathscr{C}$. In particular, we can associate a ``{\it reduced}'' extriangulated category $\mathscr{C}^{\prime}=\mathscr{C}/(\mathrm{Proj}(\mathscr{C})\cap\mathrm{Inj}(\mathscr{C}))$ satisfying $\mathrm{Proj}(\mathscr{C}^{\prime})\cap\mathrm{Inj}(\mathscr{C}^{\prime})=0$, to any extriangulated category $(\mathscr{C},\mathbb{E},\mathfrak{s})$.
\end{proposition}
\begin{proof}
Put $\overline{\mathscr{C}}=\mathscr{C}/\mathcal{I}$. Let us confirm conditions {\rm (ET1),(ET2),(ET3),(ET4)}. The other conditions {\rm (ET3)$^{\mathrm{op}}$,(ET4)$^{\mathrm{op}}$} can be shown dually.

{\rm (ET1)} Since $\mathbb{E}(\mathcal{I},\mathscr{C})=\mathbb{E}(\mathscr{C},\mathcal{I})=0$, one can define the biadditive functor $\overline{\mathbb{E}}\colon\overline{\mathscr{C}}^{\mathrm{op}}\times\overline{\mathscr{C}}\to\mathit{Ab}$ given by
\begin{itemize}
\item $\overline{\mathbb{E}}(C,A)=\mathbb{E}(C,A)$\quad$(\forall A,C\in\mathscr{C})$,
\item $\overline{\mathbb{E}}(\overline{c},\overline{a})=\mathbb{E}(c,a)$\quad$(\forall a\in\mathscr{C}(A,A^{\prime}),c\in\mathscr{C}(C,C^{\prime}))$. Here, $\overline{a}$ and $\overline{c}$ denote the images of $a$ and $c$ in $\mathscr{C}/\mathcal{I}$.
\end{itemize}


{\rm (ET2)} For any $\overline{\mathbb{E}}$-extension $\delta\in\overline{\mathbb{E}}(C,A)=\mathbb{E}(C,A)$, define
\[ \overline{\mathfrak{s}}(\delta)=\overline{\mathfrak{s}(\delta)}=[A\overset{\overline{x_0}}{\longrightarrow}B\overset{\overline{y_0}}{\longrightarrow}C], \]
using $\mathfrak{s}(\delta)=[A\overset{x_0}{\longrightarrow}B\overset{y_0}{\longrightarrow}C]$.
Let us show that $\overline{\mathfrak{s}}$ is an additive realization of $\overline{\mathbb{E}}$.

Let $(\overline{a},\overline{c})\colon\delta=(A,\delta,C)\to\delta^{\prime}=(A^{\prime},\delta^{\prime},C^{\prime})$ be any morphism of $\overline{\mathbb{E}}$-extensions. By definition, this is equivalent to that $(a,c)\colon\delta\to\delta^{\prime}$ is a morphism of $\mathbb{E}$-extensions. Put $\overline{\mathfrak{s}}(\delta)=[A\overset{\overline{x}}{\longrightarrow}B\overset{\overline{y}}{\longrightarrow}C]$, $\overline{\mathfrak{s}}(\delta^{\prime})=[A^{\prime}\overset{\overline{x^{\prime}}}{\longrightarrow}B^{\prime}\overset{\overline{y^{\prime}}}{\longrightarrow}C^{\prime}]$. Since the condition in Definition~\ref{DefRealization} does not depend on the representatives of the equivalence classes of sequences in $\overline{\mathscr{C}}$, we may assume $\mathfrak{s}(\delta)=[A\overset{x}{\longrightarrow}B\overset{y}{\longrightarrow}C]$, $\mathfrak{s}(\delta^{\prime})=[A^{\prime}\overset{x^{\prime}}{\longrightarrow}B^{\prime}\overset{y^{\prime}}{\longrightarrow}C^{\prime}]$. Then there is $b\in\mathscr{C}(B,B^{\prime})$ with which $(a,b,c)$ realizes $(a,c)$. It follows that $(\overline{a},\overline{b},\overline{c})$ realizes $(\overline{a},\overline{c})$.

As for the additivity, the equality $\overline{\mathfrak{s}}(0)=0$ is trivially satisfied. Since $\overline{\mathfrak{s}}(\delta)\oplus\overline{\mathfrak{s}}(\delta^{\prime})$ only depends on the equivalence classes $\overline{\mathfrak{s}}(\delta)$ and $\overline{\mathfrak{s}}(\delta^{\prime})$, the equality $\overline{\mathfrak{s}}(\delta\oplus\delta^{\prime})=\overline{\mathfrak{s}}(\delta)\oplus\overline{\mathfrak{s}}(\delta^{\prime})$ follows from $\mathfrak{s}(\delta\oplus\delta^{\prime})=\mathfrak{s}(\delta)\oplus\mathfrak{s}(\delta^{\prime})$.


{\rm (ET3)} Suppose we are given  $\overline{\mathfrak{s}}(\delta)=[A\overset{\overline{x}}{\longrightarrow}B\overset{\overline{y}}{\longrightarrow}C]$, $\overline{\mathfrak{s}}(\delta^{\prime})=[A^{\prime}\overset{\overline{x^{\prime}}}{\longrightarrow}B^{\prime}\overset{\overline{y^{\prime}}}{\longrightarrow}C^{\prime}]$, and morphisms $\overline{a}\in\overline{\mathscr{C}}(A,A^{\prime}), \overline{b}\in\overline{\mathscr{C}}(B,B^{\prime})$ satisfying $\overline{x^{\prime}}\circ\overline{a}=\overline{b}\circ\overline{x}$. As in the proof of {\rm (ET2)}, we may assume $\mathfrak{s}(\delta)=[A\overset{x}{\longrightarrow}B\overset{y}{\longrightarrow}C]$, $\mathfrak{s}(\delta^{\prime})=[A^{\prime}\overset{x^{\prime}}{\longrightarrow}B^{\prime}\overset{y^{\prime}}{\longrightarrow}C^{\prime}]$. By $\overline{x^{\prime}}\circ\overline{a}=\overline{b}\circ\overline{x}$, there exist $I\in\mathcal{I},i\in\mathscr{C}(A,I),j\in\mathscr{C}(I,B^{\prime})$ which satisfy $x^{\prime}\circ a=b\circ x+j\circ i$. By Lemma~\ref{LemReplaceI}, we obtain an $\mathbb{E}$-triangle $(\ref{Tria_xIyI})$. This gives the following isomorphism of $\overline{\mathbb{E}}$-triangles.
\begin{equation}\label{Repl1}
\xy
(-18,6)*+{A}="0";
(0,6)*+{B\oplus I}="2";
(20,6)*+{C\oplus I}="4";
(36,6)*+{}="6";
(-18,-6)*+{A}="10";
(0,-6)*+{B}="12";
(20,-6)*+{C}="14";
(36,-6)*+{}="16";
{\ar^{\overline{x_I}} "0";"2"};
{\ar^{\overline{y_I}} "2";"4"};
{\ar@{-->}^(0.54){p_C^{\ast}\delta} "4";"6"};
{\ar@{=} "0";"10"};
{\ar^{\overline{p_B}} "2";"12"};
{\ar^{\overline{p_C}} "4";"14"};
{\ar_{\overline{x}} "10";"12"};
{\ar_{\overline{y}} "12";"14"};
{\ar@{-->}_{\delta} "14";"16"};
{\ar@{}|\circlearrowright "0";"12"};
{\ar@{}|\circlearrowright "2";"14"};
\endxy
\end{equation}
On the other hand, by {\rm (ET3)} for $(\mathscr{C},\mathbb{E},\mathfrak{s})$, we have a morphism of $\mathbb{E}$-triangles as follows.
\begin{equation}\label{Repl2}
\xy
(-18,6)*+{A}="0";
(0,6)*+{B\oplus I}="2";
(20,6)*+{C\oplus I}="4";
(36,6)*+{}="6";
(-18,-6)*+{A^{\prime}}="10";
(0,-6)*+{B^{\prime}}="12";
(20,-6)*+{C^{\prime}}="14";
(36,-6)*+{}="16";
{\ar^{x_I} "0";"2"};
{\ar^{y_I} "2";"4"};
{\ar@{-->}^(0.54){p_C^{\ast}\delta} "4";"6"};
{\ar@{=} "0";"10"};
{\ar^{[b\ j]} "2";"12"};
{\ar^{c} "4";"14"};
{\ar_{x^{\prime}} "10";"12"};
{\ar_{y^{\prime}} "12";"14"};
{\ar@{-->}_{\delta^{\prime}} "14";"16"};
{\ar@{}|\circlearrowright "0";"12"};
{\ar@{}|\circlearrowright "2";"14"};
\endxy
\end{equation}
From $(\ref{Repl1})$ and $(\ref{Repl2})$, we obtain a morphism of $\overline{\mathbb{E}}$-extensions $(\overline{a},\overline{c}\circ\overline{p_C}^{-1})\colon\delta\to\delta^{\prime}$ which satisfies $(\overline{c}\circ\overline{p_C}^{-1})\circ\overline{y}=\overline{y^{\prime}}\circ\overline{b}$.


{\rm (ET4)} Let $A\overset{\overline{f}}{\longrightarrow}B\overset{\overline{f^{\prime}}}{\longrightarrow}D\overset{\delta}{\dashrightarrow}$ and $B\overset{\overline{g}}{\longrightarrow}C\overset{\overline{g^{\prime}}}{\longrightarrow}F\overset{\delta^{\prime}}{\dashrightarrow}$ be $\overline{\mathbb{E}}$-triangles. As in the above arguments, we may assume $A\overset{f}{\longrightarrow}B\overset{f^{\prime}}{\longrightarrow}D\overset{\delta}{\dashrightarrow}$ and $B\overset{g}{\longrightarrow}C\overset{g^{\prime}}{\longrightarrow}F\overset{\delta^{\prime}}{\dashrightarrow}$ are $\mathbb{E}$-triangles. Then by {\rm (ET4)} for $(\mathscr{C},\mathbb{E},\mathfrak{s})$, we obtain a commutative diagram $(\ref{DiagET4})$ made of conflations, satisfying $\mathfrak{s}(f^{\prime}_{\ast}\delta^{\prime})=[D\overset{d}{\longrightarrow}E\overset{e}{\longrightarrow}F]$, $d^{\ast}\delta^{\prime\prime}=\delta$ and $f_{\ast}\delta^{\prime\prime}=e^{\ast}\delta^{\prime}$. The image of this diagram in $\overline{\mathscr{C}}$ shows {\rm (ET4)} for $(\overline{\mathscr{C}},\overline{\mathbb{E}},\overline{\mathfrak{s}})$.
\end{proof}

\begin{remark}
Proposition~\ref{PropReduction} applied to an exact category, together with Corollary~\ref{CorExTraExact}, gives an another proof\footnote{The first author wishes to thank Professor Osamu Iyama for informing this to him.} of \cite[Theorem 3.5]{DI}.
\end{remark}

\begin{corollary}\label{CorReduction}
Let $\mathcal{I}\subseteq\mathscr{C}$ be a full additive subcategory closed under isomorphisms, satisfying $\mathbb{E}(\mathcal{I},\mathcal{I})=0$. Let $\mathcal{Z}\subseteq\mathscr{C}$ be the full subcategory of those $Z\in\mathscr{C}$ satisfying $\mathbb{E}(Z,\mathcal{I})=\mathbb{E}(\mathcal{I},Z)=0$. Then, $\mathcal{Z}/\mathcal{I}$ is extriangulated.
\end{corollary}
\begin{proof}
This follows from Remark~\ref{RemETrExtClosed} and Proposition~\ref{PropReduction}.
\end{proof}

\section{Cotorsion pairs}\label{section_Cot}
In the rest of this article, let $(\mathscr{C},\mathbb{E},\mathfrak{s})$ be an extriangulated category.

\subsection{Cotorsion pairs}

\begin{definition}\label{DefCotors}
Let $\mathcal{U},\mathcal{V}\subseteq\mathscr{C}$ be a pair of full additive subcategories, closed under isomorphisms and direct summands. The pair $(\mathcal{U},\mathcal{V})$ is called a {\it cotorsion pair} on $\mathscr{C}$ if it satisfies the following conditions.
\begin{enumerate}
\item[(1)] $\mathbb{E}(\mathcal{U},\mathcal{V})=0$.
\item[(2)] For any $C\in\mathscr{C}$, there exists a conflation
\[ V^C\to U^C\to C \]
satisfying $U^C\in\mathcal{U},V^C\in\mathcal{V}$.
\item[(3)] For any $C\in\mathscr{C}$, there exists a conflation
\[ C\to V_C\to U_C \]
satisfying $U_C\in\mathcal{U},V_C\in\mathcal{V}$.
\end{enumerate}
\end{definition}

\begin{definition}\label{DefSubCone}
Let $\mathcal{X},\mathcal{Y}\subseteq\mathscr{C}$ be any pair of full subcategories closed under isomorphisms. Define full subcategories $\mathrm{Cone}(\mathcal{X},\mathcal{Y})$ and $\mathrm{CoCone}(\mathcal{X},\mathcal{Y})$ of $\mathscr{C}$ as follows. These are closed under isomorphisms.
\begin{enumerate}
\item[{\rm (i)}] $C\in\mathscr{C}$ belongs to $\mathrm{Cone}(\mathcal{X},\mathcal{Y})$ if and only if it admits a conflation $X\to Y\to C$ satisfying $X\in\mathcal{X},Y\in\mathcal{Y}$.
\item[{\rm (ii)}] $C\in\mathscr{C}$ belongs to $\mathrm{CoCone}(\mathcal{X},\mathcal{Y})$ if and only if it admits a conflation $C\to X\to Y$ satisfying $X\in\mathcal{X},Y\in\mathcal{Y}$.
\end{enumerate}
If $\mathcal{X}$ and $\mathcal{Y}$ are additive subcategories of $\mathscr{C}$, then so are $\mathrm{Cone}(\mathcal{X},\mathcal{Y})$ and $\mathrm{CoCone}(\mathcal{X},\mathcal{Y})$, by condition {\rm (ET2)}.
\end{definition}

\begin{remark}
 In the case of exact categories, cotorsion pairs satisfying (2) and (3) are often called \emph{complete} cotorsion pairs. Since all cotorsion pairs considered in this article are complete, this adjective is omitted.
Also remark that completeness is equivalent to the equalities $\mathscr{C}=\mathrm{Cone}(\mathcal{V},\mathcal{U})=\mathrm{CoCone}(\mathcal{V},\mathcal{U})$, in the notation of Definition~\ref{DefSubCone}.
\end{remark}

\begin{remark}\label{RemCotors}
Let $(\mathcal{U},\mathcal{V})$ be a cotorsion pair on $(\mathscr{C},\mathbb{E},\mathfrak{s})$. By Remark~\ref{RemSplit} {\rm (1)}, the following holds for any $C\in\mathscr{C}$.
\begin{enumerate}
\item[(1)] $C\in\mathcal{U}\Leftrightarrow\mathbb{E}(C,\mathcal{V})=0$.
\item[(2)] $C\in\mathcal{V}\Leftrightarrow\mathbb{E}(\mathcal{U},C)=0$.
\end{enumerate}
\end{remark}

\begin{corollary}\label{CorSectionClosed}
Let $(\mathcal{U},\mathcal{V})$ be a cotorsion pair on $(\mathscr{C},\mathbb{E},\mathfrak{s})$. Let $C\in\mathscr{C}$, $U\in\mathcal{U}$ be any pair of objects. If there exists a section $C\to U$ or a retraction $U\to C$, then $C$ also belongs to $\mathcal{U}$. Similarly for $\mathcal{V}$.
\end{corollary}
\begin{proof}
In either case, there are $s\in\mathscr{C}(C,U)$ and $r\in\mathscr{C}(U,C)$ satisfying $r\circ s=\mathrm{id}_C$. This gives the following commutative diagram of natural transformations.
\[
\xy
(-15,-5)*+{\mathbb{E}(C,-)}="0";
(0,6)*+{\mathbb{E}(U,-)}="2";
(15,-5)*+{\mathbb{E}(C,-)}="4";
(0,-12)*+{}="5";
{\ar@{=>}^{\mathbb{E}(r,-)} "0";"2"};
{\ar@{=>}^{\mathbb{E}(s,-)} "2";"4"};
{\ar@{=>}@/_0.80pc/_{\mathbb{E}(\mathrm{id}_C,-)=\mathrm{id}} "0";"4"};
{\ar@{}|\circlearrowright "2";"5"};
\endxy
\]
Thus $\mathbb{E}(U,\mathcal{V})=0$ implies $\mathbb{E}(C,\mathcal{V})=0$, and thus $C\in\mathcal{U}$ by Remark~\ref{RemCotors}.
\end{proof}

\begin{remark}\label{RemExtClosed}
Let $(\mathcal{U},\mathcal{V})$ be a cotorsion pair on $\mathscr{C}$. By Proposition~\ref{PropExact2}, the subcategories $\mathcal{U}$ and $\mathcal{V}$ are extension-closed in $\mathscr{C}$.
\end{remark}

\begin{remark}\label{RemCotorsProj}
By Proposition~\ref{PropProj} and Remark~\ref{RemCotors}, the following are equivalent.
\begin{enumerate}
\item[(1)] $(\mathcal{X},\mathscr{C})$ is a cotorsion pair for some subcategory $\mathcal{X}\subseteq\mathscr{C}$.
\item[(2)] $(\mathrm{Proj}(\mathscr{C}),\mathscr{C})$ is a cotorsion pair.
\item[(3)] $\mathscr{C}$ has enough projectives.
\end{enumerate}
Dually, the following are equivalent.
\begin{enumerate}
\item[(1)] $(\mathscr{C},\mathcal{X})$ is a cotorsion pair for some subcategory $\mathcal{X}\subseteq\mathscr{C}$.
\item[(2)] $(\mathscr{C},\mathrm{Inj}(\mathscr{C}))$ is a cotorsion pair.
\item[(3)] $\mathscr{C}$ has enough injectives.
\end{enumerate}
\end{remark}

\subsection{Associated adjoint functors}

\begin{definition}\label{DefCore}
For a cotorsion pair $(\mathcal{U},\mathcal{V})$ on $\mathscr{C}$, put $\mathcal{I}=\mathcal{U}\cap\mathcal{V}$, and call it the {\it core} of $(\mathcal{U},\mathcal{V})$.
For any full additive subcategory $\mathcal{X}\subseteq\mathscr{C}$ containing $\mathcal{I}$, let $\mathcal{X}/\mathcal{I}$  denote the ideal quotient. The image of a morphism $f$ in the ideal quotient is denoted by $\overline{f}$.
\end{definition}

\begin{lemma}\label{LemUV}
For any cotorsion pair $(\mathcal{U},\mathcal{V})$, we have $(\mathscr{C}/\mathcal{I})(\mathcal{U}/\mathcal{I},\mathcal{V}/\mathcal{I})=0$. Namely, for any $U\in\mathcal{U}$ and $V\in\mathcal{V}$, any morphism $f\in\mathscr{C}(U,V)$ factors through some $I\in\mathcal{I}$.
\end{lemma}
\begin{proof}
Resolve $V$ by an $\mathbb{E}$-triangle
\[ V^{\prime}\overset{x}{\longrightarrow}U^{\prime}\overset{y}{\longrightarrow}V\overset{\lambda}{\dashrightarrow} \]
satisfying $U^{\prime}\in\mathcal{U},\, V^{\prime}\in\mathcal{V}$. Since $\mathcal{V}$ is extension-closed, it follows that $U^{\prime}\in\mathcal{U}\cap\mathcal{V}=\mathcal{I}$. By the exactness of
\[ \mathscr{C}(U,U^{\prime})\overset{\mathscr{C}(U,y)}{\longrightarrow}\mathscr{C}(U,V)\overset{(\lambda_\sharp)_U}{\longrightarrow}\mathbb{E}(U,V^{\prime})=0 \]
obtained in Proposition~\ref{PropExact1}, any morphism $f\in\mathscr{C}(U,V)$ factors through $U^{\prime}\in\mathcal{I}$.
\end{proof}

\begin{proposition}\label{PropUC}
Let $C\in\mathscr{C}$ be any object, and let $\lambda$ be an $\mathbb{E}$-extension with
\[ \mathfrak{s}(\lambda)=[V^C\overset{v^C}{\longrightarrow}U^C\overset{u^C}{\longrightarrow}C]\quad(U^C\in\mathcal{U},V^C\in\mathcal{V}). \]
Then the morphism $\overline{u}^C$ has the following property.
\begin{itemize}
\item For any $U\in\mathcal{U}$, the map
\begin{equation}\label{ToShowBij}
\overline{u}^C\circ-\colon(\mathscr{C}/\mathcal{I})(U,U^C)\to(\mathscr{C}/\mathcal{I})(U,C)
\end{equation}
is bijective.
\end{itemize}
\end{proposition}
\begin{proof}
By the exactness of
\[ \mathscr{C}(U,U^C)\overset{\mathscr{C}(U,u^C)}{\longrightarrow}\mathscr{C}(U,C)\overset{(\lambda_\sharp)_U}{\longrightarrow}\mathbb{E}(U,V)=0, \]
the map $\mathscr{C}(U,U^C)\to\mathscr{C}(U,C)$ is surjective. This implies the surjectivity of $(\ref{ToShowBij})$.
Let us show the injectivity of $(\ref{ToShowBij})$. Let $g\in\mathscr{C}(U,U^C)$ be any morphism which satisfies $\overline{u}^C\circ\overline{g}=\overline{u^C\circ g}=0$. By definition, there exist $I\in\mathcal{I}$, $i_1\in\mathscr{C}(U,I)$ and $i_2\in\mathscr{C}(I,C)$ which makes the following diagram commutative.
\[
\xy
(0,6)*+{U}="2";
(12,6)*+{I}="4";
(-12,-6)*+{V^C}="10";
(0,-6)*+{U^C}="12";
(12,-6)*+{C}="14";
{\ar^{i_1} "2";"4"};
{\ar_{g} "2";"12"};
{\ar^{i_2} "4";"14"};
{\ar_{v^C} "10";"12"};
{\ar_{u^C} "12";"14"};
{\ar@{}|\circlearrowright "2";"14"};
\endxy
\]
Since $\mathbb{E}(I,V^C)=0$,
\[ \mathscr{C}(I,V^C)\overset{\mathscr{C}(I,v^C)}{\longrightarrow}\mathscr{C}(I,U^C)\overset{\mathscr{C}(I,u^C)}{\longrightarrow}\mathscr{C}(I,C)\to0 \]
is exact. Thus there exists $j\in\mathscr{C}(I,U^C)$ satisfying $u^C\circ j=i_2$. Then by $u^C\circ(g-j\circ i_1)=0$, we obtain $h\in\mathscr{C}(U,V^C)$ satisfying $v^C\circ h=g-j\circ i_1$.

By Lemma~\ref{LemUV}, this $h$ factors through some $I^{\prime}\in\mathcal{I}$. It follows that $g=v^C\circ h+j\circ i_1$ factors through $I\oplus I^{\prime}\in\mathcal{I}$.
\end{proof}

Proposition~\ref{PropUC} means that $(U^C,\overline{u}^C)$ is coreflection of $C\in\mathscr{C}/\mathcal{I}$ along the inclusion functor $\mathcal{U}/\mathcal{I}\hookrightarrow\mathscr{C}/\mathcal{I}$ (\cite[Definition~3.1.1]{Bo}). As a corollary, we have the following.
\begin{corollary}\label{CorUC}
The inclusion functor $\mathcal{U}/\mathcal{I}\hookrightarrow\mathscr{C}/\mathcal{I}$ has a right adjoint $\omega_{\mathcal{U}}\colon\mathscr{C}/\mathcal{I}\to\mathcal{U}/\mathcal{I}$, which assigns $\omega_{\mathcal{U}}(C)=U^C$ for any $C\in\mathscr{C}/\mathcal{I}$, where
\[ V^C\overset{v^C}{\longrightarrow}U^C\overset{u^C}{\longrightarrow}C\quad(U^C\in\mathcal{U},V^C\in\mathcal{V}) \]
is a conflation. Moreover $\varepsilon_{\mathcal{U}}=\{\overline{u}^C\}_{C\in\mathrm{Ob}(\mathscr{C}/\mathcal{I})}$ gives the counit of this adjoint pair.
\end{corollary}

\subsection{Concentric twin cotorsion pairs}

\begin{definition}\label{DefTCP}
Let $(\mathcal{S},\mathcal{T})$ and $(\mathcal{U},\mathcal{V})$ be cotorsion pairs on $\mathscr{C}$. Then the pair $\mathcal{P}=((\mathcal{S},\mathcal{T}),(\mathcal{U},\mathcal{V}))$ is called a {\it twin cotorsion pair} if it satisfies $\mathbb{E}(\mathcal{S},\mathcal{V})=0$. (Pairs of cotorsion pairs are considered in abelian/exact categories in \cite{Ho1,Ho2}, and in triangulated categories in \cite{Na2,Na3}.)

If moreover it satisfies $\mathcal{S}\cap\mathcal{T}=\mathcal{U}\cap\mathcal{V}(=\mathcal{I})$, then $\mathcal{P}$ is called a {\it concentric} twin cotorsion pair similarly as in the triangulated case \cite{Na4}. In this case, we put $\mathcal{Z}=\mathcal{T}\cap\mathcal{U}$.
\end{definition}

\begin{remark}\label{RemTCP}
Let $\mathcal{P}=((\mathcal{S},\mathcal{T}),(\mathcal{U},\mathcal{V}))$ be a concentric twin cotorsion pair. By Corollary~\ref{CorUC}, the inclusion $\mathcal{U}/\mathcal{I}\hookrightarrow\mathscr{C}/\mathcal{I}$ has a right adjoint $\omega_{\mathcal{U}}$. Dually the inclusion $\mathcal{T}/\mathcal{I}\hookrightarrow\mathscr{C}/\mathcal{I}$ has a left adjoint $\sigma_{\mathcal{T}}$.

These restrict to yield the following.
\begin{itemize}
\item[-] Left adjoint $\sigma$ of the inclusion $\mathcal{Z}/\mathcal{I}\hookrightarrow\mathcal{U}/\mathcal{I}$.
\item[-] Right adjoint $\omega$ of the inclusion $\mathcal{Z}/\mathcal{I}\hookrightarrow\mathcal{T}/\mathcal{I}$.
\end{itemize}
\end{remark}

\begin{definition}\label{DefNiNf}
Let $\mathcal{P}=((\mathcal{S},\mathcal{T}),(\mathcal{U},\mathcal{V}))$ be a twin cotorsion pair on $\mathscr{C}$. We define full subcategories $\mathcal{N}^i,\mathcal{N}^f$of $\mathscr{C}$ as follows.
\begin{enumerate}
\item[(1)] $\mathcal{N}^i=\mathrm{Cone}(\mathcal{V},\mathcal{S})$.
\item[(2)] $\mathcal{N}^f=\mathrm{CoCone}(\mathcal{V},\mathcal{S})$.
\end{enumerate}
\end{definition}

\begin{remark}
The notation $\mathcal{N}^i$, $\mathcal{N}^f$ is motivated by Section~\ref{section: model}. Morally, an object $X$ belongs to $\mathcal{N}^i$ if and only if the morphism $0\rightarrow X$ from the initial object is a weak equivalence. For a more precise statement, see Proposition~\ref{LemRelationNandW}.
\end{remark}

\begin{remark}
If $\mathcal{P}$ is concentric, then for any $C\in\mathscr{C}$, we have
\begin{enumerate}
\item[(1)] $C\in\mathcal{N}^i\Leftrightarrow\omega_{\mathcal{U}}(C)\in\mathcal{S}/\mathcal{I}$,
\item[(2)] $C\in\mathcal{N}^f\Leftrightarrow\sigma_{\mathcal{T}}(C)\in\mathcal{V}/\mathcal{I}$.
\end{enumerate}
\end{remark}

\begin{remark}\label{RemNiNf}
Let $\mathcal{P}=((\mathcal{S},\mathcal{T}),(\mathcal{U},\mathcal{V}))$ be a twin cotorsion pair. Then the following holds.
\begin{enumerate}
\item[(1)] $\mathcal{S}\subseteq\mathcal{N}^i$, $\mathcal{V}\subseteq\mathcal{N}^f$.
\item[(2)] $\mathcal{U}\cap\mathcal{N}^i=\mathcal{S}$, $\mathcal{T}\cap\mathcal{N}^f=\mathcal{V}$.
\item[(3)] If $\mathcal{P}$ is concentric, $\mathcal{S}\subseteq\mathcal{N}^f$, $\mathcal{V}\subseteq\mathcal{N}^i$.
\end{enumerate}
\end{remark}

\begin{lemma}\label{LemConcentric1}
Let $\mathcal{P}=((\mathcal{S},\mathcal{T}),(\mathcal{U},\mathcal{V}))$ be a concentric twin cotorsion pair. Then the following holds.
\begin{enumerate}
\item[(1)] $\mathrm{Cone}(\mathcal{V},\mathcal{N}^i)\subseteq\mathcal{N}^i$.
\item[(2)] $\mathrm{CoCone}(\mathcal{N}^f,\mathcal{S})\subseteq\mathcal{N}^f$.
\end{enumerate}
\end{lemma}
\begin{proof}
{\rm (1)} By definition, $C\in\mathrm{Cone}(\mathcal{V},\mathcal{N}^i)$ admits a conflation
\[ V\to N\to C\quad(V\in\mathcal{V},N\in\mathcal{N}^i). \]
Resolve $N$ by a conflation
\[ V^N\to S^N\to N\quad(S^N\in\mathcal{S},V^N\in\mathcal{V}). \]
By {\rm (ET4)$^{\mathrm{op}}$}, we obtain a commutative diagram in $\mathscr{C}$
\[
\xy
(-6,18)*+{V^N}="-12";
(6,18)*+{V^N}="-14";
(-6,6)*+{{}^{\exists}E}="2";
(6,6)*+{S^N}="4";
(18,6)*+{C}="6";
(-6,-6)*+{V}="12";
(6,-6)*+{N}="14";
(18,-6)*+{C}="16";
{\ar@{=} "-12";"-14"};
{\ar "-12";"2"};
{\ar "-14";"4"};
{\ar "2";"4"};
{\ar "4";"6"};
{\ar "2";"12"};
{\ar "4";"14"};
{\ar@{=} "6";"16"};
{\ar "12";"14"};
{\ar "14";"16"};
{\ar@{}|\circlearrowright "-12";"4"};
{\ar@{}|\circlearrowright "2";"14"};
{\ar@{}|\circlearrowright "4";"16"};
\endxy
\]
in which $V^N\to E\to V$ and $E\to S^N\to C$ are conflations. Since $\mathcal{V}\subseteq\mathscr{C}$ is extension-closed, it follows that $E\in\mathcal{V}$. This means $C\in\mathcal{N}^i$. {\rm (2)} is dual to {\rm (1)}.
\end{proof}

\begin{lemma}\label{LemConcentric2}
Let $\mathcal{P}=((\mathcal{S},\mathcal{T}),(\mathcal{U},\mathcal{V}))$ be a concentric twin cotorsion pair. Let $U\in\mathcal{U}$ be any object. Assume there is a conflation
\[ M\to U\to S \]
satisfying $M\in\mathcal{N}^f$ and $S\in\mathcal{S}$. Then $U$ belongs to $\mathcal{N}^f$.

Dually, if $T\in\mathcal{T}$ appears in a conflation $V\to T\to N$ satisfying $V\in\mathcal{V}, N\in\mathcal{N}^i$, then $T$ belongs to $\mathcal{N}^i$.
\end{lemma}
\begin{proof}
By definition, $M$ admits a conflation
\[ M\to V_M\to S_M\quad (V_M\in\mathcal{V},S_M\in\mathcal{S}). \]
By Proposition~\ref{PropBaer} {\rm (2)}, we obtain a commutative diagram in $\mathscr{C}$
\[
\xy
(-6,6)*+{M}="2";
(6,6)*+{U}="4";
(18,6)*+{S}="6";
(-6,-6)*+{V_M}="12";
(6,-6)*+{{}^{\exists}X}="14";
(18,-6)*+{S}="16";
(-6,-18)*+{S_M}="22";
(6,-18)*+{S_M}="24";
{\ar^{} "2";"4"};
{\ar^{} "4";"6"};
{\ar_{} "2";"12"};
{\ar^{} "4";"14"};
{\ar@{=} "6";"16"};
{\ar^{} "12";"14"};
{\ar^{} "14";"16"};
{\ar_{} "12";"22"};
{\ar^{} "14";"24"};
{\ar@{=} "22";"24"};
{\ar@{}|\circlearrowright "2";"14"};
{\ar@{}|\circlearrowright "4";"16"};
{\ar@{}|\circlearrowright "12";"24"};
\endxy
\]
consisting of conflations. Since $\mathcal{U}$ is extension-closed, it follows that $X\in\mathcal{U}$. Since $\mathbb{E}(S,V_M)=0$, the $\mathbb{E}$-extension realized by $V_M\to X\to S$ splits. Especially $V_M$ is a direct summand of $X$, and thus it follows that $V_M\in\mathcal{U}\cap\mathcal{V}=\mathcal{I}$. By the extension-closedness of $\mathcal{S}$, we obtain $X\in\mathcal{S}$. Thus $U\in\mathcal{N}^f$ follows from Remark~\ref{RemNiNf} {\rm (3)} and Lemma~\ref{LemConcentric1} {\rm (2)}.
\end{proof}

\begin{lemma}\label{LemX}
Let $\mathcal{P}$ be as in Lemma~\ref{LemConcentric2}. Let $T\in\mathcal{T}, M\in\mathcal{N}^f$ be any pair of objects.
If there is a section $T\to M$
 or a retraction $M\to T$, then $T$ belongs to $\mathcal{V}$.
\end{lemma}
\begin{proof}
In either case, we have morphisms $s\in\mathscr{C}(T,M)$ and $r\in\mathscr{C}(M,T)$ satisfying $r\circ s=\mathrm{id}$. By definition, $M$ admits a conflation
\[ M\overset{v}{\longrightarrow}V\to S. \]
By $\mathbb{E}(S,T)=0$, the morphism $r$ factors through $v$. Then $v\circ s\in\mathscr{C}(T,V)$ becomes a section, and thus it follows from Corollary~\ref{CorSectionClosed} that $T\in\mathcal{V}$.
\end{proof}

\section{Bijective correspondence with model structures}\label{section: model}

In the rest, let $(\mathscr{C},\mathbb{E},\mathfrak{s})$ be an extriangulated category.

In this section, we give a bijective correspondence between {\it Hovey twin cotorsion pairs} and {\it admissible model structures} which we will soon define. This gives a unification of the following preceding works.
\begin{itemize}
\item For an abelian category, Hovey has shown their correspondence in \cite{Ho1,Ho2} ({\it abelian model structure}). This has been generalized to an exact category by Gillespie \cite{G} ({\it exact model structure}), and investigated by \v{S}\v{t}ov\'{\i}\v{c}ek \cite{S}.
\item For a triangulated category, Yang \cite{Y} has introduced an analogous notion of {\it triangulated model structure} and showed its correspondence with cotorsion pairs.
\end{itemize}

\subsection{Hovey twin cotorsion pair}
We recall that $\mathcal{N}^i$ (resp. $\mathcal{N}^f$) is the collection of all objects $X\in\mathscr{C}$ which are part of a conflation $V\longrightarrow S\longrightarrow X$ (resp. $X\longrightarrow V\longrightarrow S$), for some $V\in\mathcal{V}$ and $S\in \mathcal{S}$.

\begin{definition}\label{DefHTCP}
Let $\mathcal{P}=((\mathcal{S},\mathcal{T}),(\mathcal{U},\mathcal{V}))$ be a twin cotorsion pair. We call $\mathcal{P}$ a {\it Hovey twin cotorsion pair} if it satisfies $\mathcal{N}^f=\mathcal{N}^i$. We denote this subcategory by $\mathcal{N}$.
\end{definition}

\begin{remark}\label{RemHTCPConcentric}
Any Hovey twin cotorsion pair is concentric. In fact, we have $\mathcal{U}\cap\mathcal{V}=\mathcal{U}\cap(\mathcal{N}^f\cap\mathcal{T})=(\mathcal{U}\cap\mathcal{N}^i)\cap\mathcal{T}=\mathcal{S}\cap\mathcal{T}$ by Remark~\ref{RemNiNf} {\rm (2)}.
\end{remark}

For any Hovey twin cotorsion pair, the subcategory $\mathcal{N}\subseteq\mathscr{C}$ is extension-closed. More strongly, it satisfies the following.
\begin{proposition}\label{PropThick}
Let $\mathcal{P}=((\mathcal{S},\mathcal{T}),(\mathcal{U},\mathcal{V}))$ be a Hovey twin cotorsion pair. For any conflation $A\overset{x}{\longrightarrow}B\overset{y}{\longrightarrow}C$, if two out of $A,B,C$ belong to $\mathcal{N}$, then so does the third.
Namely, we have the following.
\begin{enumerate}
\item[(1)] $A,C\in\mathcal{N}\Rightarrow B\in\mathcal{N}$.
\item[(2)] $A,B\in\mathcal{N}\Rightarrow C\in\mathcal{N}$.
\item[(3)] $B,C\in\mathcal{N}\Rightarrow A\in\mathcal{N}$.
\end{enumerate}
\end{proposition}
\begin{proof}
{\rm (1)} Resolve $A,C$ by conflations
\[ A\to V_A\to S_A\ \ \text{and}\ \ V^C\to S^C\to C\quad(S_A,S^C\in\mathcal{S},\, V_A,V^C\in\mathcal{V}) \]
respectively. By Proposition~\ref{PropBaer} {\rm (1),(2)}, we obtain commutative diagrams
\[
\xy
(-6,12)*+{V^C}="-12";
(6,12)*+{V^C}="-14";
(-18,0)*+{A}="0";
(-6,0)*+{{}^{\exists}X}="2";
(6,0)*+{S^C}="4";
(-18,-12)*+{A}="10";
(-6,-12)*+{B}="12";
(6,-12)*+{C}="14";
{\ar@{=} "-12";"-14"};
{\ar_{} "-12";"2"};
{\ar^{} "-14";"4"};
{\ar^{} "0";"2"};
{\ar^{} "2";"4"};
{\ar@{=} "0";"10"};
{\ar_{} "2";"12"};
{\ar^{} "4";"14"};
{\ar_{} "10";"12"};
{\ar_{} "12";"14"};
{\ar@{}|\circlearrowright "-12";"4"};
{\ar@{}|\circlearrowright "0";"12"};
{\ar@{}|\circlearrowright "2";"14"};
\endxy\ \ ,\ \ 
\xy
(-6,12)*+{A}="2";
(6,12)*+{X}="4";
(18,12)*+{S^C}="6";
(-6,0)*+{V_A}="12";
(6,0)*+{{}^{\exists}Y}="14";
(18,0)*+{S^C}="16";
(-6,-12)*+{S_A}="22";
(6,-12)*+{S_A}="24";
{\ar^{} "2";"4"};
{\ar^{} "4";"6"};
{\ar_{} "2";"12"};
{\ar^{} "4";"14"};
{\ar@{=} "6";"16"};
{\ar^{} "12";"14"};
{\ar^{} "14";"16"};
{\ar_{} "12";"22"};
{\ar^{} "14";"24"};
{\ar@{=} "22";"24"};
{\ar@{}|\circlearrowright "2";"14"};
{\ar@{}|\circlearrowright "4";"16"};
{\ar@{}|\circlearrowright "12";"24"};
\endxy
\]
which are made of conflations. Since $\mathbb{E}(S^C,V_A)=0$, the conflation $V_A\to Y\to S^C$ realizes the split $\mathbb{E}$-extension. It follows that $Y\cong V_A\oplus S^C\in\mathcal{N}$. We obtain $X\in\mathcal{N}$ by Lemma~\ref{LemConcentric1} {\rm (2)}, and thus $B\in\mathcal{N}$ by Lemma~\ref{LemConcentric1} {\rm (1)}.


{\rm (2)} Resolve $A$ by a conflation
\[ A\to V_A\to S_A\quad(S_A\in\mathcal{S},\, V_A\in\mathcal{V}). \]
Then by Proposition~\ref{PropBaer} {\rm (2)}, we have a commutative diagram
\[
\xy
(-6,12)*+{A}="2";
(6,12)*+{B}="4";
(18,12)*+{C}="6";
(-6,0)*+{V_A}="12";
(6,0)*+{{}^{\exists}G}="14";
(18,0)*+{C}="16";
(-6,-12)*+{S_A}="22";
(6,-12)*+{S_A}="24";
{\ar^{} "2";"4"};
{\ar^{} "4";"6"};
{\ar_{} "2";"12"};
{\ar^{} "4";"14"};
{\ar@{=} "6";"16"};
{\ar^{} "12";"14"};
{\ar^{} "14";"16"};
{\ar_{} "12";"22"};
{\ar^{} "14";"24"};
{\ar@{=} "22";"24"};
{\ar@{}|\circlearrowright "2";"14"};
{\ar@{}|\circlearrowright "4";"16"};
{\ar@{}|\circlearrowright "12";"24"};
\endxy
\]
made of conflations. Since $B,S_A\in\mathcal{N}$, we have $G\in\mathcal{N}$ by {\rm (1)}. From Lemma~\ref{LemConcentric1} {\rm (1)}, it follows that $C\in\mathcal{N}$.

{\rm (3)} is dual to {\rm (2)}.
\end{proof}

\begin{remark}
As a corollary, $\mathcal{N}$ becomes an extriangulated category by Remark~\ref{RemETrExtClosed}. Almost by definition of $\mathcal{N}$, the pair $(\mathcal{S},\mathcal{V})$ gives a cotorsion pair on $\mathcal{N}$.
\end{remark}

\subsection{From admissible model structure to Hovey twin cotorsion pair}
Throughout this section, let $\mathcal{M}=(\mathit{Fib},\mathit{Cof},\mathbb{W})$ be a model structure on $\mathscr{C}$, where $\mathit{Fib},\mathit{Cof},\mathbb{W}$ are the classes of fibrations, cofibrations, and weak equivalences. Let $w\mathit{Fib}=\mathit{Fib}\cap\mathbb{W}$ and $w\mathit{Cof}=\mathit{Cof}\cap\mathbb{W}$ denote the classes of acyclic fibrations and acyclic cofibrations, respectively. Associate full subcategories $\mathcal{S},\mathcal{T},\mathcal{U},\mathcal{V}\subseteq\mathscr{C}$ as follows.
\begin{eqnarray*}
&C\in\mathcal{S}\Leftrightarrow(0\to C)\in w\mathit{Cof},&\\
&C\in\mathcal{T}\Leftrightarrow(C\to0)\in\mathit{Fib},&\\
&C\in\mathcal{U}\Leftrightarrow(0\to C)\in\mathit{Cof},&\\
&C\in\mathcal{V}\Leftrightarrow(C\to0)\in w\mathit{Fib}.&
\end{eqnarray*}
Remark that these are full additive subcategories of $\mathscr{C}$, closed under isomorphisms and direct summands. In particular, the definition below makes sense.

\begin{definition}\label{DefExactModel}
$\mathcal{M}$ is called an {\it admissible model structure} if it satisfies the following conditions for any morphism $f\in\mathscr{C}(A,B)$.
\begin{enumerate}
\item[(1)] $f\in w\mathit{Cof}$ if and only if it is an inflation with $\mathrm{Cone}(f)\in\mathcal{S}$.
\item[(2)] $f\in\mathit{Fib}$ if and only if it is a deflation with $\mathrm{CoCone}(f)\in\mathcal{T}$.
\item[(3)] $f\in\mathit{Cof}$ if and only if it is an inflation with $\mathrm{Cone}(f)\in\mathcal{U}$.
\item[(4)] $f\in w\mathit{Fib}$ if and only if it is a deflation with $\mathrm{CoCone}(f)\in\mathcal{V}$.
\end{enumerate}
\end{definition}

We note that the model structures which might appear in \cite{Pal} are not admissible.

\begin{proposition}\label{PropFromModelToTCP}
Let $\mathcal{M}$ be an admissible model structure. Then $\mathcal{P}=((\mathcal{S},\mathcal{T}),(\mathcal{U},\mathcal{V}))$ is a twin cotorsion pair on $(\mathscr{C},\mathbb{E},\mathfrak{s})$.
\end{proposition}
\begin{proof}
$\mathcal{S}\subseteq\mathcal{U}$ is obvious from the definition.
Since a similar argument works for $(\mathcal{S},\mathcal{T})$, we show that $(\mathcal{U},\mathcal{V})$ is a cotorsion pair. Let us confirm the conditions {\rm (1),(2),(3)} in Definition~\ref{DefCotors}.

{\rm (1)} Let $U\in\mathcal{U},V\in\mathcal{V}$ be any pair of objects, and let $\delta\in\mathbb{E}(U,V)$ be any element. Realize it as an $\mathbb{E}$-triangle
\[ V\overset{v}{\longrightarrow}B\overset{u}{\longrightarrow}U\overset{\delta}{\dashrightarrow}. \]
Since $U\in\mathcal{U}$ and $u\in w\mathit{Fib}$, there exists a section $s\in\mathscr{C}(U,B)$ of $u$. Thus $\delta$ splits by Corollary~\ref{CorExact0}.

{\rm (2)} Let $C\in\mathscr{C}$ be any object. Factorize the zero morphism $0\colon 0\to C$ as follows.
\[
\xy
(-10,6)*+{0}="0";
(10,6)*+{C}="2";
(0,-8)*+{D}="4";
(0,10)*+{}="5";
{\ar^{} "0";"2"};
{\ar_(0.4){i} "0";"4"};
{\ar@{<-}^(0.4){u} "2";"4"};
{\ar@{}|\circlearrowright "4";"5"};
\endxy
\quad\begin{array}{c}i\in\mathit{Cof},\\ u\in w\mathit{Fib}\end{array}
\]
Since $\mathcal{M}$ is admissible, we have conflations
\[ 0\overset{i}{\longrightarrow}D\overset{j}{\longrightarrow}U,\ \ \text{and}\ \ V\to D\overset{u}{\longrightarrow}C \]
with $U\in\mathcal{U},V\in\mathcal{V}$. This shows that $j$ is an isomorphism, and thus we obtain a conflation $V\to U\to C$.

{\rm (3)} is dual to {\rm (2)}.
\end{proof}

\begin{proposition}\label{LemRelationNandW}
Let $\mathcal{M}$ be an admissible model structure as above. Then the associated twin cotorsion pair $\mathcal{P}$ obtained in Proposition~\ref{PropFromModelToTCP} is a Hovey twin cotorsion pair.

Indeed, if we let $\mathcal{N}^i,\mathcal{N}^f\subseteq\mathscr{C}$ be as in Definition~\ref{DefNiNf}, then the following are equivalent for any object $N\in\mathscr{C}$.
\begin{enumerate}
\item[(1)] $N\in\mathcal{N}^i$.
\item[(2)] $(0\to N)\in\mathbb{W}$.
\item[(3)] $(N\to 0)\in\mathbb{W}$.
\item[(4)] $N\in\mathcal{N}^f$.
\end{enumerate}
\end{proposition}
\begin{proof}
$(1)\Rightarrow (2)$ If $N\in\mathcal{N}^i$, there is a conflation
\[ V\to S\overset{s}{\longrightarrow}N\quad(V\in\mathcal{V},S\in\mathcal{S}) \]
by definition. Thus $0\to N$ can be factorized as follows.
\[
\xy
(-10,6)*+{0}="0";
(10,6)*+{N}="2";
(0,-8)*+{S}="4";
(0,10)*+{}="5";
{\ar^{0} "0";"2"};
{\ar_(0.4){} "0";"4"};
{\ar@{<-}^(0.4){s} "2";"4"};
{\ar@{}|\circlearrowright "4";"5"};
\endxy\quad\begin{array}{c}s\in w\mathit{Fib},\\ (0\to S)\in w\mathit{Cof}\end{array}
\]
It follows that $(0\to N)\in w\mathit{Fib}\circ w\mathit{Cof}=\mathbb{W}$.

$(2)\Rightarrow (1)$ Factorize $0\to N$ as follows.
\[
\xy
(-10,6)*+{0}="0";
(10,6)*+{N}="2";
(0,-8)*+{D}="4";
(0,10)*+{}="5";
{\ar^{0} "0";"2"};
{\ar_(0.4){i} "0";"4"};
{\ar@{<-}^(0.4){u} "2";"4"};
{\ar@{}|\circlearrowright "4";"5"};
\endxy\quad\begin{array}{c}i\in w\mathit{Cof},\\ u\in w\mathit{Fib}\end{array}
\]
A similar argument as in the proof {\rm (2)} of Proposition~\ref{PropFromModelToTCP} gives a conflation $V\to S\to N$.

$(2)\Leftrightarrow(3)$ follows from the 2-out-of-3 property of $\mathbb{W}$.

$(3)\Leftrightarrow(4)$ is dual to $(1)\Leftrightarrow(2)$.
\end{proof}

\subsection{From Hovey twin cotorsion pair to admissible model structure}
Throughout this section, let $\mathcal{P}=((\mathcal{S},\mathcal{T}),(\mathcal{U},\mathcal{V}))$ be a Hovey twin cotorsion pair on $(\mathscr{C},\mathbb{E},\mathfrak{s})$. In addition, we assume the following condition, analogous to the weak idempotent completeness (\cite[Proposition~7.6]{Bu}).
\begin{condition}[WIC]
Let $(\mathscr{C},\mathbb{E},\mathfrak{s})$ be an extriangulated category. Consider the following conditions.
\begin{enumerate}
\item[(1)] Let $f\in\mathscr{C}(A,B),\, g\in\mathscr{C}(B,C)$ be any composable pair of morphisms. If $g\circ f$ is an inflation, then so is $f$.
\item[(2)] Let $f\in\mathscr{C}(A,B),\, g\in\mathscr{C}(B,C)$ be any composable pair of morphisms. If $g\circ f$ is a deflation, then so is $g$.
\end{enumerate}
\end{condition}

With the assumption of Condition~(WIC), we have the following analog of the nine lemma.
\begin{lemma}\label{LemNine}
Assume $(\mathscr{C},\mathbb{E},\mathfrak{s})$ is an extriangulated category satisfying Condition~(WIC). Let
\[
\xy
(-7,21)*+{K}="-12";
(7,21)*+{K^{\prime}}="-14";
(-7,7)*+{A}="2";
(7,7)*+{B}="4";
(21,7)*+{C}="6";
(35,7)*+{}="8";
(-7,-7)*+{A^{\prime}}="12";
(7,-7)*+{B^{\prime}}="14";
(21,-7)*+{C^{\prime}}="16";
(35,-7)*+{}="18";
(-7,-21)*+{}="22";
(7,-21)*+{}="24";
{\ar_{k} "-12";"2"};
{\ar^{k^{\prime}} "-14";"4"};
{\ar^{x} "2";"4"};
{\ar^{y} "4";"6"};
{\ar@{-->}^{\delta} "6";"8"};
{\ar_{a} "2";"12"};
{\ar^{b} "4";"14"};
{\ar_{x^{\prime}} "12";"14"};
{\ar_{y^{\prime}} "14";"16"};
{\ar@{-->}^{\delta^{\prime}} "16";"18"};
{\ar@{-->}_{\kappa} "12";"22"};
{\ar@{-->}^{\kappa^{\prime}} "14";"24"};
{\ar@{}|\circlearrowright "2";"14"};
\endxy
\]
be a diagram made of $\mathbb{E}$-triangles.
Then for some $X\in\mathscr{C}$, we obtain $\mathbb{E}$-triangles
\[ K\overset{m}{\longrightarrow}K^{\prime}\overset{n}{\longrightarrow}X\overset{\nu}{\dashrightarrow}\ \ \text{and}\ \ X\overset{i}{\longrightarrow}C\overset{c}{\longrightarrow}C^{\prime}\overset{\tau}{\dashrightarrow} \]
which make the following diagram commutative,
\begin{equation}\label{NineToShow}
\xy
(-7,21)*+{K}="-12";
(7,21)*+{K^{\prime}}="-14";
(21,21)*+{X}="-16";
(35,21)*+{}="-18";
(-7,7)*+{A}="2";
(7,7)*+{B}="4";
(21,7)*+{C}="6";
(35,7)*+{}="8";
(-7,-7)*+{A^{\prime}}="12";
(7,-7)*+{B^{\prime}}="14";
(21,-7)*+{C^{\prime}}="16";
(35,-7)*+{}="18";
(-7,-21)*+{}="22";
(7,-21)*+{}="24";
(21,-21)*+{}="26";
{\ar^{m} "-12";"-14"};
{\ar^{n} "-14";"-16"};
{\ar@{-->}^{\nu} "-16";"-18"};
{\ar_{k} "-12";"2"};
{\ar^{k^{\prime}} "-14";"4"};
{\ar^{i} "-16";"6"};
{\ar^{x} "2";"4"};
{\ar^{y} "4";"6"};
{\ar@{-->}^{\delta} "6";"8"};
{\ar_{a} "2";"12"};
{\ar^{b} "4";"14"};
{\ar^{c} "6";"16"};
{\ar_{x^{\prime}} "12";"14"};
{\ar_{y^{\prime}} "14";"16"};
{\ar@{-->}^{\delta^{\prime}} "16";"18"};
{\ar@{-->}_{\kappa} "12";"22"};
{\ar@{-->}^{\kappa^{\prime}} "14";"24"};
{\ar@{-->}^{\tau} "16";"26"};
{\ar@{}|\circlearrowright "-12";"4"};
{\ar@{}|\circlearrowright "-14";"6"};
{\ar@{}|\circlearrowright "2";"14"};
{\ar@{}|\circlearrowright "4";"16"};
\endxy
\end{equation}
in which, those $(k,k^{\prime},i),\, (a,b,c),\, (m,x,x^{\prime}),\, (n,y,y^{\prime})$ are morphisms of $\mathbb{E}$-triangles.
\end{lemma}
\begin{proof}
By {\rm (ET4)$^{\mathrm{op}}$}, we obtain an $\mathbb{E}$-triangle
\[ E\overset{f}{\longrightarrow}B^{\prime}\overset{y^{\prime}\circ b}{\longrightarrow}C^{\prime}\overset{\theta}{\dashrightarrow} \]
and a commutative diagram
\[
\xy
(-21,7)*+{K^{\prime}}="0";
(-7,7)*+{E}="2";
(7,7)*+{A^{\prime}}="4";
(-21,-7)*+{K^{\prime}}="10";
(-7,-7)*+{B}="12";
(7,-7)*+{B^{\prime}}="14";
(-7,-21)*+{C^{\prime}}="22";
(7,-21)*+{C^{\prime}}="24";
{\ar^{d} "0";"2"};
{\ar^{e} "2";"4"};
{\ar@{=} "0";"10"};
{\ar_{f} "2";"12"};
{\ar^{x^{\prime}} "4";"14"};
{\ar_{k^{\prime}} "10";"12"};
{\ar_{b} "12";"14"};
{\ar_{y^{\prime}\circ b} "12";"22"};
{\ar^{y^{\prime}} "14";"24"};
{\ar@{=} "22";"24"};
{\ar@{}|\circlearrowright "0";"12"};
{\ar@{}|\circlearrowright "2";"14"};
{\ar@{}|\circlearrowright "12";"24"};
\endxy
\]
in $\mathscr{C}$, satisfying the following compatibilities.
\begin{itemize}
\item[{\rm (i)}] $K^{\prime}\overset{d}{\longrightarrow}E\overset{e}{\longrightarrow}A^{\prime}\overset{x^{\prime\ast}\kappa^{\prime}}{\dashrightarrow}$ is an $\mathbb{E}$-triangle,
\item[{\rm (ii)}] $\delta^{\prime}=e_{\ast}\theta$,
\item[{\rm (iii)}] $d_{\ast}\kappa^{\prime}=y^{\prime\ast}\theta$.
\end{itemize}
By the dual of Lemma~\ref{LemHomotPush}, the upper-right square
\[
\xy
(-6,6)*+{E}="0";
(6,6)*+{B}="2";
(-6,-6)*+{A^{\prime}}="4";
(6,-6)*+{B^{\prime}}="6";
{\ar^{f} "0";"2"};
{\ar_{e} "0";"4"};
{\ar^{b} "2";"6"};
{\ar_{x^{\prime}} "4";"6"};
{\ar@{}|\circlearrowright "0";"6"};
\endxy
\]
is a weak pullback. Thus there exists a morphism $g\in\mathscr{C}(A,E)$ which makes the following diagram commutative.
\[
\xy
(-18,18)*+{A}="-2";
(-7,7)*+{E}="0";
(7,7)*+{B}="2";
(-7,-7)*+{A^{\prime}}="4";
(7,-7)*+{B^{\prime}}="6";
(-3,15)*+{}="5";
(-15,3)*+{}="7";
{\ar_{f} "0";"2"};
{\ar^{e} "0";"4"};
{\ar^{b} "2";"6"};
{\ar_{x^{\prime}} "4";"6"};
{\ar_{g} "-2";"0"};
{\ar@/^0.60pc/^{x} "-2";"2"};
{\ar@/_0.60pc/_{a} "-2";"4"};
{\ar@{}|\circlearrowright "0";"6"};
{\ar@{}|\circlearrowright "0";"5"};
{\ar@{}|\circlearrowright "0";"7"};
\endxy
\]
By Condition~(WIC), this $g$ becomes an inflation. Complete it into an $\mathbb{E}$-triangle
\[ A\overset{g}{\longrightarrow}E\overset{h}{\longrightarrow}X\overset{\mu}{\dashrightarrow}. \]
By Lemma~\ref{LemOctaRigid}, we obtain a commutative diagram
\[
\xy
(-21,7)*+{A}="0";
(-7,7)*+{E}="2";
(7,7)*+{X}="4";
(-21,-7)*+{A}="10";
(-7,-7)*+{B}="12";
(7,-7)*+{C}="14";
(-7,-21)*+{C^{\prime}}="22";
(7,-21)*+{C^{\prime}}="24";
{\ar^{g} "0";"2"};
{\ar^{h} "2";"4"};
{\ar@{=} "0";"10"};
{\ar_{f} "2";"12"};
{\ar^{i} "4";"14"};
{\ar_{x} "10";"12"};
{\ar_{y} "12";"14"};
{\ar_{y^{\prime}\circ b} "12";"22"};
{\ar^{c} "14";"24"};
{\ar@{=} "22";"24"};
{\ar@{}|\circlearrowright "0";"12"};
{\ar@{}|\circlearrowright "2";"14"};
{\ar@{}|\circlearrowright "12";"24"};
\endxy
\]
made of conflations, which satisfy
\begin{itemize}
\item[{\rm (iv)}] $X\overset{i}{\longrightarrow}C\overset{c}{\longrightarrow}C^{\prime}\overset{h_{\ast}\theta}{\dashrightarrow}$ is an $\mathbb{E}$-triangle,
\item[{\rm (v)}] $\mu=i^{\ast}\delta$,
\item[{\rm (vi)}] $g_{\ast}\delta=c^{\ast}\theta$.
\end{itemize}
By Proposition~\ref{PropShiftOctahedron}, we obtain an $\mathbb{E}$-triangle
\[ K\overset{m}{\longrightarrow}K^{\prime}\overset{n}{\longrightarrow}X\overset{\nu}{\longrightarrow} \]
which makes the diagram
\[
\xy
(-7,7)*+{K}="2";
(7,7)*+{A}="4";
(21,7)*+{A^{\prime}}="6";
(-7,-7)*+{K^{\prime}}="12";
(7,-7)*+{E}="14";
(21,-7)*+{A^{\prime}}="16";
(-7,-21)*+{X}="22";
(7,-21)*+{X}="24";
{\ar^{k} "2";"4"};
{\ar^{a} "4";"6"};
{\ar_{m} "2";"12"};
{\ar^{g} "4";"14"};
{\ar@{=} "6";"16"};
{\ar^{d} "12";"14"};
{\ar^{e} "14";"16"};
{\ar_{n} "12";"22"};
{\ar^{h} "14";"24"};
{\ar@{=} "22";"24"};
{\ar@{}|\circlearrowright "2";"14"};
{\ar@{}|\circlearrowright "4";"16"};
{\ar@{}|\circlearrowright "12";"24"};
\endxy
\]
commutative in $\mathscr{C}$, and satisfies
\begin{itemize}
\item[{\rm (vii)}] $m_{\ast}\kappa=x^{\prime\ast}\kappa^{\prime}$,
\item[{\rm (viii)}] $\mu=k_{\ast}\nu$,
\item[{\rm (ix)}] $h^{\ast}\nu+e^{\ast}\kappa=0$.
\end{itemize}
Put $\tau=h_{\ast}\theta$. It is straightforward to show that the diagram $(\ref{NineToShow})$ is indeed commutative.
Moreover,
\begin{itemize}
\item[-] {\rm (v)} and {\rm (viii)} show $k_{\ast}\nu=i^{\ast}\delta$,
\item[-] {\rm (ii)} and {\rm (vi)} show $a_{\ast}\delta=c^{\ast}\delta^{\prime}$,
\item[-] {\rm (vii)} shows $m_{\ast}\kappa=x^{\prime\ast}\kappa^{\prime}$,
\item[-] {\rm (iii)} shows $n_{\ast}\kappa^{\prime}=y^{\prime\ast}\tau$.
\end{itemize}
\end{proof}

\begin{remark}
In the proof of Lemma~\ref{LemNine}, we have obtained an extra compatibility {\rm (ix)}. This can be interpreted by the following analog of $\mathrm{Ext}^2$-group.

Let $A,D\in\mathscr{C}$ be any pair of objects. We denote triplet of $X\in\mathscr{C},\sigma\in\mathbb{E}(D,X),\tau\in\mathbb{E}(X,A)$ by $(\sigma,X,\tau)$. For any pair of such triplets $(\sigma,X,\tau)$ and $(\sigma^{\prime},X^{\prime},\tau^{\prime})$, we write as
\[ (\sigma,X,\tau)\underset{x}{\leadsto}(\sigma^{\prime},X^{\prime},\tau^{\prime})\quad(\text{or simply}\ (\sigma,X,\tau)\leadsto(\sigma^{\prime},X^{\prime},\tau^{\prime})) \]
if and only if there exists $x\in\mathscr{C}(X,X^{\prime})$ satisfying $x_{\ast}\sigma=\sigma^{\prime}$ and $\tau=x^{\ast}\tau^{\prime}$.

Let $\sim$ be the equivalence relation generated by $\leadsto$, and denote the equivalence class of $(\sigma,X,\tau)$ by $\tau\underset{X}{\circ}\sigma$. Let us denote their collection by
\[ \mathbb{E}^2(D,A)=\bigg(\coprod_{X\in\mathscr{C}}\mathbb{E}(D,X)\times\mathbb{E}(X,A)\bigg)/\sim. \]
The proof of Lemma~\ref{LemNine} shows
\[ (\delta^{\prime},A^{\prime},-\kappa)\underset{e}{\,\reflectbox{$\leadsto$}}(\theta,A^{\prime},h^{\ast}\nu)\underset{h}{\leadsto}(\tau,X,\nu) \]
and thus $(-\kappa)\underset{A^{\prime}}{\circ}\delta^{\prime}=\nu\underset{X}{\circ}\tau$ holds in $\mathbb{E}^2(C^{\prime},K)$.
\end{remark}

\begin{definition}\label{DefFibCof}
Define classes of morphisms $\mathit{Fib},w\mathit{Fib},\mathit{Cof},w\mathit{Cof}$ and $\mathbb{W}$ in $\mathscr{C}$ as follows.
\begin{enumerate}
\item[(1)] $f\in\mathit{Fib}$ if it is a deflation with $\mathrm{CoCone}(f)\in\mathcal{T}$.
\item[(2)] $f\in w\mathit{Fib}$ if it is a deflation with $\mathrm{CoCone}(f)\in\mathcal{V}$.
\item[(3)] $f\in\mathit{Cof}$ if it is a inflation with $\mathrm{Cone}(f)\in\mathcal{U}$.
\item[(4)] $f\in w\mathit{Cof}$ if it is a inflation with $\mathrm{Cone}(f)\in\mathcal{S}$.
\item[(5)] $\mathbb{W}=w\mathit{Fib}\circ w\mathit{Cof}$.
\end{enumerate}
\end{definition}

\begin{claim}\label{ClaimNW}
$\ \ $
\begin{enumerate}
\item[(1)] If a conflation $A\overset{f}{\longrightarrow}B\to N$ satisfies $N\in\mathcal{N}$, then $f$ belongs to $\mathbb{W}$.
\item[(2)] If a conflation $N\to A\overset{f}{\longrightarrow}B$ satisfies $N\in\mathcal{N}$, then $f$ belongs to $\mathbb{W}$.
\end{enumerate}
\end{claim}
\begin{proof}
This follows from Proposition~\ref{PropBaer}.
\end{proof}

\begin{proposition}\label{PropComposClosed}
$\mathit{Fib},w\mathit{Fib},\mathit{Cof},w\mathit{Cof}$ are closed under composition.
\end{proposition}
\begin{proof}
For $\mathit{Fib}$, this follows from {\rm (ET4)} and the extension-closedness of $\mathcal{T}$. Similarly for the others.
\end{proof}

\begin{proposition}\label{PropLiftST}
We have the following.
\begin{enumerate}
\item[(1)] $w\mathit{Cof}$ satisfies the left lifting property against $\mathit{Fib}$.
\item[(2)] $w\mathit{Fib}$ satisfies the right lifting property against $\mathit{Cof}$.
\end{enumerate}
\end{proposition}
\begin{proof}
{\rm (1)} Suppose we are given a commutative square
\begin{equation}\label{CommForLift}
\xy
(-6,6)*+{A}="0";
(6,6)*+{C}="2";
(-6,-6)*+{B}="4";
(6,-6)*+{D}="6";
{\ar^{a} "0";"2"};
{\ar_{f} "0";"4"};
{\ar^{g} "2";"6"};
{\ar_{b} "4";"6"};
{\ar@{}|\circlearrowright "0";"6"};
\endxy
\end{equation}
in $\mathscr{C}$, satisfying $f\in w\mathit{Cof}$ and $g\in\mathit{Fib}$. By definition, there are $\mathbb{E}$-triangles 
\begin{eqnarray*}
A\overset{f}{\longrightarrow}B\overset{s}{\longrightarrow}S\overset{\delta}{\dashrightarrow},\\
T\overset{t}{\longrightarrow}C\overset{g}{\longrightarrow}D\overset{\kappa}{\dashrightarrow}.
\end{eqnarray*}
By Corollary~\ref{ExactToShow},
\begin{eqnarray}
&\mathscr{C}(B,T)\overset{\mathscr{C}(f,T)}{\longrightarrow}\mathscr{C}(A,T)\to0\to\mathbb{E}(B,T)\overset{\mathbb{E}(f,T)}{\longrightarrow}\mathbb{E}(A,T),&\label{LiftEx1}\\
&\mathscr{C}(A,T)\overset{\mathscr{C}(A,t)}{\longrightarrow}\mathscr{C}(A,C)\overset{\mathscr{C}(A,g)}{\longrightarrow}\mathscr{C}(A,D),&\label{LiftEx2}\\
&\mathscr{C}(B,C)\overset{\mathscr{C}(B,g)}{\longrightarrow}\mathscr{C}(B,D)\overset{(\kappa_\sharp)_B}{\longrightarrow}\mathbb{E}(B,T)&\label{LiftEx3}
\end{eqnarray}
are exact.

By the commutativity of $(\ref{CommForLift})$, we have
\[ \mathbb{E}(f,T)(b^{\ast}\kappa)=f^{\ast} b^{\ast}\kappa=a^{\ast} g^{\ast}\kappa=0 \]
by Lemma~\ref{LemZero}. Exactness of $(\ref{LiftEx1})$ shows
\[ \kappa_\sharp b=b^{\ast}\kappa=0. \]
Thus by the exactness of $(\ref{LiftEx3})$, there exists $c\in\mathscr{C}(B,C)$ satisfying $g\circ c=b$. Then $a-c\circ f\in\mathscr{C}(A,C)$ satisfies
\[ g\circ(a-c\circ f)=g\circ a-b\circ f=0. \]
By the exactness of $(\ref{LiftEx2})$, there is $c^{\prime}\in\mathscr{C}(A,T)$ satisfying $t\circ c^{\prime}=a-c\circ f$.

By the exactness of $(\ref{LiftEx1})$, there is $c^{\prime\prime}\in\mathscr{C}(A,T)$ satisfying $c^{\prime\prime}\circ f=c^{\prime}$. If we put $h=c+t\circ c^{\prime\prime}\in\mathscr{C}(B,C)$, it satisfies
\[ h\circ f=c\circ f+t\circ c^{\prime\prime}\circ f=c\circ f+t\circ c^{\prime}=a \]
and $g\circ h=g\circ c+g\circ t\circ c^{\prime\prime}=b$.

{\rm (2)} is dual to {\rm (1)}.
\end{proof}

\begin{proposition}\label{PropFactor}
$\mathrm{Mor}(\mathscr{C})=w\mathit{Fib}\circ\mathit{Cof}=\mathit{Fib}\circ w\mathit{Cof}$.
\end{proposition}
\begin{proof}
We only show $\mathrm{Mor}(\mathscr{C})=w\mathit{Fib}\circ\mathit{Cof}$.
Let $f\in\mathscr{C}(A,B)$ be any morphism.
Resolve $A$ by a conflation
\[ A\overset{v_A}{\longrightarrow}V_A\overset{u_A}{\longrightarrow}U_A\quad(U_A\in\mathcal{U},V_A\in\mathcal{V}), \]
and put $f^{\prime}=\Big[\raise1ex\hbox{\leavevmode\vtop{\baselineskip-8ex \lineskip1ex \ialign{#\crcr{$f$}\crcr{$v_A$}\crcr}}}\Big]\colon A\to B\oplus V_A$. By Corollary~\ref{CorTrivWIC}, it admits some conflation
\[ A\overset{f^{\prime}}{\longrightarrow}B\oplus V_A\to C. \]
Resolve $C$ by a conflation
\[ V^C\to U^C\to C\quad(U\in\mathcal{U},V\in\mathcal{V}). \]
Then by Proposition~\ref{PropBaer} {\rm (1)}, we obtain a diagram made of conflations as follows.
\[
\xy
(-8,21)*+{V^C}="-12";
(8,21)*+{V^C}="-14";
(-24,7)*+{A}="0";
(-8,7)*+{M}="2";
(8,7)*+{U^C}="4";
(-24,-7)*+{A}="10";
(-8,-7)*+{B\oplus V_A}="12";
(8,-7)*+{C}="14";
{\ar@{=} "-12";"-14"};
{\ar_{} "-12";"2"};
{\ar^{} "-14";"4"};
{\ar^{m} "0";"2"};
{\ar^{} "2";"4"};
{\ar@{=} "0";"10"};
{\ar_{e} "2";"12"};
{\ar^{} "4";"14"};
{\ar_(0.4){f^{\prime}} "10";"12"};
{\ar_{} "12";"14"};
{\ar@{}|\circlearrowright "-12";"4"};
{\ar@{}|\circlearrowright "0";"12"};
{\ar@{}|\circlearrowright "2";"14"};
\endxy
\]
We have $m\in\mathit{Cof}$. Moreover, for the projection $p_B=[1\ 0]\in\mathscr{C}(B\oplus V_A,B)$, we have $p_B\circ e\in w\mathit{Fib}\circ w\mathit{Fib}=w\mathit{Fib}$. Thus $f=(p_B\circ e)\circ m$ gives the desired factorization.
\end{proof}

\begin{proposition}\label{PropRetract}
$\mathit{Fib},w\mathit{Fib},\mathit{Cof},w\mathit{Cof}$ are closed under retraction.
\end{proposition}
\begin{proof}
We only show the result for $\mathit{Fib}$. Suppose we are given a commutative diagram
\[
\xy
(-14,8)*+{A}="0";
(0,-8)*+{}="1";
(0,6)*+{C}="2";
(0,13)*+{}="3";
(14,8)*+{A}="4";
(-14,-8)*+{B}="10";
(0,-6)*+{D}="12";
(0,-13)*+{}="13";
(14,-8)*+{B}="14";
{\ar_{a} "0";"2"};
{\ar_{c} "2";"4"};
{\ar@/^1.00pc/^{\mathrm{id}} "0";"4"};
{\ar_{f} "0";"10"};
{\ar^{g} "2";"12"};
{\ar^{f} "4";"14"};
{\ar^{b} "10";"12"};
{\ar^{d} "12";"14"};
{\ar@/_1.00pc/_{\mathrm{id}} "10";"14"};
{\ar@{}|\circlearrowright "0";"1"};
{\ar@{}|\circlearrowright "4";"1"};
{\ar@{}|\circlearrowright "2";"3"};
{\ar@{}|\circlearrowright "12";"13"};
\endxy
\]
in $\mathscr{C}$, satisfying $g\in\mathit{Fib}$. By definition, there is an $\mathbb{E}$-triangle
\[ T\overset{t}{\longrightarrow}C\overset{g}{\longrightarrow}D\overset{\theta}{\dashrightarrow}\quad(T\in\mathcal{T}). \]

By Condition~(WIC), $d\circ b=\mathrm{id}$ implies that $d$ is a deflation. Thus $d\circ g$ becomes a deflation by {\rm (ET4)$^{\mathrm{op}}$}. Again by Condition~(WIC), it follows that $f$ is a deflation. Thus there exists an $\mathbb{E}$-triangle $X\overset{x}{\longrightarrow}A\overset{f}{\longrightarrow}B\overset{\delta}{\dashrightarrow}$.
By {\rm (ET3)$^{\mathrm{op}}$}, we obtain the following morphisms of $\mathbb{E}$-triangles.
\[
\xy
(-12,6)*+{X}="0";
(0,6)*+{A}="2";
(12,6)*+{B}="4";
(24,6)*+{}="6";
(-12,-6)*+{T}="10";
(0,-6)*+{C}="12";
(12,-6)*+{D}="14";
(24,-6)*+{}="16";
{\ar^{x} "0";"2"};
{\ar^{f} "2";"4"};
{\ar@{-->}^{\delta} "4";"6"};
{\ar_{k} "0";"10"};
{\ar^{a} "2";"12"};
{\ar^{b} "4";"14"};
{\ar_{t} "10";"12"};
{\ar_{g} "12";"14"};
{\ar@{-->}_{\theta} "14";"16"};
{\ar@{}|\circlearrowright "0";"12"};
{\ar@{}|\circlearrowright "2";"14"};
\endxy\ ,\ \ 
\xy
(-12,6)*+{T}="0";
(0,6)*+{C}="2";
(12,6)*+{D}="4";
(24,6)*+{}="6";
(-12,-6)*+{X}="10";
(0,-6)*+{A}="12";
(12,-6)*+{B}="14";
(24,-6)*+{}="16";
{\ar^{t} "0";"2"};
{\ar^{g} "2";"4"};
{\ar@{-->}^{\theta} "4";"6"};
{\ar_{\ell} "0";"10"};
{\ar^{c} "2";"12"};
{\ar^{d} "4";"14"};
{\ar_{x} "10";"12"};
{\ar_{f} "12";"14"};
{\ar@{-->}_{\delta} "14";"16"};
{\ar@{}|\circlearrowright "0";"12"};
{\ar@{}|\circlearrowright "2";"14"};
\endxy
\]
Composing them, we obtain a morphism
\[
\xy
(-12,6)*+{X}="0";
(0,6)*+{A}="2";
(12,6)*+{B}="4";
(24,6)*+{}="6";
(-12,-6)*+{X}="10";
(0,-6)*+{A}="12";
(12,-6)*+{B}="14";
(24,-6)*+{}="16";
{\ar^{x} "0";"2"};
{\ar^{f} "2";"4"};
{\ar@{-->}^{\delta} "4";"6"};
{\ar_{\ell\circ k} "0";"10"};
{\ar^{\mathrm{id}} "2";"12"};
{\ar^{\mathrm{id}} "4";"14"};
{\ar_{x} "10";"12"};
{\ar_{f} "12";"14"};
{\ar@{-->}_{\delta} "14";"16"};
{\ar@{}|\circlearrowright "0";"12"};
{\ar@{}|\circlearrowright "2";"14"};
\endxy
\]
of $\mathbb{E}$-triangles. By Corollary~\ref{CorExact1}, it follows that $\ell\circ k$ is an isomorphism. Especially $k$ is a section, and thus $X\in\mathcal{T}$. This means that $f$ belongs to $\mathit{Fib}$.
\end{proof}

\begin{lemma}\label{LemY}
Suppose that a commutative diagram in $\mathscr{C}$
\[
\xy
(-10,-6)*+{A}="0";
(0,5)*+{B}="2";
(10,-6)*+{C}="4";
(0,-9.5)*+{}="3";
{\ar^{f} "0";"2"};
{\ar^{g} "2";"4"};
{\ar_{h} "0";"4"};
{\ar@{}|\circlearrowright "2";"3"};
\endxy
\]
satisfies $f\in w\mathit{Cof},g\in\mathit{Fib}$ and $h\in w\mathit{Cof}$. Then $g$ belongs to $w\mathit{Fib}$.
\end{lemma}
\begin{proof}
By assumption, there are conflations
\begin{eqnarray*}
A\overset{f}{\longrightarrow}B\overset{s_1}{\longrightarrow}S_1&&(S_1\in\mathcal{S}),\\
T\overset{t}{\longrightarrow}B\overset{g}{\longrightarrow}C&&(T\in\mathcal{T}),\\
A\overset{h}{\longrightarrow}C\overset{s_2}{\longrightarrow}S_2&&(S_2\in\mathcal{S}).
\end{eqnarray*}
By the dual of Lemma~\ref{PropShiftOctahedron}, we obtain the following commutative diagram made of conflations.
\[
\xy
(-6,18)*+{T}="-12";
(6,18)*+{T}="-14";
(-18,6)*+{A}="0";
(-6,6)*+{B}="2";
(6,6)*+{S_1}="4";
(-18,-6)*+{A}="10";
(-6,-6)*+{C}="12";
(6,-6)*+{S_2}="14";
{\ar@{=} "-12";"-14"};
{\ar_{t} "-12";"2"};
{\ar^{} "-14";"4"};
{\ar^{f} "0";"2"};
{\ar^{s_1} "2";"4"};
{\ar@{=} "0";"10"};
{\ar_{g} "2";"12"};
{\ar^{} "4";"14"};
{\ar_{h} "10";"12"};
{\ar_{s_2} "12";"14"};
{\ar@{}|\circlearrowright "-12";"4"};
{\ar@{}|\circlearrowright "0";"12"};
{\ar@{}|\circlearrowright "2";"14"};
\endxy
\]
By Lemma~\ref{LemConcentric1} {\rm (2)} and Remark~\ref{RemNiNf} {\rm (2)}, we obtain $T\in\mathcal{T}\cap\mathcal{N}=\mathcal{V}$. This means $g\in w\mathit{Fib}$.
\end{proof}

\begin{proposition}\label{PropComposW}
$\mathbb{W}$ is closed under composition.
\end{proposition}
\begin{proof}
It suffices to show that $w\mathit{Cof}\circ w\mathit{Fib}\subseteq\mathbb{W}$.
Let $a\in w\mathit{Fib}$ and $x^{\prime}\in w\mathit{Cof}$. By Proposition~\ref{PropFactor}, there are some $x\in w\mathit{Cof}$ and $b\in\mathit{Fib}$ such that $b\circ x=x^{\prime}\circ a$. It is thus enough to show that $b$ belongs to $w\mathit{Fib}$. By definition, there is a commutative diagram of $\mathbb{E}$-triangles:
\[
\xy
(-7,21)*+{V}="-12";
(7,21)*+{T}="-14";
(-7,7)*+{A}="2";
(7,7)*+{B}="4";
(21,7)*+{S}="6";
(35,7)*+{}="8";
(-7,-7)*+{A^{\prime}}="12";
(7,-7)*+{B^{\prime}}="14";
(21,-7)*+{S^{\prime}}="16";
(35,-7)*+{}="18";
(-7,-21)*+{}="22";
(7,-21)*+{}="24";
{\ar_{k} "-12";"2"};
{\ar^{k^{\prime}} "-14";"4"};
{\ar^{x} "2";"4"};
{\ar^{y} "4";"6"};
{\ar@{-->}^{\delta} "6";"8"};
{\ar_{a} "2";"12"};
{\ar^{b} "4";"14"};
{\ar_{x^{\prime}} "12";"14"};
{\ar_{y^{\prime}} "14";"16"};
{\ar@{-->}^{\delta^{\prime}} "16";"18"};
{\ar@{-->}_{\kappa} "12";"22"};
{\ar@{-->}^{\kappa^{\prime}} "14";"24"};
{\ar@{}|\circlearrowright "2";"14"};
\endxy
\]
with $V\in\mathcal{V}$, $T\in\mathcal{T}$ and $S,S^{\prime}\in\mathcal{S}$. Applying Lemma~\ref{LemNine} gives some $X\in\mathscr{C}$ and two conflations $X\overset{i}{\longrightarrow}S\overset{c}{\longrightarrow}S^{\prime}$ and $V\overset{m}{\longrightarrow}T\overset{n}{\longrightarrow}X$. The existence of the first conflation (and Lemma~\ref{LemConcentric1}(2)) shows that $X$ belongs to $\mathcal{N}$; that of the latter conflation and the dual of Lemma~\ref{LemConcentric2} imply that $T$ belongs to $\mathcal{V}$, and therefore that $b\in w\mathit{Fib}$.
\end{proof}

\begin{lemma}\label{LemWmodif}
Suppose that a commutative diagram in $\mathscr{C}$
\[
\xy
(-10,-6)*+{A}="0";
(0,5)*+{B}="2";
(10,-6)*+{C}="4";
(0,-9.5)*+{}="3";
{\ar^{f} "0";"2"};
{\ar^{g} "2";"4"};
{\ar_{h} "0";"4"};
{\ar@{}|\circlearrowright "2";"3"};
\endxy
\]
satisfies $f\in w\mathit{Cof},g\in \mathit{Fib}$ and $h\in w\mathit{Fib}$. Then $g$ belongs to $w\mathit{Fib}$.
\end{lemma}
\begin{proof}
By assumption, there are conflations
\begin{eqnarray*}
A\overset{f}{\longrightarrow}B\overset{s}{\longrightarrow}S&&(S\in\mathcal{S}),\\
T\overset{t}{\longrightarrow}B\overset{g}{\longrightarrow}C&&(T\in\mathcal{T}),\\
V\overset{v}{\longrightarrow}A\overset{h}{\longrightarrow}C&&(V\in\mathcal{V}).
\end{eqnarray*}
By Proposition~\ref{PropShiftOctahedron}, we obtain a conflation $V\to T\to S$. Thus from Proposition~\ref{PropThick} {\rm (1)} and Remark~\ref{RemNiNf} {\rm (2)}, it follows $T\in\mathcal{T}\cap\mathcal{N}=\mathcal{V}$. This means $g\in w\mathit{Fib}$.
\end{proof}

\begin{lemma}\label{LemW}
Suppose that a commutative diagram in $\mathscr{C}$
\[
\xy
(-10,-6)*+{A}="0";
(0,5)*+{B}="2";
(10,-6)*+{C}="4";
(0,-9.5)*+{}="3";
{\ar^{f} "0";"2"};
{\ar^{g} "2";"4"};
{\ar_{h} "0";"4"};
{\ar@{}|\circlearrowright "2";"3"};
\endxy
\]
satisfies $f\in w\mathit{Fib},g\in\mathit{Fib}$ and $h\in w\mathit{Fib}$. Then $g$ belongs to $w\mathit{Fib}$.
\end{lemma}
\begin{proof}
By assumption, there are conflations
\begin{eqnarray*}
V_f\to A\overset{f}{\longrightarrow}B&&(V_f\in\mathcal{V}),\\
T\to B\overset{g}{\longrightarrow}C&&(T\in\mathcal{T}),\\
V_h\to A\overset{h}{\longrightarrow}C&&(V_h\in\mathcal{V}).
\end{eqnarray*}
By {\rm (ET4)$^{\mathrm{op}}$}, we obtain a conflation $V_f\to V_h\to T$. Thus from Lemma~\ref{LemConcentric1} {\rm (1)} and Remark~\ref{RemNiNf} {\rm (2)}, it follows that $T\in\mathcal{T}\cap\mathcal{N}=\mathcal{V}$. This means $g\in w\mathit{Fib}$.
\end{proof}

\begin{proposition}\label{PropW23}
Suppose that a commutative diagram in $\mathscr{C}$
\[
\xy
(-10,-6)*+{A}="0";
(0,5)*+{B}="2";
(10,-6)*+{C}="4";
(0,-9.5)*+{}="3";
{\ar^{f} "0";"2"};
{\ar^{g} "2";"4"};
{\ar_{h} "0";"4"};
{\ar@{}|\circlearrowright "2";"3"};
\endxy
\]
satisfies $f,h\in\mathbb{W}$. Then $g$ also belongs to $\mathbb{W}$.
\end{proposition}
\begin{proof}
By definition and the dual of Proposition~\ref{PropFactor}, the morphisms $f,g,h$ can be factorized as
\[
\xy
(-10,6)*+{A}="0";
(10,6)*+{B}="2";
(0,-8)*+{X_f}="4";
(0,10)*+{}="5";
{\ar^{f} "0";"2"};
{\ar_(0.4){f_1} "0";"4"};
{\ar@{<-}^(0.4){f_2} "2";"4"};
{\ar@{}|\circlearrowright "4";"5"};
\endxy,\ \ 
\xy
(-10,6)*+{B}="0";
(10,6)*+{C}="2";
(0,-8)*+{X_g}="4";
(0,10)*+{}="5";
{\ar^{g} "0";"2"};
{\ar_(0.4){g_1} "0";"4"};
{\ar@{<-}^(0.4){g_2} "2";"4"};
{\ar@{}|\circlearrowright "4";"5"};
\endxy,\ \ 
\xy
(-10,6)*+{A}="0";
(10,6)*+{C}="2";
(0,-8)*+{X_h}="4";
(0,10)*+{}="5";
{\ar^{h} "0";"2"};
{\ar_(0.4){h_1} "0";"4"};
{\ar@{<-}^(0.4){h_2} "2";"4"};
{\ar@{}|\circlearrowright "4";"5"};
\endxy,
\]
with $f_1,g_1,h_1\in w\mathit{Cof}$, $f_2,h_2\in w\mathit{Fib}$ and $g_2\in\mathit{Fib}$.

By Proposition~\ref{PropComposW}, the morphism $g_1\circ f_2\in\mathbb{W}$ can be factorized as follows.
\[
\xy
(-10,6)*+{X_f}="0";
(10,6)*+{X_g}="2";
(0,-8)*+{X}="4";
(0,10)*+{}="5";
{\ar^{g_1\circ f_2} "0";"2"};
{\ar_(0.4){w_1} "0";"4"};
{\ar@{<-}^(0.4){w_2} "2";"4"};
{\ar@{}|\circlearrowright "4";"5"};
\endxy
\quad\begin{array}{c}(w_1\in w\mathit{Cof})\\ (w_2\in w\mathit{Fib})\end{array}
\]
By Proposition~\ref{PropLiftST}, there is $k\in\mathscr{C}(X_h,X)$ which makes
\[
\xy
(-8,8)*+{A}="0";
(2,-2)*+{}="1";
(8,8)*+{X}="2";
(-8,-8)*+{X_h}="4";
(-2,2)*+{}="5";
(8,-8)*+{C}="6";
{\ar^{w_1\circ f_1} "0";"2"};
{\ar_{h_1} "0";"4"};
{\ar^{g_2\circ w_2} "2";"6"};
{\ar_{h_2} "4";"6"};
{\ar|*+{_{k}} "4";"2"};
{\ar@{}|\circlearrowright "0";"1"};
{\ar@{}|\circlearrowright "5";"6"};
\endxy
\]
commutative in $\mathscr{C}$. By Proposition~\ref{PropFactor}, we can factorize $k$ as follows
\[
\xy
(-10,6)*+{X_h}="0";
(10,6)*+{X}="2";
(0,-8)*+{X_k}="4";
(0,10)*+{}="5";
{\ar^{k} "0";"2"};
{\ar_(0.4){k_1} "0";"4"};
{\ar@{<-}^(0.4){k_2} "2";"4"};
{\ar@{}|\circlearrowright "4";"5"};
\endxy
\quad\begin{array}{c}(k_1\in w\mathit{Cof})\\ (k_2\in\mathit{Fib})\end{array}
\]
Thus we obtain the following commutative diagram.
\[
\xy
(-20,8)*+{A}="0";
(-8,-2)*+{}="1";
(2,8)*+{X}="2";
(2,-12)*+{}="3";
(12,2)*+{X_g}="4";
(-9,0)*+{X_k}="6";
(-20,-8)*+{X_h}="7";
(12,-8)*+{C}="8";
{\ar^{w_1\circ f_1} "0";"2"};
{\ar^{w_2} "2";"4"};
{\ar_{h_1} "0";"7"};
{\ar_{k_2} "6";"2"};
{\ar^{g_2} "4";"8"};
{\ar_{k_1} "7";"6"};
{\ar_{h_2} "7";"8"};
{\ar@{}|\circlearrowright "0";"1"};
{\ar@{}|\circlearrowright "2";"3"};
\endxy
\quad\begin{array}{c}(k_1,h_1,w_1\circ f_1\in w\mathit{Cof})\\ (w_2,h_2\in w\mathit{Fib})\\ (k_2,g_2\in\mathit{Fib})\end{array}
\]
Lemma~\ref{LemY} shows $k_2\in w\mathit{Fib}$. On the other hand, Lemma~\ref{LemWmodif} shows $g_2\circ w_2\circ k_2\in w\mathit{Fib}$. Thus Lemma~\ref{LemW} shows $g_2\in w\mathit{Fib}$.
\end{proof}

\begin{corollary}\label{CorW2outof3}
The class $\mathbb{W}$ satisfies the 2-out-of-3 condition.
\end{corollary}
\begin{proof}
This follows from Proposition~\ref{PropComposW}, Proposition~\ref{PropW23} and its dual.
\end{proof}

When a category has enough pull-backs or enough push-outs, the fact that weak equivalences are stable under retracts follows from the other axioms (e.g. \cite[Proposition~E.1.3]{Joy}, attributed to Joyal--Tierney). However, that proof does not carry over to the setup of extriangulated categories. The following lemma will thus be used for proving that the class $\mathbb{W}$ is closed under retracts.

\begin{lemma}\label{LemWRetract}
Let $A\overset{x}{\longrightarrow}B\overset{y}{\longrightarrow}C\overset{\delta}{\dashrightarrow}$ and $A\overset{x^{\prime}}{\longrightarrow}B^{\prime}\overset{y^{\prime}}{\longrightarrow}C^{\prime}\overset{\delta^{\prime}}{\dashrightarrow}$ be $\mathbb{E}$-triangles. Suppose that $b\in\mathscr{C}(B,B^{\prime})$ belongs to $\mathbb{W}$ and satisfies $b\circ x=x^{\prime}$. Then there is $c\in\mathscr{C}(C,C^{\prime})$ which belongs to $\mathbb{W}$ and gives a morphism of $\mathbb{E}$-triangles as follows.
\[
\xy
(-12,6)*+{A}="0";
(0,6)*+{B}="2";
(12,6)*+{C}="4";
(24,6)*+{}="6";
(-12,-6)*+{A}="10";
(0,-6)*+{B^{\prime}}="12";
(12,-6)*+{C^{\prime}}="14";
(24,-6)*+{}="16";
{\ar^{x} "0";"2"};
{\ar^{y} "2";"4"};
{\ar@{-->}^{\delta} "4";"6"};
{\ar@{=} "0";"10"};
{\ar^{b} "2";"12"};
{\ar^{c} "4";"14"};
{\ar_{x^{\prime}} "10";"12"};
{\ar_{y^{\prime}} "12";"14"};
{\ar@{-->}_{\delta^{\prime}} "14";"16"};
{\ar@{}|\circlearrowright "0";"12"};
{\ar@{}|\circlearrowright "2";"14"};
\endxy
\]
\end{lemma}
\begin{proof}
By definition, $b$ can be factorized as $b=v\circ s$,
using $\mathbb{E}$-triangles $B\overset{s}{\longrightarrow}P\to S\overset{\theta}{\dashrightarrow}$ and $V\to P\overset{v}{\longrightarrow}B^{\prime}\overset{\tau}{\dashrightarrow}$ with $S\in\mathcal{S},V\in\mathcal{V}$. By {\rm (ET4)}, and then by the dual of Proposition~\ref{PropShiftOctahedron}, we obtain the following commutative diagrams made of $\mathbb{E}$-triangles,
\[
\xy
(-21,14)*+{A}="0";
(-7,14)*+{B}="2";
(7,14)*+{C}="4";
(-21,0)*+{A}="10";
(-7,0)*+{P}="12";
(7,0)*+{{}^{\exists}Q}="14";
(-7,-14)*+{S}="22";
(7,-14)*+{S}="24";
{\ar^{x} "0";"2"};
{\ar^{y} "2";"4"};
{\ar@{-->}^{\delta} "4";(19,14)};
{\ar@{=} "0";"10"};
{\ar_{s} "2";"12"};
{\ar^{{}^{\exists}c_1} "4";"14"};
{\ar_{s\circ x} "10";"12"};
{\ar_{{}^{\exists}p} "12";"14"};
{\ar@{-->}^{{}^{\exists}\nu} "14";(19,0)};
{\ar_{} "12";"22"};
{\ar^{} "14";"24"};
{\ar@{=} "22";"24"};
{\ar@{-->}_{\theta} "22";(-7,-27)};
{\ar@{-->}^{y_{\ast}\theta} "24";(7,-27)};
{\ar@{}|\circlearrowright "0";"12"};
{\ar@{}|\circlearrowright "2";"14"};
{\ar@{}|\circlearrowright "12";"24"};
\endxy
\qquad,\qquad
\xy
(-7,14)*+{A}="2";
(7,14)*+{A}="4";
(-21,0)*+{V}="10";
(-7,0)*+{P}="12";
(7,0)*+{B^{\prime}}="14";
(21,0)*+{}="16";
(-21,-14)*+{V}="20";
(-7,-14)*+{Q}="22";
(7,-14)*+{C^{\prime}}="24";
(21,-14)*+{}="26";
(-7,-28)*+{}="32";
(7,-28)*+{}="34";
{\ar@{=} "2";"4"};
{\ar_{s\circ x} "2";"12"};
{\ar^{x^{\prime}} "4";"14"};
{\ar^{} "10";"12"};
{\ar_{v} "12";"14"};
{\ar@{-->}^{\tau} "14";"16"};
{\ar@{=} "10";"20"};
{\ar^{p} "12";"22"};
{\ar^{y^{\prime}} "14";"24"};
{\ar_{} "20";"22"};
{\ar_{{}^{\exists}c_2} "22";"24"};
{\ar@{-->}_{} "24";"26"};
{\ar@{-->}_{\nu} "22";"32"};
{\ar@{-->}^{\delta^{\prime}} "24";"34"};
{\ar@{}|\circlearrowright "2";"14"};
{\ar@{}|\circlearrowright "10";"22"};
{\ar@{}|\circlearrowright "12";"24"};
\endxy
\]
in which $c_1^{\ast}\nu=\delta$ holds.
Then $c=c_2\circ c_1$ belongs to $\mathit{wCof}\circ\mathit{wFib}\subseteq\mathbb{W}$ by Proposition~\ref{PropComposW}, satisfies $c\circ y=c_2\circ p\circ s=y^{\prime}\circ v\circ s=y^{\prime}\circ b$ and $c^{\ast}\delta^{\prime}=c_1^{\ast}c_2^{\ast}\delta^{\prime}=c_1^{\ast}\nu=\delta$.
\end{proof}

\begin{proposition}\label{PropWRetract}
The class $\mathbb{W}$ is closed under retracts.
\end{proposition}
\begin{proof}
Suppose we are given a commutative diagram in $\mathscr{C}$
\[
\xy
(-12,6)*+{A}="0";
(0,6)*+{C}="2";
(0,12)*+{}="3";
(12,6)*+{A}="4";
(-12,-6)*+{B}="10";
(0,-6)*+{D}="12";
(0,-12)*+{}="13";
(12,-6)*+{B}="14";
{\ar@/^1.20pc/^{\mathrm{id}} "0";"4"};
{\ar_{a} "0";"2"};
{\ar_{c} "2";"4"};
{\ar_{f} "0";"10"};
{\ar^{g} "2";"12"};
{\ar^{f} "4";"14"};
{\ar@/_1.20pc/_{\mathrm{id}} "10";"14"};
{\ar_{b} "10";"12"};
{\ar_{d} "12";"14"};
{\ar@{}|\circlearrowright "2";"3"};
{\ar@{}|\circlearrowright "0";"12"};
{\ar@{}|\circlearrowright "2";"14"};
{\ar@{}|\circlearrowright "12";"13"};
\endxy
\]
in which $g\in\mathbb{W}$. Let us show that $f\in\mathbb{W}$.
If we decompose $f$ and $g$ as
\[
\xy
(-8,5)*+{A}="0";
(8,5)*+{B}="2";
(0,-6)*+{M}="4";
(0,8)*+{}="5";
{\ar^{f} "0";"2"};
{\ar_(0.4){i} "0";"4"};
{\ar@{<-}^(0.4){x} "2";"4"};
{\ar@{}|\circlearrowright "4";"5"};
\endxy\ \ ,\ \ 
\xy
(-8,5)*+{C}="0";
(8,5)*+{D}="2";
(0,-6)*+{N}="4";
(0,8)*+{}="5";
{\ar^{g} "0";"2"};
{\ar_(0.4){j} "0";"4"};
{\ar@{<-}^(0.4){y} "2";"4"};
{\ar@{}|\circlearrowright "4";"5"};
\endxy
\quad\left(\begin{array}{c}i\in\mathit{Cof},\, j\in\mathit{wCof},\\ x,y\in w\mathit{Fib}\end{array}\right)
\]
by Proposition~\ref{PropFactor}, then there exist morphisms $m,n$ which make the following diagram commutative by Proposition~\ref{PropLiftST} .
\[
\xy
(-14,12)*+{A}="2";
(0,12)*+{C}="4";
(14,12)*+{A}="6";
(-14,0)*+{M}="12";
(0,0)*+{N}="14";
(14,0)*+{M}="16";
(-14,-12)*+{B}="22";
(0,-12)*+{D}="24";
(14,-12)*+{B}="26";
{\ar^{a} "2";"4"};
{\ar^{c} "4";"6"};
{\ar_{i} "2";"12"};
{\ar^{j} "4";"14"};
{\ar^{i} "6";"16"};
{\ar_{m} "12";"14"};
{\ar_{n} "14";"16"};
{\ar_{x} "12";"22"};
{\ar^{y} "14";"24"};
{\ar^{x} "16";"26"};
{\ar_{b} "22";"24"};
{\ar_{d} "24";"26"};
{\ar@{}|\circlearrowright "2";"14"};
{\ar@{}|\circlearrowright "4";"16"};
{\ar@{}|\circlearrowright "12";"24"};
{\ar@{}|\circlearrowright "14";"26"};
\endxy
\]
By Corollary~\ref{CorW2outof3} applied to the lower half, it follows $n\circ m\in\mathbb{W}$. By definition of $\mathit{Cof}$ and $\mathit{wCof}$, there are $\mathbb{E}$-triangles
\[ A\overset{i}{\longrightarrow}M\overset{p}{\longrightarrow}U\overset{\rho}{\dashrightarrow},\qquad C\overset{j}{\longrightarrow}N\overset{q}{\longrightarrow}S\overset{\tau}{\dashrightarrow} \]
with $U\in\mathcal{U},S\in\mathcal{S}$. It suffices to show $U\in\mathcal{S}$.

Realize $c_{\ast}\tau$ by an $\mathbb{E}$-triangle $A\overset{j^{\prime}}{\longrightarrow}N^{\prime}\overset{q^{\prime}}{\longrightarrow}S\overset{c_{\ast}\tau}{\dashrightarrow}$. Put $c^{\prime}=\Big[\raise1ex\hbox{\leavevmode\vtop{\baselineskip-8ex \lineskip1ex \ialign{#\crcr{$-c$}\crcr{$\,\, j$}\crcr}}}\Big]\colon C\to A\oplus N$. Then by an argument similar to that of the proof of Corollary~\ref{CorTrivWIC}, we can find a morphism of $\mathbb{E}$-triangles
\[
\xy
(-12,6)*+{C}="0";
(0,6)*+{N}="2";
(12,6)*+{S}="4";
(24,6)*+{}="6";
(-12,-6)*+{A}="10";
(0,-6)*+{N^{\prime}}="12";
(12,-6)*+{S}="14";
(24,-6)*+{}="16";
{\ar^{j} "0";"2"};
{\ar^{q} "2";"4"};
{\ar@{-->}^{\tau} "4";"6"};
{\ar_{c} "0";"10"};
{\ar_{n_1} "2";"12"};
{\ar@{=} "4";"14"};
{\ar_{j^{\prime}} "10";"12"};
{\ar_{q^{\prime}} "12";"14"};
{\ar@{-->}_{c_{\ast}} "14";"16"};
{\ar@{}|\circlearrowright "0";"12"};
{\ar@{}|\circlearrowright "2";"14"};
\endxy
\]
which gives an $\mathbb{E}$-triangle $C\overset{c^{\prime}}{\longrightarrow}A\oplus N\overset{[j^{\prime}\ n_1]}{\longrightarrow}N^{\prime}\overset{q^{\prime\ast}\tau}{\dashrightarrow}$ (cf. \cite[Proposition~1.20]{LN}). Since we have $[i\ n]\circ c^{\prime}=n\circ j-i\circ c=0$, there is $n^{\prime}\in\mathscr{C}(N^{\prime},M)$ which satisfies $n^{\prime}\circ [j^{\prime}\ n_1]=[i\ n]$, namely $n^{\prime}\circ j^{\prime}=i$ and $n^{\prime}\circ n_1=n$.

Put $m^{\prime}=n_1\circ m$. This satisfies $n^{\prime}\circ m^{\prime}=n\circ m\in\mathbb{W}$ and $m^{\prime}\circ i=j^{\prime}$. Resolve $N^{\prime}$ by an $\mathbb{E}$-triangle
\[ T^{\prime}\to S^{\prime}\overset{s^{\prime}}{\longrightarrow}N^{\prime}\overset{\theta}{\dashrightarrow}\quad(S^{\prime}\in\mathcal{S},T^{\prime}\in\mathcal{T}). \]
Then by the dual of Corollary~\ref{CorTrivWIC}, the morphism $[m^{\prime}\ s^{\prime}]\colon M\oplus S^{\prime}\to N^{\prime}$ can be completed into an $\mathbb{E}$-triangle ${}^{\exists}L\to M\oplus S^{\prime}\overset{[m^{\prime}\ s^{\prime}]}{\longrightarrow}N^{\prime}\dashrightarrow$. By the dual of Proposition~\ref{PropShiftOctahedron}, we obtain the following commutative diagram made of conflations.
\begin{equation}\label{DiagLMN}
\xy
(-8,21)*+{A}="-12";
(8,21)*+{A}="-14";
(-24,7)*+{L}="0";
(-8,7)*+{M\oplus S^{\prime}}="2";
(8,7)*+{N^{\prime}}="4";
(-24,-7)*+{L}="10";
(-8,-7)*+{U\oplus S^{\prime}}="12";
(8,-7)*+{S}="14";
{\ar@{=} "-12";"-14"};
{\ar_{} "-12";"2"};
(-11.6,15)*+{\Big[\ \Big]}="-1";
(-11.6,16.6)*+{\scriptstyle{i}};
(-11.6,13.4)*+{\scriptstyle{0}};
{\ar^{j^{\prime}} "-14";"4"};
{\ar^{} "0";"2"};
{\ar^(0.62){[m^{\prime}\ s^{\prime}]} "2";"4"};
{\ar@{=} "0";"10"};
{\ar|*+{_{p\oplus\mathrm{id}}} "2";"12"};
{\ar^{q^{\prime}} "4";"14"};
{\ar_(0.4){{}^{\exists}\ell} "10";"12"};
{\ar_(0.6){{}^{\exists}k} "12";"14"};
{\ar@{}|\circlearrowright "-12";"4"};
{\ar@{}|\circlearrowright "0";"12"};
{\ar@{}|\circlearrowright "2";"14"};
\endxy
\end{equation}
If we put $m_0=n^{\prime}\circ [m^{\prime}\ s^{\prime}]$, then since
\[
\xy
(-10,-6)*+{M\oplus S^{\prime}}="0";
(0,5)*+{M}="2";
(10,-6)*+{M}="4";
(0,-9.5)*+{}="3";
{\ar_{} "2";"0"};
(-7.6,4)*+{\Big[\ \Big]}="-1";
(-7.6,5.6)*+{\scriptstyle{1}};
(-7.6,2.4)*+{\scriptstyle{0}};
{\ar_(0.56){m_0} "0";"4"};
{\ar^{n^{\prime}\circ m^{\prime}} "2";"4"};
{\ar@{}|\circlearrowright "2";"3"};
\endxy
\]
is commutative, $m_0\in\mathbb{W}$ follows from $\Big[\raise1ex\hbox{\leavevmode\vtop{\baselineskip-8ex \lineskip1ex \ialign{#\crcr{$1$}\crcr{$0$}\crcr}}}\Big]\in\mathit{wCof}$ and $n^{\prime}\circ m^{\prime}\in\mathbb{W}$, by Corollary~\ref{CorW2outof3}.

Applying Lemma~\ref{LemWRetract} to
\[
\xy
(-18,7)*+{A}="0";
(0,7)*+{M\oplus S^{\prime}}="2";
(18,7)*+{U\oplus S^{\prime}}="4";
(36,7)*+{}="6";
(-18,-7)*+{A}="10";
(0,-7)*+{M}="12";
(18,-7)*+{U}="14";
(36 ,-7)*+{}="16";
{\ar^{} "0";"2"};
(-11.6,11)*+{\Big[\ \Big]}="-1";
(-11.6,12.6)*+{\scriptstyle{i}};
(-11.6,9.4)*+{\scriptstyle{0}};
{\ar^{p\oplus\mathrm{id}} "2";"4"};
{\ar@{-->}^{} "4";"6"};
{\ar@{=} "0";"10"};
{\ar^{m_0} "2";"12"};
%
{\ar_{i} "10";"12"};
{\ar_{p} "12";"14"};
{\ar@{-->}_{} "14";"16"};
{\ar@{}|\circlearrowright "0";"12"};
\endxy
\]
we obtain $u\in\mathscr{C}(U\oplus S^{\prime},U)$ which belongs to $\mathbb{W}$ satisfying $u\circ(p\oplus\mathrm{id}_{S^{\prime}})=p\circ m_0$. Then since $u\circ\ell=0$, we see that $u$ factors through $S$, in the bottom $\mathbb{E}$-triangle in $(\ref{DiagLMN})$. Thus if we apply the functor $\sigma\colon\mathcal{U}/\mathcal{I}\to\mathcal{Z}/\mathcal{I}$, it follows $\sigma(\underline{u})=0$. On the other hand, it can be easily seen that $u\in\mathbb{W}$ implies that $u$ can be written as composition of $U\oplus S^{\prime}\overset{u^{\prime}}{\longrightarrow}U\oplus I\overset{[1\ 0]}{\longrightarrow}U$, with $u^{\prime}\in\mathit{wCof}$ and $I\in\mathcal{I}$. By {\rm (ET4)}, we can show that $\sigma(\underline{u}^{\prime})$ is an isomorphism, and thus $\sigma(\underline{u})$ is an isomorphism.
This means $\sigma(U)=0$ in $\mathcal{Z}/\mathcal{I}$, which shows $U\in\mathcal{U}\cap\mathrm{CoCone(\mathcal{I},\mathcal{S})}\subseteq\mathcal{U}\cap\mathcal{N}=\mathcal{S}$.
\end{proof}

By the argument so far, admissible model structures and Hovey twin cotorsion pairs on $(\mathscr{C},\mathbb{E},\mathfrak{s})$ correspond bijectively. Remark that, a model structure induces an equivalence
\[ \mathscr{C}_{\mathit{cf}}/_\sim\overset{\simeq}{\longrightarrow}\mathscr{C}[\mathbb{W}^{-1}]. \]
Here, the right hand side is the localization $\ell\colon\mathscr{C}\to\mathscr{C}[\mathbb{W}^{-1}]$. The left hand side is the category of fibrant-cofibrant objects modulo homotopies.
Let us describe it in terms of the corresponding Hovey twin cotorsion pair $\mathcal{P}=((\mathcal{S},\mathcal{T}),(\mathcal{U},\mathcal{V}))$.
\begin{itemize}
\item[-] $X\in\mathscr{C}$ is fibrant if and only if $(X\to0)\in\mathit{Fib}$, if and only if $X\in\mathcal{T}$. Dually, $X$ is cofibrant if and only if $X\in\mathcal{U}$. Thus the full subcategory of fibrant-cofibrant objects in $\mathscr{C}$ agrees with $\mathcal{Z}\subseteq\mathscr{C}$.
\item[-] For any $X,Y\in\mathcal{Z}$, morphisms $f,g\in\mathcal{Z}(X,Y)$ satisfy $f\sim g$ if and only if $f-g$ factors through some object $I\in\mathcal{I}$.
\end{itemize}
Thus we have $\mathscr{C}_{\mathit{cf}}/_\sim=\mathcal{Z}/\mathcal{I}$. In summary, we obtain the following.
This gives an explanation for the equivalence in \cite[Proposition~6.10]{Na4} and \cite[Theorem~4.1]{IYa}.
\begin{corollary}\label{CorHTCP}
Let $(\mathscr{C},\mathbb{E},\mathfrak{s})$ be an extriangulated category, and let $\mathcal{P}=((\mathcal{S},\mathcal{T}),(\mathcal{U},\mathcal{V}))$ be a Hovey twin cotorsion pair. Then for $\mathbb{W}=w\mathit{Fib}\circ w\mathit{Cof}$ defined as above, we have an equivalence $\mathcal{Z}/\mathcal{I}\overset{\simeq}{\longrightarrow}\mathscr{C}[\mathbb{W}^{-1}]$ which makes the following diagram commutative up to natural isomorphism.
\[
\xy
(-10,8)*+{\mathcal{Z}}="0";
(10,8)*+{\mathscr{C}}="2";
(-10,-8)*+{\mathcal{Z}/\mathcal{I}}="4";
(10,-8)*+{\mathscr{C}[\mathbb{W}^{-1}]}="6";
{\ar^{\text{inclusion}} "0";"2"};
{\ar_(0.42){\text{ideal}}_(0.58){\text{quotient}} "0";"4"};
{\ar^{\ell} "2";"6"};
{\ar_(0.46){\simeq} "4";"6"};
{\ar@{}|\circlearrowright "0";"6"};
\endxy
\]
In particular, the map
\[ (\mathcal{Z}/\mathcal{I})(X,Y)\to\mathscr{C}[\mathbb{W}^{-1}](X,Y)\ ;\ \overline{f}\mapsto\ell(f) \]
is an isomorphism for any $X,Y\in\mathcal{Z}$.
\end{corollary}

\begin{remark}\label{RemUTMono}
By the generality of a model structure, we can also deduce that
\[ (\mathscr{C}/\mathcal{I})(U,T)\to\mathscr{C}[\mathbb{W}^{-1}](U,T)\ ;\ \overline{f}\mapsto\ell(f) \]
is an isomorphism for any $U\in\mathcal{U}$ and $T\in\mathcal{T}$. (This also follows from the adjoint property given in Remark~\ref{RemTCP}.)
\end{remark}

\begin{remark}
If $\mathscr{C}$ is abelian, then we can show easily that $\mathbb{W}$ agrees with the class of morphisms $f$ satisfying $\mathrm{Ker}(f)\in\mathcal{N}$ and $\mathrm{Cok}(f)\in\mathcal{N}$. Remark that $\mathcal{N}\subseteq\mathscr{C}$ becomes a Serre subcategory only when $\mathcal{N}=\mathscr{C}$. Indeed, if $\mathcal{N}\subseteq\mathscr{C}$ is a Serre subcategory, then the localization $\ell\colon\mathscr{C}\to\mathscr{C}[\mathbb{W}^{-1}]$ becomes an exact functor between abelian categories. Since any $C\in\mathscr{C}$ admits an inflation $C\to V$ to some $V\in\mathcal{V}$, this shows that $C=0$ holds in $\mathscr{C}[\mathbb{W}^{-1}]$, which means $C\in\mathcal{N}$.
\end{remark}

\section{Triangulation of the homotopy category}\label{section_Tria}

In this section, we assume $((\mathcal{S},\mathcal{T}),(\mathcal{U},\mathcal{V}))$ to be a Hovey twin cotorsion pair.
Put $\widetilde{\mathscr{C}}=\mathscr{C}[\mathbb{W}^{-1}]$ and let $\ell\colon\mathscr{C}\to\widetilde{\mathscr{C}}$ be the localization functor. We will show that $\widetilde{\mathscr{C}}$ is triangulated (Theorem~\ref{ThmTriaLoc}).

\begin{remark}
 We note that the category $\mathscr{C}$ is usually not complete and cocomplete, so that the model structure is not stable. However, axiom {\rm (ET4)} gives specific choices of weak bicartesian squares which will compensate for the lack of stability.
\end{remark}

\subsection{Shift functor}
We first aim at defining a shift functor on the category $\widetilde{\mathscr{C}}$. 

\begin{definition}\label{DefShiftLoc}
Let us fix a choice, for any object $A\in\mathscr{C}$, of an $\mathbb{E}$-triangle $A\overset{v_A}{\longrightarrow}V_A\overset{u_A}{\longrightarrow}U_A\overset{\rho_A}{\dashrightarrow}$, with $V_A\in\mathcal{V}$ and $U_A\in\mathcal{U}$. The functor $[1]\colon\mathscr{C}\to\widetilde{\mathscr{C}}$ is defined on objects by $A[1] = U_A$, and on morphisms as follows: Let $A\overset{f}{\longrightarrow}B$ be a morphism in $\mathscr{C}$. Then there exists $u_f\in\mathscr{C}(U_A,U_B)$ which gives a morphism of $\mathbb{E}$-extensions $(f,u_f)\colon\rho_A\to\rho_B$. Indeed, since $\mathbb{E}(U_A,V_B)=0$, one shows, by using the long exact sequence of Proposition~\ref{PropExact1} that there is a morphism $V_A\overset{v_f}{\longrightarrow}V_B$ such that $v_f\circ v_A = v_B\circ f$. There is an induced morphism of $\mathbb{E}$-triangles as follows.
\[
\xy
(-12,6)*+{A}="0";
(0,6)*+{V_A}="2";
(12,6)*+{U_A}="4";
(24,6)*+{}="6";
(-12,-6)*+{B}="10";
(0,-6)*+{V_{B}}="12";
(12,-6)*+{U_{B}}="14";
(24,-6)*+{}="16";
{\ar^{v_A} "0";"2"};
{\ar^{u_A} "2";"4"};
{\ar@{-->}^{\rho_A} "4";"6"};
{\ar_{f} "0";"10"};
{\ar^{v_f} "2";"12"};
{\ar^{u_f} "4";"14"};
{\ar_{v_{B}} "10";"12"};
{\ar_{u_{B}} "12";"14"};
{\ar@{-->}_{\rho_B} "14";"16"};
{\ar@{}|\circlearrowright "0";"12"};
{\ar@{}|\circlearrowright "2";"14"};
\endxy
\]
Define $f[1]$ to be the image $\ell(u_f)$ of $u_f$ in $\widetilde{\mathscr{C}}$. Since $f_\ast\rho_A = u_f^\ast\rho_B$, the morphism $u_f$ is uniquely defined in $\mathscr{C}/\mathcal{V}$ by Corollary~\ref{CorExact0}. This implies that $f[1]=\ell(u_f)$ is well-defined.
\end{definition}

\begin{claim}\label{ClaimNatIso}
 The functor $[1]$ does not essentially depend on the choices made. More precisely, fix any other choice of $\mathbb{E}$-triangles $A\overset{v^{\prime}_A}{\longrightarrow}V^{\prime}_A\overset{u^{\prime}_A}{\longrightarrow}U^{\prime}_A\overset{\rho^{\prime}_A}{\dashrightarrow}$ and let $\{1\}$ be the functor defined as above by means of these $\mathbb{E}$-triangles. Then $[1]$ and $\{1\}$ are naturally isomorphic.
\end{claim}
\begin{proof}
The identity on $A$ induces two morphisms of $\mathbb{E}$-extensions
\[ (\mathrm{id},{}^{\exists}t_A)\colon\rho_A\to\rho^{\prime}_A\quad\text{and}\quad(\mathrm{id},{}^{\exists}t^{\prime}_A)\colon\rho^{\prime}_A\to\rho_A. \]
Then, since both
\[ (\mathrm{id},t^{\prime}_A\circ t_A),\ (\mathrm{id},\mathrm{id})\colon\rho_A\to\rho_A \]
are morphisms of $\mathbb{E}$-extensions, we have $\ell(t^{\prime}_A\circ t_A)=\ell(\mathrm{id})=\mathrm{id}$ as in the argument in Definition~\ref{DefShiftLoc}. Similarly we have $\ell(t_A\circ t^{\prime}_A)=\mathrm{id}$, and thus $\ell(t_A)$ is an isomorphism.

Put $\tau_A=\ell(t_A)$, and let us show the naturality of
\[ \tau=\{\tau_A\in\widetilde{\mathscr{C}}(A[1],A\{1\})\}_{A\in\mathscr{C}}. \]
Let $f\in\mathscr{C}(A,B)$ be any morphism. By definition, $f[1]=\ell(u)$ and $f\{1\}=\ell(u^{\prime})$ are given by morphisms of $\mathbb{E}$-extensions
\[ (f,u)\colon\rho_A\to\rho_B\quad\text{and}\quad (f^{\prime},u^{\prime})\colon\rho^{\prime}_A\to\rho^{\prime}_B. \]
Then, since both $(f,t_B\circ u)$ and $(f,u^{\prime}\circ t_A)$ are morphisms $\rho_A\to\rho^{\prime}_B$, we obtain $\ell(t_B\circ u)=\ell(u^{\prime}\circ t_A)$, namely $\tau_B\circ f[1]=f\{1\}\circ\tau_A$.
\end{proof}

We would like to show that $[1]\colon\mathscr{C}\to\widetilde{\mathscr{C}}$ induces an endofunctor of $\widetilde{\mathscr{C}}$. For this, it is enough to show that $[1]$ sends weak equivalences to isomorphisms. Since weak equivalences are compositions of a morphism in $w\mathit{Cof}$ followed by a morphism in $w\mathit{Fib}$,
it is enough to show that $[1]$ inverts all morphisms in $w\mathit{Cof}$ and in $w\mathit{Fib}$.

\begin{lemma}\label{LemWFibToIso}
 Let $j\in w\mathit{Cof}$. Then $j[1]$ is an isomorphism in $\widetilde{\mathscr{C}}$.
\end{lemma}
\begin{proof}
Let $A\overset{j}{\longrightarrow}B$ be a morphism in $w\mathit{Cof}$. There is an $\mathbb{E}$-triangle $A\overset{j}{\longrightarrow}B\to S \dashrightarrow$, with $S\in\mathcal{S}$. Then {\rm (ET4)} gives morphisms of $\mathbb{E}$-triangles as follows.
\[
\xy
(-21,7)*+{A}="0";
(-7,7)*+{B}="2";
(7,7)*+{S}="4";
(-21,-7)*+{A}="10";
(-7,-7)*+{V_B}="12";
(7,-7)*+{C}="14";
(-7,-21)*+{U_B}="22";
(7,-21)*+{U_B}="24";
{\ar^{j} "0";"2"};
{\ar "2";"4"};
{\ar@{-->} "4";(19,7)};
{\ar@{=} "0";"10"};
{\ar_{v_B} "2";"12"};
{\ar@[gray] "4";"14"};
{\ar_{v^{\prime}_A} "10";"12"};
{\ar@[gray]_{u^{\prime}_A} "12";"14"};
{\ar@[gray]@{-->}^{\rho^{\prime}} "14";(19,-7)};
{\ar_{u_B} "12";"22"};
{\ar@[gray]^{w} "14";"24"};
{\ar@[gray]@{=} "22";"24"};
{\ar@{-->}^{\rho_B} "22";(-7,-34)};
{\ar@[gray]@{-->} "24";(7,-34)};
{\ar@{}|\circlearrowright "0";"12"};
{\ar@{}|\circlearrowright "2";"14"};
{\ar@{}|\circlearrowright "12";"24"};
\endxy
\]
Since $S$ and $U_B$ belong to $\mathcal{U}$, so does $C$. Moreover, $w$ is a weak equivalence by Claim~\ref{ClaimNW} {\rm (2)}. Claim~\ref{ClaimNatIso} allows to conclude that $j[1]$ is an isomorphism in $\widetilde{\mathscr{C}}$, since
\[
\xy
(-7,6)*+{A}="12";
(7,6)*+{V_B}="14";
(21,6)*+{C}="16";
(-7,-6)*+{B}="22";
(7,-6)*+{V_B}="24";
(21,-6)*+{U_B}="26";
{\ar^{v^{\prime}_A} "12";"14"};
{\ar^{u^{\prime}_A} "14";"16"};
{\ar@{-->}^{\rho^{\prime}_A} "16";(33,6)};
{\ar_{j} "12";"22"};
{\ar@{=} "14";"24"};
{\ar^{w} "16";"26"};
{\ar_{v_B} "22";"24"};
{\ar_{u_B} "24";"26"};
{\ar@{-->}_{\rho_B} "26";(33,-6)};
{\ar@{}|\circlearrowright "12";"24"};
{\ar@{}|\circlearrowright "14";"26"};
\endxy
\]
is a morphism of $\mathbb{E}$-triangles.
\end{proof}

\begin{lemma}\label{LemWCofToIso}
 Let $q\in w\mathit{Fib}$. Then $q[1]$ is an isomorphism in $\widetilde{\mathscr{C}}$.
\end{lemma}
\begin{proof}
Let $X\overset{q}{\longrightarrow}Y$ be a morphism in $w\mathit{Fib}$. It induces a morphism of $\mathbb{E}$-triangles:
\[
\xy
(-12,6)*+{X}="0";
(0,6)*+{V_X}="2";
(12,6)*+{U_X}="4";
(24,6)*+{}="6";
(-12,-6)*+{Y}="10";
(0,-6)*+{V_{Y}}="12";
(12,-6)*+{U_{Y}}="14";
(24,-6)*+{}="16";
{\ar^{v_X} "0";"2"};
{\ar^{u_X} "2";"4"};
{\ar@{-->}^{\rho_X} "4";"6"};
{\ar_{q} "0";"10"};
{\ar^{n} "2";"12"};
{\ar^{q[1]} "4";"14"};
{\ar_{v_{Y}} "10";"12"};
{\ar_{u_{Y}} "12";"14"};
{\ar@{-->}_{\rho_Y} "14";"16"};
{\ar@{}|\circlearrowright "0";"12"};
{\ar@{}|\circlearrowright "2";"14"};
\endxy
\]
Since $V_X\to 0$ and $V_Y\to 0$ are acyclic fibrations, in particular they are weak equivalences. By the 2-out-of-3 property, this implies that $V_X\overset{n}{\longrightarrow}V_Y$ is a weak equivalence. Factor $n$ as an acyclic cofibration $j$ followed by an acyclic fibration $p$. Then {\rm (ET4)} gives a diagram of $\mathbb{E}$-triangles:
\[
\xy
(-21,7)*+{X}="0";
(-7,7)*+{V_X}="2";
(7,7)*+{U_X}="4";
(-21,-7)*+{X}="10";
(-7,-7)*+{B}="12";
(7,-7)*+{C}="14";
(-7,-21)*+{S}="22";
(7,-21)*+{S}="24";
{\ar^{v_X} "0";"2"};
{\ar^{u_X} "2";"4"};
{\ar@{=} "0";"10"};
{\ar_{j} "2";"12"};
{\ar@[gray] "4";"14"};
{\ar_{j\circ v_X} "10";"12"};
{\ar@[gray] "12";"14"};
{\ar "12";"22"};
{\ar@[gray] "14";"24"};
{\ar@[gray]@{=} "22";"24"};
{\ar@{}|\circlearrowright "0";"12"};
{\ar@{}|\circlearrowright "2";"14"};
{\ar@{}|\circlearrowright "12";"24"};
\endxy
\]
where $S\in \mathcal{S}$. We have $B\in\mathcal{N}$ (since $V_X,S\in\mathcal{N}$) and $C\in\mathcal{U}$ (since $U_X,S\in\mathcal{U}$). The nine Lemma~\ref{LemNine} gives morphisms of $\mathbb{E}$-triangles:
\[
\xy
(-7,21)*+{V}="-12";
(7,21)*+{V^{\prime}}="-14";
(21,21)*+{D}="-16";
(35,21)*+{}="-18";
(-7,7)*+{X}="2";
(7,7)*+{B}="4";
(21,7)*+{C}="6";
(35,7)*+{}="8";
(-7,-7)*+{Y}="12";
(7,-7)*+{V_Y}="14";
(21,-7)*+{U_Y}="16";
(35,-7)*+{}="18";
(-7,-21)*+{}="22";
(7,-21)*+{}="24";
(21,-21)*+{}="26";
{\ar@[gray] "-12";"-14"};
{\ar@[gray] "-14";"-16"};
{\ar@[gray]@{-->} "-16";"-18"};
{\ar "-12";"2"};
{\ar "-14";"4"};
{\ar@[gray] "-16";"6"};
{\ar^{j\circ v_X} "2";"4"};
{\ar "4";"6"};
{\ar@{-->} "6";"8"};
{\ar_{q} "2";"12"};
{\ar^{p} "4";"14"};
{\ar^{q^{\prime}} "6";"16"};
{\ar^{v_Y} "12";"14"};
{\ar^{u_Y} "14";"16"};
{\ar@{-->}^{\rho_Y} "16";"18"};
{\ar@{-->} "12";"22"};
{\ar@{-->} "14";"24"};
{\ar@[gray]@{-->} "16";"26"};
{\ar@{}|\circlearrowright "-12";"4"};
{\ar@{}|\circlearrowright "-14";"6"};
{\ar@{}|\circlearrowright "2";"14"};
{\ar@{}|\circlearrowright "4";"16"};
\endxy
\]
with $V,V^{\prime}\in\mathcal{V}$, and thus $B\in\mathcal{V}$. This implies $D\in\mathcal{N}$ and thus $q^{\prime}$ is a weak equivalence by Claim~\ref{ClaimNW} {\rm (2)}. We conclude that $q[1]$ is an isomorphism in $\widetilde{\mathscr{C}}$ in the same manner as in the end of the proof of Lemma~\ref{LemWFibToIso}.
\end{proof}

\begin{corollary}
The functor $[1]\colon\mathscr{C}\to\widetilde{\mathscr{C}}$ induces an endofunctor of $\widetilde{\mathscr{C}}$, which we denote by the same symbol $[1]$. Thus $f[1]$ for a morphism $f$ in $\mathscr{C}$ will be denoted also by $\ell(f)[1]$ in the rest.
\end{corollary}
\begin{proof}
By Lemma~\ref{LemWFibToIso} and Lemma~\ref{LemWCofToIso}, the functor $[1]\colon\mathscr{C}\to\widetilde{\mathscr{C}}$ inverts all weak equivalences. By the universal property of the localization of a category, it induces a functor $[1]\colon\widetilde{\mathscr{C}}\to\widetilde{\mathscr{C}}$.
\end{proof}

\begin{definition}
By using $\mathbb{E}$-triangles $T^C\overset{t^C}{\longrightarrow}S^C\overset{s^C}{\longrightarrow}C\overset{\lambda^C}{\dashrightarrow}$ one can define dually a functor $[-1]\colon\mathscr{C}\to\widetilde{\mathscr{C}}$ with $C[-1]=T^C$, which induces an endofunctor $[-1]$ of $\widetilde{\mathscr{C}}$.
\end{definition}

\subsection{Connecting morphism}
Define the bifunctor
\[ \widetilde{\mathbb{E}}\colon\widetilde{\mathscr{C}}^{\mathrm{op}}\times\widetilde{\mathscr{C}}\to\widetilde{\mathscr{C}} \]
by $\widetilde{\mathbb{E}}=\widetilde{\mathscr{C}}(-,-[1])$. In this section, we will construct a homomorphism
\[ \widetilde{\ell}=\widetilde{\ell}_{C,A}\colon\mathbb{E}(C,A)\to\widetilde{\mathbb{E}}(C,A) \]
for each pair $A,C\in\mathscr{C}$.
For any $A\in\mathscr{C}$, we continue to use the $\mathbb{E}$-triangle
\begin{equation}\label{AUchosen}
A\overset{v_A}{\longrightarrow}V_A\overset{u_A}{\longrightarrow}U_A\overset{\rho_A}{\dashrightarrow}
\end{equation}
chosen to define the shift functor $[1]$.

\begin{lemma}\label{Lem50_X1}
Let $A,C\in\mathscr{C}$ be any pair of objects. If $f_1,f_2\in\mathscr{C}(X,U_A)$ satisfy
\[ f_1^{\ast}\rho_A=f_2^{\ast}\rho_A, \]
then $\ell(f_1)=\ell(f_2)$ holds.
\end{lemma}
\begin{proof}
This immediately follows from the exactness of
\[ \mathscr{C}(X,V_A)\overset{u_A\circ-}{\longrightarrow}\mathscr{C}(X,U_A)\overset{(\rho_A)_\sharp}{\longrightarrow}\mathbb{E}(X,A) \]
shown in Proposition~\ref{PropExact1}.
\end{proof}

\begin{definition}\label{Def50_X2}
For any $\mathbb{E}$-extension $\delta\in\mathbb{E}(C,A)$, define $\widetilde{\ell}(\delta)\in\widetilde{\mathbb{E}}(C,A)$ by the following.
\begin{itemize}
\item Take a span of morphisms $(C\overset{w}{\longleftarrow}D\overset{d}{\longrightarrow}U_A)$ from some $D\in\mathscr{C}$, which satisfy
\begin{equation}\label{Eq_Ast_50}
w\in w\mathit{Fib}\quad\text{and}\quad w^{\ast}\delta=d^{\ast}\rho_A.
\end{equation}
Then, define as $\widetilde{\ell}(\delta)=\ell(d)\circ\ell(w)^{-1}.$
\end{itemize}
\end{definition}

\begin{claim}\label{Claim50_X3}
For any $\delta\in\mathbb{E}(C,A)$, the morphism $\widetilde{\ell}(\delta)$ in Definition~\ref{Def50_X2} is well-defined. More precisely, the following holds.
\begin{enumerate}
\item[(1)] Take any $D\in\mathscr{C},\, w\in\mathscr{C}(D,C)$. If both $d_1,d_2\in\mathscr{C}(D,U_A)$ satisfy
\[ w^{\ast}\delta=d_i^{\ast}\rho_A\quad(i=1,2), \]
then $\ell(d_1)=\ell(d_2)$ holds.
\item[(2)] If both spans $(C\overset{w_1}{\longleftarrow}D_1\overset{d_1}{\longrightarrow}U_A)$ and $(C\overset{w_2}{\longleftarrow}D_2\overset{d_2}{\longrightarrow}U_A)$ satisfy $(\ref{Eq_Ast_50})$, then
\[ \ell(d_1)\circ\ell(w_1)^{-1}=\ell(d_2)\circ\ell(w_2)^{-1} \]
holds.
\item[(3)] There exists at least one span $(C\overset{w}{\longleftarrow}D\overset{d}{\longrightarrow}U_A)$ satisfying $(\ref{Eq_Ast_50})$.
\end{enumerate}
\end{claim}
\begin{proof}
{\rm (1)} This immediately follows from Lemma~\ref{Lem50_X1}.

{\rm (2)} Let $V_i\overset{v_i}{\longrightarrow}D_i\overset{w_i}{\longrightarrow}C$ $(i=1,2)$ be conflations. By Proposition~\ref{PropBaer} {\rm (1)}, we have a commutative diagram made of conflations as follows.
\[
\xy
(-7,21)*+{V_2}="-12";
(7,21)*+{V_2}="-14";
(-21,7)*+{V_1}="0";
(-7,7)*+{{}^{\exists}D}="2";
(7,7)*+{D_2}="4";
(-21,-7)*+{V_1}="10";
(-7,-7)*+{D_1}="12";
(7,-7)*+{C}="14";
{\ar@{=} "-12";"-14"};
{\ar_{m_2} "-12";"2"};
{\ar^{v_2} "-14";"4"};
{\ar^{m_1} "0";"2"};
{\ar^{e_1} "2";"4"};
{\ar@{=} "0";"10"};
{\ar_{e_2} "2";"12"};
{\ar^{w_2} "4";"14"};
{\ar_{v_1} "10";"12"};
{\ar_{w_1} "12";"14"};
{\ar@{}|\circlearrowright "-12";"4"};
{\ar@{}|\circlearrowright "0";"12"};
{\ar@{}|\circlearrowright "2";"14"};
\endxy
\]
If we put $w=w_1\circ e_2=w_2\circ e_1$, then we have $w\in w\mathit{Fib}\circ w\mathit{Fib}=w\mathit{Fib}$. If we put $k_1=d_2\circ e_1$ and $k_2=d_1\circ e_2$, then they give
\begin{eqnarray*}
&\ell(k_1)\circ\ell(w)^{-1}=\ell(d_2\circ e_1)\circ\ell(w_2\circ e_1)^{-1}=\ell(d_2)\circ\ell(w_2)^{-1},&\\
&\ell(k_2)\circ\ell(w)^{-1}=\ell(d_1\circ e_2)\circ\ell(w_1\circ e_2)^{-1}=\ell(d_1)\circ\ell(w_1)^{-1}.&
\end{eqnarray*}
Since both $k_1,k_2$ satisfy
\begin{eqnarray*}
&k_1^{\ast}\delta=e_1^{\ast} d_2^{\ast}\delta=e_1^{\ast} w_2^{\ast}\delta=w^{\ast}\delta,&\\
&k_2^{\ast}\delta=e_2^{\ast} d_1^{\ast}\delta=e_2^{\ast} w_1^{\ast}\delta=w^{\ast}\delta,&
\end{eqnarray*}
we obtain $\ell(k_1)=\ell(k_2)$ by {\rm (1)}. Thus it follows that $\ell(d_2)\circ\ell(w_2)^{-1}=\ell(d_1)\circ\ell(w_1)^{-1}$.

{\rm (3)} Realize $\delta$ by an $\mathbb{E}$-triangle
\[ A\overset{x}{\longrightarrow}B\overset{y}{\longrightarrow}C\overset{\delta}{\dashrightarrow}. \]
Then Proposition~\ref{PropBaer} {\rm (2)} gives a commutative diagram made of $\mathbb{E}$-triangles
\[
\xy
(-7,7)*+{A}="2";
(7,7)*+{B}="4";
(21,7)*+{C}="6";
(35,7)*+{}="8";
(-7,-7)*+{V_A}="12";
(7,-7)*+{{}^{\exists}D}="14";
(21,-7)*+{C}="16";
(35,-7)*+{}="18";
(-7,-21)*+{U_A}="22";
(7,-21)*+{U_A}="24";
(-7,-35)*+{}="32";
(7,-35)*+{}="34";
{\ar^{x} "2";"4"};
{\ar^{y} "4";"6"};
{\ar@{-->}^{\delta} "6";"8"};
{\ar_{v_A} "2";"12"};
{\ar^{} "4";"14"};
{\ar@{=} "6";"16"};
{\ar^{} "12";"14"};
{\ar_{w} "14";"16"};
{\ar@{-->}_{(v_A)_{\ast}\delta} "16";"18"};
{\ar_{u_A} "12";"22"};
{\ar^{e} "14";"24"};
{\ar@{=} "22";"24"};
{\ar@{-->}_{\rho_A} "22";"32"};
{\ar@{-->}^{x_{\ast}\rho_A} "24";"34"};
{\ar@{}|\circlearrowright "2";"14"};
{\ar@{}|\circlearrowright "4";"16"};
{\ar@{}|\circlearrowright "12";"24"};
\endxy
\]
satisfying $w^{\ast}\delta+e^{\ast}\rho_A=0$. Thus the span $(C\overset{w}{\longleftarrow}D\overset{-e}{\longrightarrow}U_A)$ satisfies $(\ref{Eq_Ast_50})$.
\end{proof}

\begin{proposition}\label{Prop50_X4}
For any $A,C\in\mathscr{C}$, the map $\widetilde{\ell}\colon\mathbb{E}(C,A)\to\widetilde{\mathbb{E}}(C,A)$ is an additive homomorphism.
\end{proposition}
\begin{proof}
Let $\delta_1,\delta_2\in\mathbb{E}(C,A)$ be any pair of elements. By {\rm (3)} of Claim~\ref{Claim50_X3}, we can find spans $(C\overset{w_1}{\longleftarrow}D_1\overset{d_1}{\longrightarrow}U_A),\, (C\overset{w_2}{\longleftarrow}D_2\overset{d_2}{\longrightarrow}U_A)$ which give $\widetilde{\ell}(\delta_1),\, \widetilde{\ell}(\delta_2)$. As in the proof of {\rm (2)} in Claim~\ref{Claim50_X3}, replacing $(w_i,D_i)$ by a common $(w,D)$, we may assume
\[ D_1=D_2=D \quad\text{and}\quad w_1=w_2=w\]
from the start.
Then by $w^{\ast}\delta_1=d_1^{\ast}\rho_A$ and $w^{\ast}\delta_2=d_2^{\ast}\rho_A$, we have
\[ w^{\ast}(\delta_1+\delta_2)=(d_1+d_2)^{\ast}\rho_A. \]
Thus the span $(C\overset{w}{\longleftarrow}D\overset{d_1+d_2}{\longrightarrow}U_A)$ satisfies $(\ref{Eq_Ast_50})$ for $\delta_1+\delta_2$. This shows
\begin{eqnarray*}
\widetilde{\ell}(\delta_1+\delta_2)&=&\ell(d_1+d_2)\circ\ell(w)^{-1}\\
&=&\ell(d_1)\circ\ell(w)^{-1}+\ell(d_2)\circ\ell(w)^{-1}\ =\ \widetilde{\ell}(\delta_1)+\widetilde{\ell}(\delta_2).
\end{eqnarray*}
\end{proof}

\begin{lemma}\label{Lem35b}
For any $U\in\mathcal{U}$ and $T\in\mathcal{T}$, the map
\[ \widetilde{\ell}\colon\mathbb{E}(U,T)\to\widetilde{\mathbb{E}}(U,T) \]
is monomorphic.
\end{lemma}
\begin{proof}
Let $\delta\in\mathbb{E}(U,T)$ be any $\mathbb{E}$-extension. Realize it by an $\mathbb{E}$-triangle
\[ T\overset{x}{\longrightarrow}A\overset{y}{\longrightarrow}U\overset{\delta}{\dashrightarrow}. \]
Let $T\overset{v_T}{\longrightarrow}V_T\overset{u_T}{\longrightarrow}U_T\overset{\rho_T}{\dashrightarrow}$ be the chosen $\mathbb{E}$-triangle, as before. By Proposition~\ref{PropBaer} {\rm (2)}, we obtain a diagram made of $\mathbb{E}$-triangles
\[
\xy
(-7,7)*+{T}="2";
(7,7)*+{A}="4";
(21,7)*+{U}="6";
(35,7)*+{}="8";
(-7,-7)*+{V_T}="12";
(7,-7)*+{M}="14";
(21,-7)*+{U}="16";
(35,-7)*+{}="18";
(-7,-21)*+{U_T}="22";
(7,-21)*+{U_T}="24";
(-7,-35)*+{}="32";
(7,-35)*+{}="34";
{\ar^{x} "2";"4"};
{\ar^{y} "4";"6"};
{\ar@{-->}^{\delta} "6";"8"};
{\ar_{v_T} "2";"12"};
{\ar^{m} "4";"14"};
{\ar@{=} "6";"16"};
{\ar^{m^{\prime}} "12";"14"};
{\ar_{e^{\prime}} "14";"16"};
{\ar@{-->}_{(v_T)_{\ast}\delta} "16";"18"};
{\ar_{u_T} "12";"22"};
{\ar^{e} "14";"24"};
{\ar@{=} "22";"24"};
{\ar@{-->}_{\rho_T} "22";"32"};
{\ar@{-->}^{x_{\ast}(\rho_T)} "24";"34"};
{\ar@{}|\circlearrowright "2";"14"};
{\ar@{}|\circlearrowright "4";"16"};
{\ar@{}|\circlearrowright "12";"24"};
\endxy
\]
satisfying $e^{\ast}(\rho_T)+e^{\prime\ast}\delta=0$. As the proof of {\rm (3)} of Claim~\ref{Claim50_X3} suggests, we have
\[ \widetilde{\ell}(\delta)=-\ell(e)\circ\ell(e^{\prime})^{-1}. \]

By $\mathbb{E}(U,V_T)=0$, we have $(v_T)_{\ast}\delta=0$. Thus, replacing $M$ by an isomorphic object, we may assume
\[ M=V_T\oplus U,\ \ m^{\prime}=\iota_{V_T}\colon V_T\to M,\ \ e^{\prime}=p_U\colon M\to U, \]
where
\[
\xy
(-18,0)*+{V_T}="-2";
(-15.5,1)*+{}="-10";
(-15.5,-1)*+{}="-11";
(-6.5,1)*+{}="-0";
(-6.5,-1)*+{}="-1";
(0,0)*+{V_T\oplus U}="0";
(15.5,1)*+{}="10";
(15.5,-1)*+{}="11";
(6.5,1)*+{}="+0";
(6.5,-1)*+{}="1";
(18,0)*+{U}="2";
{\ar^{\iota_{V_T}} "-10";"-0"};
{\ar^{p_{V_T}} "-1";"-11"};
{\ar@{<-}^{\iota_U} "+0";"10"};
{\ar@{<-}^{p_U} "11";"1"};
\endxy
\]
is a biproduct.
Put $q=-e\circ\iota_U\colon U\to U_T$. Then we have $e=u_T\circ p_{V_T}-q\circ e^{\prime}$. Since $\ell(p_{V_T})=0$, this gives $\ell(e)=-\ell(q)\circ\ell(e^{\prime})$, namely
\begin{equation}\label{Eq_Airport1}
\widetilde{\ell}(\delta)=\ell(q).
\end{equation}
On the other hand, since the morphism $e^{\prime\ast}$ is a monomorphism, the equality
\[ e^{\prime\ast}\delta=-e^{\ast}(\rho_T)=-p_{V_T}^{\ast} u_T^{\ast}(\rho_T)+e^{\prime\ast}q^{\ast}(\rho_T)=e^{\prime\ast}q^{\ast}(\rho_T) \]
implies
\begin{equation}\label{Eq_Airport2}
\delta=q^{\ast}(\rho_T).
\end{equation}
By $(\ref{Eq_Airport1})$ and $(\ref{Eq_Airport2})$, it suffices to show
\[ \ell(q)=0\ \Rightarrow\ q^{\ast}(\rho_T)=0. \]

Assume $\ell(q)=0$.
Take $\mathbb{E}$-triangles
\begin{eqnarray*}
&U_T\overset{z}{\longrightarrow}Z\to S\dashrightarrow\ \ (Z\in\mathcal{Z},\, S\in\mathcal{S}),&\\
&T^{\prime}\to I\overset{i}{\longrightarrow}Z\dashrightarrow\ \ (I\in\mathcal{I},\, T^{\prime}\in\mathcal{T}).&
\end{eqnarray*}
Then by $\ell(z\circ q)=\ell(z)\circ0=0$, the morphism $z\circ q\in\mathscr{C}(U,Z)$ factors through $i$ by Remark~\ref{RemUTMono}. Namely, there exists $k\in\mathscr{C}(U,I)$ which makes the following diagram commutative.
\[
\xy
(-7,6)*+{U}="0";
(7,6)*+{U_T}="2";
(-7,-6)*+{I}="4";
(7,-6)*+{Z}="6";
{\ar^{q} "0";"2"};
{\ar_{k} "0";"4"};
{\ar^{z} "2";"6"};
{\ar_{i} "4";"6"};
{\ar@{}|\circlearrowright "0";"6"};
\endxy
\]
By {\rm (ET4)$^{\mathrm{op}}$}, we obtain a diagram made of conflations
\[
\xy
(-21,7)*+{T^{\prime}}="0";
(-7,7)*+{{}^{\exists}N}="2";
(7,7)*+{U_T}="4";
(-21,-7)*+{T^{\prime}}="10";
(-7,-7)*+{I}="12";
(7,-7)*+{Z}="14";
(-7,-21)*+{S}="22";
(7,-21)*+{S}="24";
{\ar^{} "0";"2"};
{\ar^{{}^{\exists}p} "2";"4"};
{\ar@{=} "0";"10"};
{\ar_{{}^{\exists}n} "2";"12"};
{\ar^{z} "4";"14"};
{\ar_{} "10";"12"};
{\ar_{i} "12";"14"};
{\ar_{} "12";"22"};
{\ar^{} "14";"24"};
{\ar@{=} "22";"24"};
{\ar@{}|\circlearrowright "0";"12"};
{\ar@{}|\circlearrowright "2";"14"};
{\ar@{}|\circlearrowright "12";"24"};
\endxy
\]
in which, $N$ belongs to $\mathcal{N}$ by Lemma~\ref{LemConcentric1} {\rm (2)}.

By the dual of Lemma~\ref{LemHomotPush}, we obtain $j\in\mathscr{C}(U,N)$ which makes the following diagram commutative.
\[
\xy
(-18,18)*+{U}="-2";
(-7,7)*+{N}="0";
(7,7)*+{U_T}="2";
(-7,-7)*+{I}="4";
(7,-7)*+{Z}="6";
(-3,15)*+{}="5";
(-15,3)*+{}="7";
{\ar_{p} "0";"2"};
{\ar^{n} "0";"4"};
{\ar^{z} "2";"6"};
{\ar_{i} "4";"6"};
{\ar_{j} "-2";"0"};
{\ar@/^0.60pc/^{q} "-2";"2"};
{\ar@/_0.60pc/_{k} "-2";"4"};
{\ar@{}|\circlearrowright "0";"6"};
{\ar@{}|\circlearrowright "0";"5"};
{\ar@{}|\circlearrowright "0";"7"};
\endxy
\]
By $N\in\mathcal{N}=\mathrm{Cone}(\mathcal{V},\mathcal{S})$ and $\mathbb{E}(U,\mathcal{V})=0$, this $j$ factors through some $S_0\in\mathcal{S}$, as follows.
\[
\xy
(-11,11)*+{U}="0";
(-15.5,-5.5)*+{S_0}="2";
(-16,5)*+{}="3";
(0,0)*+{N}="4";
(5,8)*+{}="5";
(16,0)*+{U_T}="6";
{\ar^{} "0";"2"};
{\ar^{j} "0";"4"};
{\ar_{{}^{\exists}r} "2";"4"};
{\ar_{p} "4";"6"};
{\ar@/^0.70pc/^{q} "0";"6"};
{\ar@{}|\circlearrowright "3";"4"};
{\ar@{}|\circlearrowright "4";"5"};
\endxy
\]
By $\mathbb{E}(S_0,T)=0$, the morphism $p\circ r$ factors through $u_T$. This implies $q^{\ast}(\rho_T)=0$ by Lemma~\ref{LemZero}.
\end{proof}

\begin{lemma}\label{LemMorphExtExt}
Let $(A,\delta,C), (A^{\prime},\delta^{\prime},C^{\prime})$ be $\mathbb{E}$-extensions, and let $(a,c)\colon\delta\to\delta^{\prime}$ be a morphism of $\mathbb{E}$-triangles. Then
\[ (\ell(a),\ell(c))\colon\widetilde{\ell}(\delta)\to\widetilde{\ell}(\delta^{\prime}) \]
is a morphism of $\widetilde{\mathbb{E}}$-extensions. Namely,
\[
\xy
(-8,6)*+{C}="0";
(8,6)*+{U_A}="2";
(17,6)*+{=A[1]}="3";
(-8,-6)*+{C^{\prime}}="4";
(8,-6)*+{U_{A^{\prime}}}="6";
(17,-6)*+{=A^{\prime}[1]}="7";
{\ar^{\widetilde{\ell}(\delta)} "0";"2"};
{\ar_{\ell(c)} "0";"4"};
{\ar^{\ell(a)[1]} "2";"6"};
{\ar_{\widetilde{\ell}(\delta^{\prime})} "4";"6"};
{\ar@{}|\circlearrowright "0";"6"};
\endxy
\]
is commutative in $\widetilde{\mathscr{C}}$.
\end{lemma}
\begin{proof}
Take spans $(C\overset{w}{\longleftarrow}D\overset{d}{\longrightarrow}U_A)$ and $(C^{\prime}\overset{w^{\prime}}{\longleftarrow}D^{\prime}\overset{d^{\prime}}{\longrightarrow}U_{A^{\prime}})$ satisfying
\[ w,w^{\prime}\in w\mathit{Fib},\ \ w^{\ast}\delta=d^{\ast}\rho_A,\ \ w^{\prime\ast}\delta^{\prime}=d^{\prime\ast}\rho_{A^{\prime}}. \]
By definition, we have
\[ \widetilde{\ell}(\delta)=\ell(d)\circ\ell(w)^{-1},\ \ \widetilde{\ell}(\delta^{\prime})=\ell(d^{\prime})\circ\ell(w^{\prime})^{-1}. \]
Remark that $\ell(a)[1]=\ell(u)$ is given by a morphism of $\mathbb{E}$-triangles
\[
\xy
(-12,6)*+{A}="0";
(0,6)*+{V_A}="2";
(12,6)*+{U_A}="4";
(24,6)*+{}="6";
(-12,-6)*+{A^{\prime}}="10";
(0,-6)*+{V_{A^{\prime}}}="12";
(12,-6)*+{U_{A^{\prime}}}="14";
(24,-6)*+{}="16";
{\ar^{v_A} "0";"2"};
{\ar^{u_A} "2";"4"};
{\ar@{-->}^{\rho_A} "4";"6"};
{\ar_{a} "0";"10"};
{\ar^{v} "2";"12"};
{\ar^{u} "4";"14"};
{\ar_{v_{A^{\prime}}} "10";"12"};
{\ar_{u_{A^{\prime}}} "12";"14"};
{\ar@{-->}_{\rho_{A^{\prime}}} "14";"16"};
{\ar@{}|\circlearrowright "0";"12"};
{\ar@{}|\circlearrowright "2";"14"};
\endxy.
\]
Since $w^{\prime}\in w\mathit{Fib}$, there exists an $\mathbb{E}$-triangle
\[ V^{\prime}\to D^{\prime}\overset{w^{\prime}}{\longrightarrow}C^{\prime}\overset{\nu^{\prime}}{\dashrightarrow}\quad(V^{\prime}\in\mathcal{V}). \]By realizing $c^{\ast}\nu^{\prime}$, we obtain a morphism of $\mathbb{E}$-triangles
\[
\xy
(-12,6)*+{V^{\prime}}="0";
(0,6)*+{{}^{\exists}D^{\prime\prime}}="2";
(12,6)*+{C}="4";
(24,6)*+{}="6";
(-12,-6)*+{V^{\prime}}="10";
(0,-6)*+{D^{\prime}}="12";
(12,-6)*+{C^{\prime}}="14";
(24,-6)*+{}="16";
{\ar^{} "0";"2"};
{\ar^{{}^{\exists}w^{\prime\prime}} "2";"4"};
{\ar@{-->}^{c^{\ast}\nu^{\prime}} "4";"6"};
{\ar@{=} "0";"10"};
{\ar_{{}^{\exists}f} "2";"12"};
{\ar^{c} "4";"14"};
{\ar_{} "10";"12"};
{\ar_{w^{\prime}} "12";"14"};
{\ar@{-->}_{\nu^{\prime}} "14";"16"};
{\ar@{}|\circlearrowright "0";"12"};
{\ar@{}|\circlearrowright "2";"14"};
\endxy.
\]
Then both spans $(C\overset{w}{\longleftarrow}D\overset{u\circ d}{\longrightarrow}U_{A^{\prime}})$ and $(C\overset{w^{\prime\prime}}{\longleftarrow}D^{\prime\prime}\overset{d^{\prime}\circ f}{\longrightarrow}U_{A^{\prime}})$ satisfy
\begin{eqnarray*}
w^{\ast}(c^{\ast}\delta^{\prime})&=&w^{\ast} a_{\ast}\delta\ =\ a_{\ast} w^{\ast}\delta\\
&=&a_{\ast} d^{\ast}\rho_A\ =\ d^{\ast} u^{\ast}\rho_{A^{\prime}}\ =\ (u\circ d)^{\ast}\rho_{A^{\prime}},
\end{eqnarray*}
\[ w^{\prime\prime\ast}(c^{\ast}\delta^{\prime})=f^{\ast} w^{\prime\ast}\delta^{\prime}\ =\ f^{\ast} d^{\prime\ast}\rho_{A^{\prime}}=(d^{\prime}\circ f)^{\ast}\rho_{A^{\prime}}. \]
Thus by Claim~\ref{Claim50_X3} {\rm (2)}, we obtain
\begin{eqnarray*}
\ell(u\circ d)\circ\ell(w)^{-1}&=&\ell(d^{\prime}\circ f)\circ\ell(w^{\prime\prime})^{-1}\\
&=&\ell(d^{\prime})\circ(\ell(f)\circ\ell(w^{\prime\prime})^{-1})\ =\ \ell(d^{\prime})\circ(\ell(w^{\prime})^{-1}\circ\ell(c)),
\end{eqnarray*}
which means $\ell(u)\circ\widetilde{\ell}(\delta)=\widetilde{\ell}(\delta^{\prime})\circ\ell(c)$.
\end{proof}

\begin{proposition}\label{Prop50_X7}
The functor $[1]\colon\widetilde{\mathscr{C}}\to\widetilde{\mathscr{C}}$ is an auto-equivalence, with quasi-inverse $[-1]$.
\end{proposition}
\begin{proof}
By the definitions of $[-1]$ and $[1]$, there are $\mathbb{E}$-triangles
\begin{eqnarray*}
&C[-1]\to S^C\to C\overset{\lambda^C}{\dashrightarrow}\quad(S^C\in\mathcal{S}),&\\
&C[-1]\to V_{C[-1]}\to (C[-1])[1]\overset{\rho_{C[-1]}}{\dashrightarrow}\quad(V_{C[-1]}\in\mathcal{V})&
\end{eqnarray*}
for each $C\in\mathscr{C}$. Then by Proposition~\ref{PropBaer} {\rm (2)}, we have a commutative diagram made of conflations
\[
\xy
(-10,7)*+{C[-1]}="2";
(11,7)*+{S^C}="4";
(26,7)*+{C}="6";
(-10,-7)*+{V_{C[-1]}}="12";
(11,-7)*+{{}^{\exists}D}="14";
(26,-7)*+{C}="16";
(-10,-21)*+{(C[-1])[1]}="22";
(11,-21)*+{(C[-1])[1]}="24";
{\ar^(0.42){} "2";"4"};
{\ar^(0.58){} "4";"6"};
{\ar_{} "2";"12"};
{\ar^{} "4";"14"};
{\ar@{=} "6";"16"};
{\ar^{} "12";"14"};
{\ar^{w} "14";"16"};
{\ar_{} "12";"22"};
{\ar^{e} "14";"24"};
{\ar@{=} "22";"24"};
{\ar@{}|\circlearrowright "2";"14"};
{\ar@{}|\circlearrowright "4";"16"};
{\ar@{}|\circlearrowright "12";"24"};
\endxy
\]
which gives $\widetilde{\ell}(\lambda^C)=-\ell(e)\circ\ell(w)^{-1}$ as in the proof of {\rm (3)} in Claim~\ref{Claim50_X3}.
Since $e\in\mathbb{W}$, it follows that $\widetilde{\ell}(\lambda^C)$ is an isomorphism in $\widetilde{\mathscr{C}}$. Let us show the naturality of
\[ \{\widetilde{\ell}(\lambda^C)\colon C\to (C[-1])[1]\}_{C\in\widetilde{\mathscr{C}}}. \]

For this purpose, it suffices to show the naturality with respect to the morphisms in $\mathscr{C}$.
For any morphism $f\in\mathscr{C}(C,C^{\prime})$, the morphism $\ell(f)[-1]=\ell(g)\in\widetilde{\mathscr{C}}(C[-1],C^{\prime}[-1])$ is given by a morphism of $\mathbb{E}$-extensions $(g,f)\colon \lambda^C\to\lambda^{C^{\prime}}$, dually to Definition~\ref{DefShiftLoc}. Thus by Lemma~\ref{LemMorphExtExt},
\[
\xy
(-10,6)*+{C}="0";
(10,6)*+{(C[-1])[1]}="2";
(-10,-6)*+{C^{\prime}}="4";
(10,-6)*+{(C^{\prime}[-1])[1]}="6";
{\ar^(0.4){\widetilde{\ell}(\lambda^C)} "0";"2"};
{\ar_{\ell(f)} "0";"4"};
{\ar^{\ell(g)[1]=(\ell(f)[-1])[1]} "2";"6"};
{\ar_(0.4){\widetilde{\ell}(\lambda^{C^{\prime}})} "4";"6"};
{\ar@{}|\circlearrowright "0";"6"};
\endxy
\]
becomes commutative. This shows $[1]\circ [-1]\cong\mathrm{Id}$.
The isomorphism $[-1]\circ [1]\cong\mathrm{Id}$ can be shown dually.
\end{proof}

\subsection{Triangulation}

\begin{definition}\label{DefTriaLoc}
For an $\mathbb{E}$-triangle $A\overset{x}{\longrightarrow}B\overset{y}{\longrightarrow}C\overset{\delta}{\dashrightarrow}$, its associated {\it standard triangle} in $\widetilde{\mathscr{C}}$ is defined to be
\[ A\overset{\ell(x)}{\longrightarrow}B\overset{\ell(y)}{\longrightarrow}C\overset{\widetilde{\ell}(\delta)}{\longrightarrow}A[1]. \]
A {\it distinguished triangle} in $\widetilde{\mathscr{C}}$ is a triangle isomorphic to some standard triangle.
\end{definition}

\begin{proposition}\label{PropMorphExtExt}
Any morphism of $\mathbb{E}$-triangles
\[
\xy
(-12,6)*+{A}="0";
(0,6)*+{B}="2";
(12,6)*+{C}="4";
(24,6)*+{}="6";
(-12,-6)*+{A^{\prime}}="10";
(0,-6)*+{B^{\prime}}="12";
(12,-6)*+{C^{\prime}}="14";
(24,-6)*+{}="16";
{\ar^{x} "0";"2"};
{\ar^{y} "2";"4"};
{\ar@{-->}^{\delta} "4";"6"};
{\ar_{a} "0";"10"};
{\ar^{b} "2";"12"};
{\ar^{c} "4";"14"};
{\ar_{x^{\prime}} "10";"12"};
{\ar_{y^{\prime}} "12";"14"};
{\ar@{-->}_{\delta^{\prime}} "14";"16"};
{\ar@{}|\circlearrowright "0";"12"};
{\ar@{}|\circlearrowright "2";"14"};
\endxy
\]
gives the following morphism between standard triangles.
\[
\xy
(-18,6)*+{A}="0";
(-4,6)*+{B}="2";
(12,6)*+{C}="4";
(28,6)*+{A[1]}="6";
(-18,-6)*+{A^{\prime}}="10";
(-4,-6)*+{B^{\prime}}="12";
(12,-6)*+{C^{\prime}}="14";
(28,-6)*+{A^{\prime}[1]}="16";
{\ar^{\ell(x)} "0";"2"};
{\ar^{\ell(y)} "2";"4"};
{\ar^{\ell(\delta)} "4";"6"};
{\ar_{\ell(a)} "0";"10"};
{\ar^{\ell(b)} "2";"12"};
{\ar^{\ell(c)} "4";"14"};
{\ar^{\ell(a)[1]} "6";"16"};
{\ar_{\ell(x^{\prime})} "10";"12"};
{\ar_{\ell(y^{\prime})} "12";"14"};
{\ar_{\ell(\delta^{\prime})} "14";"16"};
{\ar@{}|\circlearrowright "0";"12"};
{\ar@{}|\circlearrowright "2";"14"};
{\ar@{}|\circlearrowright "4";"16"};
\endxy
\]
\end{proposition}
\begin{proof}
This immediately follows from Lemma~\ref{LemMorphExtExt}.
\end{proof}

This gives a cofibrant replacement of  a standard triangle, as follows.
\begin{corollary}\label{CorCofibReplStan}
Assume $(\mathscr{C},\mathbb{E},\mathfrak{s})$ satisfies Condition~(WIC) as before. 
Any standard triangle is isomorphic to a standard triangle associated to an $\mathbb{E}$-triangle $U\to U^{\prime}\to U^{\prime\prime}\dashrightarrow$ whose terms satisfy $U,U^{\prime},U^{\prime\prime}\in\mathcal{U}$. 
\end{corollary}
\begin{proof}
Let $A\overset{x}{\longrightarrow}B\overset{y}{\longrightarrow}C\overset{\delta}{\dashrightarrow}$ be any $\mathbb{E}$-triangle. Resolve $A$ by an  $\mathbb{E}$-triangle $V\overset{v}{\longrightarrow}U\overset{a}{\longrightarrow}A\overset{\lambda}{\dashrightarrow}$ satisfying $U\in\mathcal{U}$ and $V\in\mathcal{V}$.
By Proposition~\ref{PropFactor}, we have $x\circ a\in w\mathit{Fib}\circ\mathit{Cof}$. Namely, there are $\mathbb{E}$-triangles
\[ U\overset{x^{\prime}}{\longrightarrow}B^{\prime}\overset{y^{\prime}}{\longrightarrow}U_0\overset{\delta^{\prime}}{\dashrightarrow}\ \ \text{and}\ \ V_0\overset{v^{\prime}}{\longrightarrow}B^{\prime}\overset{b}{\longrightarrow}B\overset{}{\dashrightarrow} \]
satisfying $U_0\in\mathcal{U},\, V_0\in\mathcal{V}$ and $x\circ a=b\circ x^{\prime}$.

Since $\mathcal{U}$ is extension-closed, it follows that $B^{\prime}\in\mathcal{U}$. Moreover, by Lemma~\ref{LemNine} and Lemma~\ref{LemConcentric1} {\rm (1)}, we obtain a morphism of $\mathbb{E}$-triangles
\[
\xy
(-12,6)*+{U}="0";
(0,6)*+{B^{\prime}}="2";
(12,6)*+{U_0}="4";
(24,6)*+{}="6";
(-12,-6)*+{A}="10";
(0,-6)*+{B}="12";
(12,-6)*+{C}="14";
(24,-6)*+{}="16";
{\ar^{x^{\prime}} "0";"2"};
{\ar^{y^{\prime}} "2";"4"};
{\ar@{-->}^{\delta^{\prime}} "4";"6"};
{\ar_{a} "0";"10"};
{\ar^{b} "2";"12"};
{\ar^{c} "4";"14"};
{\ar_{x} "10";"12"};
{\ar_{y} "12";"14"};
{\ar@{-->}_{\delta} "14";"16"};
{\ar@{}|\circlearrowright "0";"12"};
{\ar@{}|\circlearrowright "2";"14"};
\endxy
\]
which admits an $\mathbb{E}$-triangle
\[ N\to U_0\overset{c}{\longrightarrow}C\dashrightarrow \]
with some $N\in\mathcal{N}$. By Claim~\ref{ClaimNW}, it follows that $c\in\mathbb{W}$.

By Proposition~\ref{PropMorphExtExt}, we obtain an isomorphism of standard triangles
\[
\xy
(-18,6)*+{U}="0";
(-4,6)*+{B^{\prime}}="2";
(12,6)*+{U_0}="4";
(28,6)*+{U[1]}="6";
(-18,-6)*+{A}="10";
(-4,-6)*+{B}="12";
(12,-6)*+{C}="14";
(28,-6)*+{A[1]}="16";
{\ar^{\ell(x^{\prime})} "0";"2"};
{\ar^{\ell(y^{\prime})} "2";"4"};
{\ar^{\widetilde{\ell}(\delta^{\prime})} "4";"6"};
{\ar_{\ell(a)}^{\cong} "0";"10"};
{\ar^{\ell(b)}_{\cong} "2";"12"};
{\ar^{\ell(c)}_{\cong} "4";"14"};
{\ar^{\ell(a)[1]}_{\cong} "6";"16"};
{\ar_{\ell(x)} "10";"12"};
{\ar_{\ell(y)} "12";"14"};
{\ar_{\widetilde{\ell}(\delta)} "14";"16"};
{\ar@{}|\circlearrowright "0";"12"};
{\ar@{}|\circlearrowright "2";"14"};
{\ar@{}|\circlearrowright "4";"16"};
\endxy
\]
in $\widetilde{\mathscr{C}}$.
\end{proof}

\begin{remark}\label{RemFibReplStan}
Similarly, for any $\mathbb{E}$-triangle $A\overset{x}{\longrightarrow}B\overset{y}{\longrightarrow}C\overset{\delta}{\dashrightarrow}$, we can construct a morphism of $\mathbb{E}$-triangles
\[
\xy
(-21,6)*+{A}="0";
(-7,6)*+{B}="2";
(7,6)*+{C}="4";
(20,6)*+{}="6";
(-21,-6)*+{T_A}="10";
(-7,-6)*+{T_B}="12";
(7,-6)*+{T_C}="14";
(20,-6)*+{}="16";
{\ar^{x} "0";"2"};
{\ar^{y} "2";"4"};
{\ar@{-->}^{\delta} "4";"6"};
{\ar_{a} "0";"10"};
{\ar^{b} "2";"12"};
{\ar^{c} "4";"14"};
{\ar_{} "10";"12"};
{\ar_{} "12";"14"};
{\ar@{-->}_{} "14";"16"};
{\ar@{}|\circlearrowright "0";"12"};
{\ar@{}|\circlearrowright "2";"14"};
\endxy
\]
which satisfies $T_A,T_B,T_C\in\mathcal{T}$ and $a,b,c\in\mathbb{W}$.
\end{remark}

\begin{lemma}\label{LemReplaceTriaLoc}
Let
\begin{equation}\label{Rep1}
A\overset{x}{\longrightarrow}B\overset{y}{\longrightarrow}C\overset{\delta}{\dashrightarrow}
\end{equation}
be any $\mathbb{E}$-triangle. From the $\mathbb{E}$-triangle $A\overset{v_A}{\longrightarrow}V_A\overset{u_A}{\longrightarrow}U_A\overset{\rho_A}{\dashrightarrow}$, we obtain an $\mathbb{E}$-triangle
\begin{equation}\label{Rep2}
A\overset{\Big[\raise1ex\hbox{\leavevmode\vtop{\baselineskip-8ex \lineskip1ex \ialign{#\crcr{$x$}\crcr{$v_A$}\crcr}}}\Big]}{\longrightarrow}B\oplus V_A\to X\overset{\theta}{\dashrightarrow}
\end{equation}
by Corollary~\ref{CorTrivWIC}. Then the standard triangles associated to $(\ref{Rep1}), (\ref{Rep2})$ are isomorphic. Also remark that we have $V_A\in\mathcal{I}$ if $A\in\mathcal{U}$.
\end{lemma}
\begin{proof}
By the dual of Proposition~\ref{PropShiftOctahedron}, we obtain a commutative diagram made of $\mathbb{E}$-triangles, as follows.
\[
\xy
(-8,7)*+{V_A}="2";
(8,7)*+{V_A}="4";
(-24,-7)*+{A}="10";
(-8,-7)*+{B\oplus V_A}="12";
(8,-7)*+{X}="14";
(23,-7)*+{}="16";
(-24,-21)*+{A}="20";
(-8,-21)*+{B}="22";
(8,-21)*+{C}="24";
(23,-21)*+{}="26";
(-8,-35)*+{}="32";
(8,-35)*+{}="34";
{\ar@{=} "2";"4"};
{\ar "2";"12"};
{\ar "4";"14"};
{\ar "10";"12"};
(-18,-2.5)*+{\Big[\ \ \Big]}="-1";
(-18,-0.5)*+{\scriptstyle{x}};
(-18,-4.5)*+{\scriptstyle{v_A}};
{\ar^{} "12";"14"};
{\ar@{-->}^{\theta} "14";"16"};
{\ar@{=} "10";"20"};
{\ar^{[1\ 0]} "12";"22"};
{\ar^{{}^{\exists}e} "14";"24"};
{\ar_{x} "20";"22"};
{\ar_{y} "22";"24"};
{\ar@{-->}_{\delta} "24";"26"};
{\ar@{-->}_{0} "22";"32"};
{\ar@{-->}^{0} "24";"34"};
{\ar@{}|\circlearrowright "2";"14"};
{\ar@{}|\circlearrowright "10";"22"};
{\ar@{}|\circlearrowright "12";"24"};
\endxy
\]
Remark that $\ell([1\ \, 0])\colon B\to B\oplus V_A$ and $\ell(e)\colon C\to Z$ are isomorphisms in $\widetilde{\mathscr{C}}$.
Thus Lemma~\ref{LemReplaceTriaLoc} follows from Proposition~\ref{PropMorphExtExt}.
\end{proof}

\begin{theorem}\label{ThmTriaLoc}
The shift functor in Definition~\ref{DefShiftLoc} and the class of distinguished triangles in Definition~\ref{DefTriaLoc} give a triangulation of $\widetilde{\mathscr{C}}$.
\end{theorem}
\begin{proof}
{\rm (TR1)} By definition, the class of distinguished triangles is closed under isomorphisms. From the $\mathbb{E}$-triangle $A\overset{\mathrm{id}_A}{\longrightarrow}A\to 0\dashrightarrow$, we obtain a distinguished triangle $A\overset{\mathrm{id}_A}{\longrightarrow}A\to 0\to A[1]$.

Let $\alpha\in\widetilde{\mathscr{C}}(A,B)$ be any morphism, and let us show the existence of a distinguished triangle of the form
\[ A\overset{\alpha}{\longrightarrow}B\to C\to A[1]. \]
Up to isomorphism in $\widetilde{\mathscr{C}}$, we may assume $A\in\mathcal{U},B\in\mathcal{T}$ from the start. As in Remark~\ref{RemUTMono}, then there is a morphism $f\in\mathscr{C}(A,B)$ satisfying $\ell(f)=\alpha$.

By Corollary~\ref{CorTrivWIC}, we have an $\mathbb{E}$-triangle
\[ A\overset{\Big[\raise1ex\hbox{\leavevmode\vtop{\baselineskip-8ex \lineskip1ex \ialign{#\crcr{$f$}\crcr{$v_A$}\crcr}}}\Big]}{\longrightarrow}B\oplus V_A\overset{g}{\longrightarrow}C\overset{\delta}{\dashrightarrow}, \]
which gives a standard triangle compatible with $\ell(f)$ as follows.
\[
\xy
(-27,0)*+{A}="0";
(-8,0)*+{B\oplus V_A}="2";
(-20,-9)*+{}="3";
(8,0)*+{C}="4";
(24,0)*+{A[1]}="6";
(-8,-12)*+{B}="12";
{\ar "0";"2"};
(-20,4.5)*+{\Big[\ \ \ \ \ \Big]}="-1";
(-20,6.5)*+{\scriptstyle{\ell(f)}};
(-20,2.5)*+{\scriptstyle{\ell(v_A)}};
{\ar^(0.54){\ell(g)} "2";"4"};
{\ar^(0.46){\widetilde{\ell}(\delta)} "4";"6"};
{\ar_{\ell(f)} "0";"12"};
{\ar^{\cong} "2";"12"};
{\ar@{}|\circlearrowright "2";"3"};
\endxy
\]

{\rm (TR2)} It suffices to show this axiom for standard triangles. Let $A\overset{x}{\longrightarrow}B\overset{y}{\longrightarrow}C\overset{\delta}{\dashrightarrow}$ be an $\mathbb{E}$-triangle, and let $A\overset{\ell(x)}{\longrightarrow}B\overset{\ell(y)}{\longrightarrow}C\overset{\widetilde{\ell}(\delta)}{\longrightarrow}A[1]$ be its associated standard triangle. Let us show that
\[ B\overset{\ell(y)}{\longrightarrow}C\overset{\widetilde{\ell}(\delta)}{\longrightarrow}A[1]\overset{-\ell(x)[1]}{\longrightarrow}B[1] \]
is distinguished. By Proposition~\ref{PropBaer} {\rm (2)}, we obtain a commutative diagram made of $\mathbb{E}$-triangles
\[
\xy
(-7,7)*+{A}="2";
(7,7)*+{B}="4";
(21,7)*+{C}="6";
(35,7)*+{}="8";
(-7,-7)*+{V_A}="12";
(7,-7)*+{{}^{\exists}D}="14";
(21,-7)*+{C}="16";
(35,-7)*+{}="18";
(-7,-21)*+{U_A}="22";
(7,-21)*+{U_A}="24";
(-7,-35)*+{}="32";
(7,-35)*+{}="34";
{\ar^{x} "2";"4"};
{\ar^{y} "4";"6"};
{\ar@{-->}^{\delta} "6";"8"};
{\ar_{v_A} "2";"12"};
{\ar^{m} "4";"14"};
{\ar@{=} "6";"16"};
{\ar^{} "12";"14"};
{\ar_{w} "14";"16"};
{\ar@{-->}_{(v_A)_{\ast}\delta} "16";"18"};
{\ar_{u_A} "12";"22"};
{\ar^{e} "14";"24"};
{\ar@{=} "22";"24"};
{\ar@{-->}_{\rho_A} "22";"32"};
{\ar@{-->}^{x_{\ast}\rho_A} "24";"34"};
{\ar@{}|\circlearrowright "2";"14"};
{\ar@{}|\circlearrowright "4";"16"};
{\ar@{}|\circlearrowright "12";"24"};
\endxy
\]
satisfying $w^{\ast}\delta+e^{\ast}\rho_A=0$.
In particular, we obtain a distinguished triangle
\[ B\overset{\ell(m)}{\longrightarrow}D\overset{\ell(e)}{\longrightarrow}A[1]\overset{\widetilde{\ell}(x_{\ast}\rho_A)}{\longrightarrow}B[1]. \]
Remark that we have $\widetilde{\ell}(\delta)=-\ell(e)\circ\ell(w)^{-1}$, as the proof of {\rm (3)} of Claim~\ref{Claim50_X3} suggests. Thus it remains to show the commutativity of the right-most square of
\[
\xy
(-27,7)*+{B}="0";
(-9,7)*+{D}="2";
(9,7)*+{A[1]}="4";
(27,7)*+{B[1]}="6";
(-27,-7)*+{B}="10";
(-9,-7)*+{C}="12";
(9,-7)*+{A[1]}="14";
(27,-7)*+{B[1]}="16";
{\ar^{\ell(m)} "0";"2"};
{\ar^{\ell(e)} "2";"4"};
{\ar^(0.48){\widetilde{\ell}(x_{\ast}\rho_A)} "4";"6"};
{\ar@{=} "0";"10"};
{\ar^{\ell(w)}_{\cong} "2";"12"};
{\ar^{-1}_{\cong} "4";"14"};
{\ar@{=} "6";"16"};
{\ar_{\ell(y)} "10";"12"};
{\ar_{\widetilde{\ell}(\delta)} "12";"14"};
{\ar_(0.48){-\ell(x)[1]} "14";"16"};
{\ar@{}|\circlearrowright "0";"12"};
{\ar@{}|\circlearrowright "2";"14"};
{\ar@{}|\circlearrowright "4";"16"};
\endxy
\]
in $\widetilde{\mathscr{C}}$.
Let 
\[
\xy
(-21,6)*+{A}="0";
(-7,6)*+{V_A}="2";
(7,6)*+{U_A}="4";
(20,6)*+{}="6";
(-21,-6)*+{B}="10";
(-7,-6)*+{V_B}="12";
(7,-6)*+{U_B}="14";
(20,-6)*+{}="16";
{\ar^{v_A} "0";"2"};
{\ar^{u_A} "2";"4"};
{\ar@{-->}^{\rho_A} "4";"6"};
{\ar_{x} "0";"10"};
{\ar^{v} "2";"12"};
{\ar^{u} "4";"14"};
{\ar_{v_B} "10";"12"};
{\ar_{u_B} "12";"14"};
{\ar@{-->}_{\rho_B} "14";"16"};
{\ar@{}|\circlearrowright "0";"12"};
{\ar@{}|\circlearrowright "2";"14"};
\endxy
\]
be a morphism of $\mathbb{E}$-triangles, which gives $\ell(x)[1]=\ell(u)$. Then, since the span $(U_A\overset{\mathrm{id}}{\longleftarrow}U_A\overset{u}{\longrightarrow}U_B)$ satisfies $\mathrm{id}^{\ast}(x_{\ast}\rho_A)=x_{\ast}\rho_A=u^{\ast}\rho_B$, it follows that $\widetilde{\ell}(x_{\ast}\rho_A)=\ell(u)=\ell(x)[1]$.

{\rm (TR3)} Up to isomorphism of triangles, it suffices to show this axiom for standard triangles. Let
\[ A\overset{x}{\longrightarrow}B\overset{y}{\longrightarrow}C\overset{\delta}{\dashrightarrow},\quad A^{\prime}\overset{x^{\prime}}{\longrightarrow}B^{\prime}\overset{y^{\prime}}{\longrightarrow}C^{\prime}\overset{\delta^{\prime}}{\dashrightarrow}\]
be $\mathbb{E}$-triangles, and suppose we are given a commutative diagram
\begin{equation}\label{CommTR3Comm}
\xy
(-7,6)*+{A}="0";
(7,6)*+{B}="2";
(-7,-6)*+{A^{\prime}}="4";
(7,-6)*+{B^{\prime}}="6";
{\ar^{\ell(x)} "0";"2"};
{\ar_{\alpha} "0";"4"};
{\ar^{\beta} "2";"6"};
{\ar_{\ell(x^{\prime})} "4";"6"};
{\ar@{}|\circlearrowright "0";"6"};
\endxy
\end{equation}
in $\widetilde{\mathscr{C}}$. By Corollary~\ref{CorCofibReplStan} and Remark~\ref{RemFibReplStan}, we may assume $A,B,C\in\mathcal{U}$ and $A^{\prime},B^{\prime},C^{\prime}\in\mathcal{T}$ from the start. In that case, $\alpha$ and $\beta$ can be written as $\alpha=\ell(a),\, \beta=\ell(b)$ for some $a\in\mathscr{C}(A,A^{\prime})$ and $b\in\mathscr{C}(B,B^{\prime})$ by Remark~\ref{RemUTMono}. Moreover, the commutativity of $(\ref{CommTR3Comm})$ means that $b\circ x-x^{\prime}\circ a$ factors through some $I\in\mathcal{I}$. By the exactness of

\[ \mathscr{C}(V_A,I)\overset{-\circ v_A}{\longrightarrow}\mathscr{C}(A,I)\to\mathbb{E}(U_A,I)=0, \]
this shows that there exists $j\in\mathscr{C}(V_A,B^{\prime})$ which makes
\[
\xy
(-8,6)*+{A}="0";
(8,6)*+{B\oplus V_A}="2";
(-8,-6)*+{A^{\prime}}="4";
(8,-6)*+{B^{\prime}}="6";
{\ar^(0.42){x_0} "0";"2"};
{\ar_{a} "0";"4"};
{\ar^{[b\ j]} "2";"6"};
{\ar_{x^{\prime}} "4";"6"};
{\ar@{}|\circlearrowright "0";"6"};
\endxy
\]
commutative in $\mathscr{C}$, where $x_0=\Big[\raise1ex\hbox{\leavevmode\vtop{\baselineskip-8ex \lineskip1ex \ialign{#\crcr{$x$}\crcr{$v_A$}\crcr}}}\Big]$. By Lemma~\ref{LemReplaceTriaLoc}, replacing $A\overset{x}{\longrightarrow}B$ by $A\overset{x_0}{\longrightarrow}B\oplus V_A$, we may assume
\[
\xy
(-12,6)*+{A}="0";
(0,6)*+{B}="2";
(12,6)*+{C}="4";
(24,6)*+{}="6";
(-12,-6)*+{A^{\prime}}="10";
(0,-6)*+{B^{\prime}}="12";
(12,-6)*+{C^{\prime}}="14";
(24,-6)*+{}="16";
{\ar^{x} "0";"2"};
{\ar^{y} "2";"4"};
{\ar@{-->}^{\delta} "4";"6"};
{\ar_{a} "0";"10"};
{\ar^{b} "2";"12"};
{\ar_{x^{\prime}} "10";"12"};
{\ar_{y^{\prime}} "12";"14"};
{\ar@{-->}_{\delta^{\prime}} "14";"16"};
{\ar@{}|\circlearrowright "0";"12"};
\endxy
\]
is commutative, from the start. Now {\rm (TR3)} follows from {\rm (ET3)} and Proposition~\ref{PropMorphExtExt}.

{\rm (TR4)} Let
\begin{eqnarray*}
&A\overset{\alpha}{\longrightarrow}B\to D\to A[1],&\\
&B\overset{\beta}{\longrightarrow}C\to F\to B[1],&\\
&A\overset{\gamma}{\longrightarrow}C\to E\to A[1]&
\end{eqnarray*}
be any distinguished triangles in $\widetilde{\mathscr{C}}$ satisfying $\beta\circ\alpha=\gamma$. In a similar way as in the proof of {\rm (TR3)}, we may assume that they are standard triangles associated to $\mathbb{E}$-triangles
\begin{eqnarray*}
&A\overset{f}{\longrightarrow}B\overset{f^{\prime}}{\longrightarrow}D\overset{\delta_f}{\dashrightarrow},&\\
&B\overset{g}{\longrightarrow}C\overset{g^{\prime}}{\longrightarrow}F\overset{\delta_g}{\dashrightarrow},&\\
&A\overset{h}{\longrightarrow}C\overset{h^{\prime}}{\longrightarrow}E\overset{\delta_h}{\dashrightarrow}&
\end{eqnarray*}
satisfying $g\circ f=h$. By Lemma~\ref{LemOctaRigid}, we obtain a commutative diagram made of $\mathbb{E}$-triangles as follows.
\[
\xy
(-21,8)*+{A}="0";
(-7,8)*+{B}="2";
(7,8)*+{D}="4";
(21,8)*+{}="6";
(-21,-7)*+{A}="10";
(-7,-7)*+{C}="12";
(7,-7)*+{E}="14";
(21,-7)*+{}="16";
(-7,-21)*+{F}="22";
(7,-21)*+{F}="24";
(-7,-35)*+{}="32";
(7,-35)*+{}="34";
{\ar^{f} "0";"2"};
{\ar^{f^{\prime}} "2";"4"};
{\ar@{-->}^{\delta_f} "4";"6"};
{\ar@{=} "0";"10"};
{\ar_{g} "2";"12"};
{\ar^{{}^{\exists}d} "4";"14"};
{\ar_{h} "10";"12"};
{\ar_{h^{\prime}} "12";"14"};
{\ar@{-->}^{\delta_h} "14";"16"};
{\ar_{g^{\prime}} "12";"22"};
{\ar^{{}^{\exists}e} "14";"24"};
{\ar@{=} "22";"24"};
{\ar@{-->}_{\delta_g} "22";"32"};
{\ar@{-->}^{f^{\prime}_{\ast}\delta_g} "24";"34"};
{\ar@{}|\circlearrowright "0";"12"};
{\ar@{}|\circlearrowright "2";"14"};
{\ar@{}|\circlearrowright "12";"24"};
\endxy
\]
Thus by Proposition~\ref{PropMorphExtExt}, we obtain the following diagram made of distinguished triangles, as desired.
\[
\xy
(-24,23)*+{A}="0";
(-8,23)*+{B}="2";
(8,23)*+{D}="4";
(24,23)*+{A[1]}="6";
(-24,8)*+{A}="10";
(-8,8)*+{C}="12";
(8,8)*+{E}="14";
(24,8)*+{A[1]}="16";
(-8,-7)*+{F}="22";
(8,-7)*+{F}="24";
(24,-7)*+{B[1]}="26";
(-8,-22)*+{B[1]}="32";
(8,-22)*+{D[1]}="34";
{\ar^{\ell(f)} "0";"2"};
{\ar^{\ell(f^{\prime})} "2";"4"};
{\ar^{\widetilde{\ell}(\delta_f)} "4";"6"};
{\ar@{=} "0";"10"};
{\ar_{\ell(g)} "2";"12"};
{\ar^{\ell(d)} "4";"14"};
{\ar@{=} "6";"16"};
{\ar_{\ell(h)} "10";"12"};
{\ar^{\ell(h^{\prime})} "12";"14"};
{\ar^{\widetilde{\ell}(\delta_h)} "14";"16"};
{\ar_{\ell(g^{\prime})} "12";"22"};
{\ar^{\ell(e)} "14";"24"};
{\ar^{\ell(f)[1]} "16";"26"};
{\ar@{=} "22";"24"};
{\ar_{\widetilde{\ell}(\delta_g)} "24";"26"};
{\ar_{\widetilde{\ell}(\delta_g)} "22";"32"};
{\ar^{\widetilde{\ell}(f^{\prime}_{\ast}\delta_g)} "24";"34"};
{\ar_{\ell(x^{\prime})[1]} "32";"34"};
{\ar@{}|\circlearrowright "0";"12"};
{\ar@{}|\circlearrowright "2";"14"};
{\ar@{}|\circlearrowright "4";"16"};
{\ar@{}|\circlearrowright "12";"24"};
{\ar@{}|\circlearrowright "14";"26"};
{\ar@{}|\circlearrowright "22";"34"};
\endxy
\]
\end{proof}

The following argument ensures the dual arguments concerning distinguished triangles, in the following sections.
Recall that the functor $[-1]\colon\widetilde{\mathscr{C}}\to\widetilde{\mathscr{C}}$ induced from the chosen $\mathbb{E}$-triangle
\[ T^C\overset{t^C}{\longrightarrow}S^C\overset{s^C}{\longrightarrow}C\overset{\lambda^C}{\dashrightarrow}\quad(S^C\in\mathcal{S},\, T^C\in\mathcal{T},\, T^C=C[-1]) \]
for each $C$, gives a quasi-inverse of $[1]$ by Proposition~\ref{Prop50_X7}. Its proof shows that the isomorphisms
\[ \widetilde{\ell}(\lambda^C)\colon C\to T^C[1] \]
give a natural isomorphism $\mathrm{Id}\overset{\cong}{\Longrightarrow}[1]\circ [-1]$.

The dual construction of Definition~\ref{Def50_X2} goes as follows.
\begin{definition}
For any $\mathbb{E}$-triangle $A\overset{x}{\longrightarrow}B\overset{y}{\longrightarrow}C\overset{\delta}{\dashrightarrow}$, take a cospan of morphisms
\begin{equation}\label{cospan}
(T^C\overset{m}{\longrightarrow}E\overset{n}{\longleftarrow}A)
\end{equation}
to some $T\in\mathscr{C}$ satisfying
\begin{equation}\label{cospaneq}
n\in w\mathit{Cof}\quad\text{and}\quad n_{\ast}\delta=m_{\ast}\lambda^C.
\end{equation}
Then, $\ell^{\dag}(\delta)=\ell(n)^{-1}\circ\ell(m)\in\widetilde{\mathscr{C}}(C[-1],A)$ is well-defined.
\end{definition}
With this definition, we can give a triangulation of $\widetilde{\mathscr{C}}$ by requiring
\begin{equation}\label{LeftTria}
T^C\overset{\ell^\dag(\delta)}{\longrightarrow}A\overset{\ell(x)}{\longrightarrow}B\overset{\ell(y)}{\longrightarrow}C
\end{equation}
to be a left triangle.
The following proposition (and its dual) shows that the resulting triangulation is the same as that defined in Definition~\ref{DefTriaLoc}.
\begin{proposition}\label{PropLeftAgree}
For any $\mathbb{E}$-triangle $A\overset{x}{\longrightarrow}B\overset{y}{\longrightarrow}C\overset{\delta}{\dashrightarrow}$,
\begin{equation}\label{LeftTriaToShow}
T^C\overset{\ell^\dag(\delta)}{\longrightarrow}A\overset{\ell(x)}{\longrightarrow}B\overset{\widetilde{\ell}(\lambda^C)\circ\ell(y)}{\longrightarrow}T^C[1]
\end{equation}
becomes a distinguished triangle in $\widetilde{\mathscr{C}}$, with respect to the triangulation given in Definition~\ref{DefTriaLoc}.
\end{proposition}
\begin{proof}
Take the standard triangle $A\overset{\ell(x)}{\longrightarrow}B\overset{\ell(y)}{\longrightarrow}C\overset{\widetilde{\ell}(\delta)}{\longrightarrow}A[1]$. Since $\widetilde{\mathscr{C}}$ is triangulated, by the converse of {\rm (TR2)}, it suffices to show the commutativity of the following diagram.
\[
\xy
(-24,6)*+{A}="0";
(-8,6)*+{B}="2";
(2,-4)*+{}="3";
(14,6)*+{T^C[1]}="4";
(26,-4)*+{}="5";
(36,6)*+{A[1]}="6";
(14,-10)*+{C}="8";
{\ar^{\ell(x)} "0";"2"};
{\ar^(0.46){\widetilde{\ell}(\lambda^C)\circ\ell(y)} "2";"4"};
{\ar^{-\ell^\dag(\delta)[1]} "4";"6"};
{\ar_{\ell(y)} "2";"8"};
{\ar^{\cong}_{\widetilde{\ell}(\lambda^C)} "8";"4"};
{\ar_{\widetilde{\ell}(\delta)} "8";"6"};
{\ar@{}|\circlearrowright "3";"4"};
{\ar@{}|\circlearrowright "4";"5"};
\endxy
\]
As $\ell^\dag(\delta)$ does not depend on the choice of a cospan $(\ref{cospan})$, we may take it in the following way.

By Proposition~\ref{PropBaer} {\rm (1)}, we obtain a commutative diagram made of $\mathbb{E}$-triangles
\[
\xy
(-7,7)*+{T^C}="2";
(7,7)*+{T^C}="4";
(-21,-7)*+{A}="10";
(-7,-7)*+{{}^{\exists}E}="12";
(7,-7)*+{S^C}="14";
(21,-7)*+{}="16";
(-21,-21)*+{A}="20";
(-7,-21)*+{B}="22";
(7,-21)*+{C}="24";
(21,-21)*+{}="26";
(-7,-35)*+{}="32";
(7,-35)*+{}="34";
{\ar@{=} "2";"4"};
{\ar_{k} "2";"12"};
{\ar^{t^C} "4";"14"};
{\ar^{n} "10";"12"};
{\ar_{} "12";"14"};
{\ar@{-->}^{(s^C)^{\ast}\delta} "14";"16"};
{\ar@{=} "10";"20"};
{\ar^{} "12";"22"};
{\ar^{s^C} "14";"24"};
{\ar_{x} "20";"22"};
{\ar_{y} "22";"24"};
{\ar@{-->}_{\delta} "24";"26"};
{\ar@{-->}_{y^{\ast}\lambda^C} "22";"32"};
{\ar@{-->}^{\lambda^C} "24";"34"};
{\ar@{}|\circlearrowright "2";"14"};
{\ar@{}|\circlearrowright "10";"22"};
{\ar@{}|\circlearrowright "12";"24"};
\endxy
\]
satisfying $n_{\ast}\delta+k_{\ast}\lambda^C=0$. Then the cospan $(T^C\overset{-k}{\longrightarrow}E\overset{n}{\longleftarrow}A)$ satisfies the desired property $(\ref{cospaneq})$, and thus gives $\ell^{\dag}(\delta)=-\ell(n)^{-1}\circ\ell(k)$. If we put $\theta=n_{\ast}\delta=-k_{\ast}\lambda^C$, then $(n,\mathrm{id}_C)\colon\delta\to\theta$ and $(-k,\mathrm{id}_C)\colon\lambda^C\to\theta$ are morphisms of $\mathbb{E}$-extensions. Thus
\[
\xy
(-18,7)*+{C}="0";
(0,7)*+{C}="2";
(18,7)*+{C}="4";
(-18,-7)*+{T^C[1]}="10";
(0,-7)*+{E[1]}="12";
(18,-7)*+{A[1]}="14";
{\ar@{=} "0";"2"};
{\ar@{=} "2";"4"};
{\ar_{\widetilde{\ell}(\lambda^C)} "0";"10"};
{\ar^{\widetilde{\ell}(\theta)} "2";"12"};
{\ar^{\widetilde{\ell}(\delta)} "4";"14"};
{\ar_{-\ell(k)[1]} "10";"12"};
{\ar^{\ell(n)[1]} "14";"12"};
{\ar@{}|\circlearrowright "0";"12"};
{\ar@{}|\circlearrowright "2";"14"};
\endxy
\]
becomes commutative by Lemma~\ref{LemMorphExtExt}. This shows
\[ (\ell^{\dag}(\delta)[1])\circ\widetilde{\ell}(\lambda^C)=(-\ell(n)^{-1} [1]\circ\ell(k)[1])\circ\widetilde{\ell}(\lambda^C)=\widetilde{\ell}(\delta). \]
\end{proof}

\section{Reduction and mutation via localization}\label{section_Mut}

\subsection{Happel and Iyama-Yoshino's construction}

\begin{definition}\label{DefFrob}
An extriangulated category $(\mathscr{C},\mathbb{E},\mathfrak{s})$ is said to be {\it Frobenius} if it satisfies the following conditions.
\begin{enumerate}
\item[(1)] $(\mathscr{C},\mathbb{E},\mathfrak{s})$ has enough injectives and enough projectives.
\item[(2)] $\mathrm{Proj}(\mathscr{C})=\mathrm{Inj}(\mathscr{C})$.
\end{enumerate}
\end{definition}

\begin{example}
\begin{enumerate}
\item[(1)] If $(\mathscr{C},\mathbb{E},\mathfrak{s})$ is an exact category, then this agrees with the usual definition (\cite[section I.2]{Ha}).
\item[(2)] Suppose that $\mathcal{T}$ is a triangulated category and $(\mathcal{Z},\mathcal{Z})$ is an $\mathcal{I}$-mutation pair in the sense of \cite[Definition~2.5]{IYo}. 
Then $\mathcal{Z}$ becomes a Frobenius extriangulated category, with the extriangulated structure given in Remark~\ref{RemETrExtClosed}.
\end{enumerate}
\end{example}

\begin{remark}
Let $(\mathscr{C},\mathbb{E},\mathfrak{s})$ be an extriangulated category, as before.
By Remark~\ref{RemCotorsProj}, the following are equivalent.
\begin{enumerate}
\item[(1)] $((\mathcal{I},\mathscr{C}),(\mathscr{C},\mathcal{I}))$ is a twin cotorsion pair for some subcategory $\mathcal{I}\subseteq\mathscr{C}$.
\item[(2)] $(\mathscr{C},\mathbb{E},\mathfrak{s})$ is Frobenius.
\end{enumerate}
Moreover $\mathcal{I}$ in {\rm (1)} should satisfy $\mathcal{I}=\mathrm{Proj}(\mathscr{C})=\mathrm{Inj}(\mathscr{C})$.
\end{remark}

The following can be regarded as a generalization of the constructions by Happel \cite{Ha} and Iyama-Yoshino \cite{IYo}. See also \cite{Li} for the triangulated case.

\begin{corollary}\label{CorStable}
Let $(\mathscr{F},\mathbb{E},\mathfrak{s})$ be a Frobenius extriangulated category satisfying Condition (WIC), with $\mathcal{I}=\mathrm{Inj}(\mathscr{F})$. Then its stable category, namely the ideal quotient $\mathscr{F}/\mathcal{I}$, becomes triangulated.
\end{corollary}
\begin{proof}
Since $((\mathcal{I},\mathscr{F}),(\mathscr{F},\mathcal{I}))$ becomes a Hovey twin cotorsion pair with $\mathrm{Cone}(\mathcal{I},\mathcal{I})=\mathrm{CoCone}(\mathcal{I},\mathcal{I})=\mathcal{I}$, this follows from Corollary~\ref{CorHTCP} and Theorem~\ref{ThmTriaLoc}.
\end{proof}

\begin{remark}
A direct proof for Corollary~\ref{CorStable} is not difficult either, by imitating the proofs by \cite{Ha} or \cite{IYo}, even without assuming Condition (WIC).
\end{remark}

\begin{corollary}
For any category $\mathscr{C}$, the following are equivalent.
\begin{enumerate}
\item[(1)] $(\mathscr{C},\mathbb{E},\mathfrak{s})$ is triangulated, as in Proposition \ref{PropTriaExt}.
\item[(2)] $(\mathscr{C},\mathbb{E},\mathfrak{s})$ is a Frobenius extriangulated category, with $\mathrm{Proj}(\mathscr{C})=\mathrm{Inj}(\mathscr{C})=0$.
\end{enumerate}
\end{corollary}
\begin{proof}
$(1)\Rightarrow(2)$ is trivial. $(2)\Rightarrow(1)$ follows from Corollary \ref{CorStable}.
\end{proof}

\subsection{Mutable cotorsion pairs}

\begin{lemma}\label{LemUWU_TWT}
$\ \ $
\begin{enumerate}
\item[(1)] For any weak equivalence $f\in\mathscr{C}(U,U^{\prime})$ between $U,U^{\prime}\in\mathcal{U}$, there exist $I\in\mathcal{I}$ and $i\in\mathscr{C}(U,I)$, with which
\[ \Big[\raise1.2ex\hbox{\leavevmode\vtop{\baselineskip-8ex \lineskip1ex \ialign{#\crcr{$f$}\crcr{$i$}\crcr}}}\Big]\colon U\to U^{\prime}\oplus I \]
becomes an acyclic cofibration.
\item[(2)] Dually, for any weak equivalence $g\in\mathscr{C}(T,T^{\prime})$ between $T,T^{\prime}\in\mathcal{T}$, there exist $J\in\mathcal{I}$ and $j\in\mathscr{C}(J,T^{\prime})$, with which
\[ [g\ j]\colon T\oplus J\to T^{\prime} \]
becomes an acyclic fibration.
\end{enumerate}
\end{lemma}
\begin{proof}
We only show {\rm (1)}. Since $f\in\mathbb{W}=w\mathit{Fib}\circ w\mathit{Cof}$, there are $\mathbb{E}$-triangles
\begin{eqnarray}
&U\overset{m}{\longrightarrow}E\to S\dashrightarrow,&\label{ET_UWU_TWT1}\\
&V\to E\overset{e}{\longrightarrow}U^{\prime}\overset{\delta}{\dashrightarrow}&\label{ET_UWU_TWT2}
\end{eqnarray}
satisfying $S\in\mathcal{S},\, V\in\mathcal{V}$ and $e\circ m=f$. By $\mathbb{E}(U^{\prime},V)=0$, we have $\delta=0$. Thus we may assume
\[ E=U^{\prime}\oplus V\ \ \text{and}\ \ e=[1\ 0] \]
in $(\ref{ET_UWU_TWT2})$, from the start.

By the extension-closedness of $\mathcal{U}\subseteq\mathscr{C}$, the $\mathbb{E}$-triangle $(\ref{ET_UWU_TWT1})$ gives $U^{\prime}\oplus V=E\in\mathcal{U}$, which implies $V\in\mathcal{I}$. Moreover by $e\circ m=f$, the acyclic cofibration $m\colon U\to U^{\prime}\oplus V$ should be of the form $m=\Big[\raise1.2ex\hbox{\leavevmode\vtop{\baselineskip-8ex \lineskip1ex \ialign{#\crcr{$f$}\crcr{$i$}\crcr}}}\Big]$, with some $i\in\mathscr{C}(U,V)$.
\end{proof}

The following is an immediate consequence of the existence of the model structure.

\begin{remark}\label{RemInvW}
For any morphism $f\in\mathscr{C}(A,B)$ in $\mathscr{C}$, the following are equivalent.
\begin{enumerate}
\item[(1)] $f\in\mathbb{W}$.
\item[(2)] $\ell(f)$ is an isomorphism in $\widetilde{\mathscr{C}}$.
\end{enumerate}
\end{remark}

For any extriangulated category $(\mathscr{C},\mathbb{E},\mathfrak{s})$, let $\mathfrak{CP}(\mathscr{C})$ denote the class of cotorsion pairs on $\mathscr{C}$. Since $\widetilde{\mathscr{C}}$ is triangulated as shown in Theorem~\ref{ThmTriaLoc}, we may use the usual notation $\mathrm{Ext}^1_{\widetilde{\mathscr{C}}}$ for $\widetilde{\mathbb{E}}$.
\begin{definition}
Let $\mathcal{P}=((\mathcal{S},\mathcal{T}),(\mathcal{U},\mathcal{V}))$ be a Hovey twin cotorsion pair on $\mathscr{C}$ and let $\ell\colon\mathscr{C}\to\widetilde{\mathscr{C}}$ be the associated localization functor as before. Define the class of {\it mutable cotorsion pairs on} $\mathscr{C}$ {\it with respect to} $\mathcal{P}$ by
\[ \mathfrak{M}_{\mathcal{P}}=\Set{(\mathcal{A},\mathcal{B})\in\mathfrak{CP}(\mathscr{C})|\begin{array}{c}\mathcal{S}\subseteq\mathcal{A}\subseteq\mathcal{U}\\ \mathcal{V}\subseteq\mathcal{B}\subseteq\mathcal{T}\end{array},\ \mathrm{Ext}^1_{\widetilde{\mathscr{C}}}(\ell(\mathcal{A}),\ell(\mathcal{B}))=0 }. \]

Here, $\ell(\mathcal{A}),\ell(\mathcal{B})\subseteq\widetilde{\mathscr{C}}$ denote the essential images of $\mathcal{A},\mathcal{B}$ under $\ell$.
Remark that $\mathcal{S}\subseteq\mathcal{A}$ is equivalent to $\mathcal{B}\subseteq\mathcal{T}$, and $\mathcal{A}\subseteq\mathcal{U}$ is equivalent to $\mathcal{V}\subseteq\mathcal{B}$, for any $(\mathcal{A},\mathcal{B})\in\mathfrak{CP}(\mathscr{C})$.
\end{definition}

\begin{theorem}\label{ThmRecoll}
For any Hovey twin cotorsion pair $\mathcal{P}=((\mathcal{S},\mathcal{T}),(\mathcal{U},\mathcal{V}))$ on $\mathscr{C}$, we have mutually inverse bijective correspondences
\[ \mathbb{R}=\mathbb{R}_{\mathcal{P}}\colon\mathfrak{M}_{\mathcal{P}}\to \mathfrak{CP}(\widetilde{\mathscr{C}}), \]
\[\mathbb{I}=\mathbb{I}_{\mathcal{P}}\colon\mathfrak{CP}(\widetilde{\mathscr{C}})\to\mathfrak{M}_{\mathcal{P}} \]
given by
\[ \mathbb{R}((\mathcal{A},\mathcal{B}))=(\ell(\mathcal{A}),\ell(\mathcal{B})), \]
\[ \mathbb{I}((\mathcal{L},\mathcal{R}))=(\mathcal{U}\cap\ell^{-1}(\mathcal{L}),\mathcal{T}\cap\ell^{-1}(\mathcal{R})). \]
\end{theorem}
\begin{proof}
It suffices to show the following.
\begin{enumerate}
\item[(1)] For any $(\mathcal{A},\mathcal{B})\in\mathfrak{M}_{\mathcal{P}}$, we have $\mathbb{R}((\mathcal{A},\mathcal{B}))\in\mathfrak{CP}(\widetilde{\mathscr{C}})$.
\item[(2)] For any $(\mathcal{L},\mathcal{R})\in\mathfrak{CP}(\widetilde{\mathscr{C}})$, we have $\mathbb{I}((\mathcal{L},\mathcal{R}))\in\mathfrak{M}_{\mathcal{P}}$.
\item[(3)] $\mathbb{I}\circ\mathbb{R}=\mathrm{id}$.
\item[(4)] $\mathbb{R}\circ\mathbb{I}=\mathrm{id}$.
\end{enumerate}
To distinguish, in this proof, let $\ell(X)\in\widetilde{\mathscr{C}}$ denote the image under $\ell$ of an object $X\in\mathscr{C}$. 


{\rm (1)} Since $\ell(\mathcal{A})$ is the essential image of $\mathcal{A}$ under $\ell$, it is closed under isomorphisms and finite direct sums. $\mathscr{C}=\mathrm{Cone}(\mathcal{B},\mathcal{A})$ implies $\widetilde{\mathscr{C}}=\ell(\mathcal{A})\ast\ell(\mathcal{B})[1]$, by Definition~\ref{DefTriaLoc}. $\mathrm{Ext}^1_{\widetilde{\mathscr{C}}}(\ell(\mathcal{A}),\ell(\mathcal{B}))=0$ follows from the definition of $\mathfrak{M}_{\mathcal{P}}$.

It remains to show that $\ell(\mathcal{A}),\ell(\mathcal{B})\subseteq\widetilde{\mathscr{C}}$ are closed under direct summands. To show that $\ell(\mathcal{A})\subseteq\widetilde{\mathscr{C}}$ is closed under direct summands, it suffices to show $\ell(\mathcal{A})={}^{\perp}\ell(\mathcal{B})[1]$. Take any $X\in\mathscr{C}$, and suppose it satisfies $\mathrm{Ext}^1_{\widetilde{\mathscr{C}}}(\ell(X),\ell(\mathcal{B}))=0$.

Let us show $\ell(X)\in\ell(\mathcal{A})$. By a cofibrant replacement, we may assume $X$ belongs to $\mathcal{U}$. Resolve $X$ by an $\mathbb{E}$-triangle in $\mathscr{C}$
\[ B\to A\to X\overset{\delta}{\dashrightarrow} \quad(A\in\mathcal{A},B\in\mathcal{B}). \]
Since $\ell(\delta)=0$ by assumption, we obtain $\delta=0$ by Lemma~\ref{Lem35b}. Thus $X$ is a direct summand of $A$, which implies that $X$ itself belongs to $\mathcal{A}$.
Similarly for $\ell(\mathcal{B})\subseteq\widetilde{\mathscr{C}}$.


(2) Put $\mathcal{A}=\mathcal{U}\cap\ell^{-1}(\mathcal{L})$, $\mathcal{B}=\mathcal{T}\cap\ell^{-1}(\mathcal{R})$. Since both $\mathcal{U}$ and $\ell^{-1}(\mathcal{L})$ are closed under isomorphisms, finite direct sums and direct summands, so is their intersection $\mathcal{A}$. Similarly for $\mathcal{B}$.
By $\ell(\mathcal{S})\subseteq\ell(\mathcal{N})=0$, we have $\mathcal{S}\subseteq\mathcal{A}\subseteq\mathcal{U}$. By Lemma~\ref{Lem35b}, $\mathrm{Ext}^1_{\widetilde{\mathscr{C}}}(\ell(\mathcal{A}),\ell(\mathcal{B}))=0$ implies $\mathbb{E}(\mathcal{A},\mathcal{B})=0$.
It remains to show $\mathscr{C}=\mathrm{Cone}(\mathcal{B},\mathcal{A})=\mathrm{CoCone}(\mathcal{B},\mathcal{A})$. Let us show $\mathscr{C}=\mathrm{Cone}(\mathcal{B},\mathcal{A})$.

Let $X\in\mathscr{C}$ be any object. By the assumption, there exist $R\in\ell^{-1}(\mathcal{R}),\, L\in\ell^{-1}(\mathcal{L})$ and a distinguished triangle
\[ \ell(R)\to \ell(L)\to \ell(X)\to \ell(R)[1] \]
in $\widetilde{\mathscr{C}}$. By definition, it is isomorphic to the standard triangle associated to an $\mathbb{E}$-triangle, which we may assume to be of the form
\begin{equation}\label{ET_36A}
R_0\overset{x}{\longrightarrow}L_0\overset{y}{\longrightarrow}Z\overset{\delta}{\dashrightarrow}
\end{equation}
satisfying $R_0,L_0,Z\in\mathcal{Z}$, by a fibrant-cofibrant replacement (Corollary~\ref{CorCofibReplStan} and Remark~\ref{RemFibReplStan}).
Thus we have an $\mathbb{E}$-triangle $(\ref{ET_36A})$ satisfying $R_0\in\mathcal{Z}\cap\ell^{-1}(\mathcal{R}),\, L_0\in\mathcal{Z}\cap\ell^{-1}(\mathcal{L})$ and $Z\in\mathcal{Z}$, together with an  isomorphism $\zeta\colon \ell(Z)\overset{\cong}{\longrightarrow}\ell(X)$ in $\widetilde{\mathscr{C}}$.
Resolve $X$ by an $\mathbb{E}$-triangle
\[ X\overset{t_X}{\longrightarrow}T_X\overset{s_X}{\longrightarrow}S_X\overset{\rho_X}{\dashrightarrow}\quad(T_X\in\mathcal{T},\, S_X\in\mathcal{S}). \]
Then there exists a morphism $z\in\mathscr{C}(Z,T_X)$ which satisfies $\zeta=\ell(t_X)^{-1}\circ\ell(z)$. Since $\zeta$ is an isomorphism, it follows that $z\in\mathbb{W}$. By Lemma~\ref{LemUWU_TWT} {\rm (2)}, there exists an $\mathbb{E}$-triangle
\[ V\to Z\oplus I\overset{[z\ i]}{\longrightarrow}T_X\dashrightarrow\quad(V\in\mathcal{V},\, I\in\mathcal{I}). \]

On the other hand by {\rm (ET2)}, we have an $\mathbb{E}$-triangle
\[ R_0\overset{x_0}{\longrightarrow}L_0\oplus I\overset{y_0}{\longrightarrow}Z\oplus I\overset{}{\dashrightarrow} \]
from $(\ref{ET_36A})$, where $x_0=\Big[\raise1ex\hbox{\leavevmode\vtop{\baselineskip-8ex \lineskip1ex \ialign{#\crcr{$x$}\crcr{$0$}\crcr}}}\Big]$, $y_0=y\oplus\mathrm{id}_I$. Thus by {\rm (ET4)$^{\mathrm{op}}$}, we obtain a diagram
\begin{eqnarray*}
&\xy
(-24,8)*+{R_0}="0";
(-8,8)*+{{}^{\exists}E}="2";
(8,8)*+{V}="4";
(-24,-7)*+{R_0}="10";
(-8,-7)*+{L_0\oplus I}="12";
(8,-7)*+{Z\oplus I}="14";
(-8,-22)*+{T_X}="22";
(8,-22)*+{T_X}="24";
{\ar^{{}^{\exists}r} "0";"2"};
{\ar^{} "2";"4"};
{\ar@{=} "0";"10"};
{\ar_{{}^{\exists}e} "2";"12"};
{\ar^{} "4";"14"};
{\ar_(0.42){x_0} "10";"12"};
{\ar_{y_0} "12";"14"};
{\ar_{d} "12";"22"};
{\ar^{[z\ i]} "14";"24"};
{\ar@{=} "22";"24"};
{\ar@{}|\circlearrowright "0";"12"};
{\ar@{}|\circlearrowright "2";"14"};
{\ar@{}|\circlearrowright "12";"24"};
\endxy &\\
&(d=[z\ \ i]\circ y_0=[z\circ y\ \ i])&
\end{eqnarray*}
made of conflations.
Since $\ell(r)$ is an isomorphism in $\widetilde{\mathscr{C}}$, we have $E\in\ell^{-1}(\mathcal{R})$. Besides, $R_0,V\in\mathcal{T}$ implies $E\in\mathcal{T}$.

By {\rm (ET4)$^{\mathrm{op}}$}, we obtain a diagram
\[
\xy
(-24,8)*+{E}="0";
(-8,8)*+{{}^{\exists}F}="2";
(8,8)*+{X}="4";
(-24,-7)*+{E}="10";
(-8,-7)*+{L_0\oplus I}="12";
(8,-7)*+{T_X}="14";
(-8,-22)*+{S_X}="22";
(8,-22)*+{S_X}="24";
{\ar^{} "0";"2"};
{\ar^{} "2";"4"};
{\ar@{=} "0";"10"};
{\ar_{{}^{\exists}f} "2";"12"};
{\ar^{t_X} "4";"14"};
{\ar_(0.42){e} "10";"12"};
{\ar_(0.56){d} "12";"14"};
{\ar_{} "12";"22"};
{\ar^{s_X} "14";"24"};
{\ar@{=} "22";"24"};
{\ar@{}|\circlearrowright "0";"12"};
{\ar@{}|\circlearrowright "2";"14"};
{\ar@{}|\circlearrowright "12";"24"};
\endxy
\]
made of conflations.
Since $\ell(f)$ is an isomorphism in $\widetilde{\mathscr{C}}$, this shows $F\in\ell^{-1}(\mathcal{L})$. Resolve $F$ by an $\mathbb{E}$-triangle
\[ V^F\to U^F\to F\dashrightarrow\quad(U^F\in\mathcal{U},\, V^F\in\mathcal{V}). \]
By {\rm (ET4)$^{\mathrm{op}}$}, we obtain a diagram
\[
\xy
(-21,7)*+{V^F}="0";
(-7,7)*+{{}^{\exists}G}="2";
(7,7)*+{E}="4";
(-21,-7)*+{V^F}="10";
(-7,-7)*+{U^F}="12";
(7,-7)*+{F}="14";
(-7,-21)*+{X}="22";
(7,-21)*+{X}="24";
{\ar^{} "0";"2"};
{\ar^{} "2";"4"};
{\ar@{=} "0";"10"};
{\ar_{} "2";"12"};
{\ar^{} "4";"14"};
{\ar_{} "10";"12"};
{\ar_{} "12";"14"};
{\ar_{} "12";"22"};
{\ar^{} "14";"24"};
{\ar@{=} "22";"24"};
{\ar@{}|\circlearrowright "0";"12"};
{\ar@{}|\circlearrowright "2";"14"};
{\ar@{}|\circlearrowright "12";"24"};
\endxy
\]
made of conflations. Then in the $\mathbb{E}$-triangle $G\to U^F\to X\dashrightarrow$, we have $U^F\in\mathcal{U}\cap\ell^{-1}(\mathcal{L})=\mathcal{A}$ and $G\in\mathcal{T}\cap\ell^{-1}(\mathcal{R})=\mathcal{B}$.


(3) For any $(\mathcal{A},\mathcal{B})\in\mathfrak{M}_{\mathcal{P}}$, we have
\[ \mathbb{I}\circ\mathbb{R}((\mathcal{A},\mathcal{B}))=(\mathcal{U}\cap\ell^{-1}(\ell(\mathcal{A})),\mathcal{T}\cap\ell^{-1}(\ell(\mathcal{B}))). \]
Obviously, $\mathcal{A}\subseteq\mathcal{U}\cap\ell^{-1}(\ell(\mathcal{A}))$ and $\mathcal{B}\subseteq\mathcal{T}\cap\ell^{-1}(\ell(\mathcal{B}))$ hold. Since both $(\mathcal{A},\mathcal{B})$ and $\mathbb{I}\circ\mathbb{R}((\mathcal{A},\mathcal{B}))$ are cotorsion pairs, this means $(\mathcal{A},\mathcal{B})=\mathbb{I}\circ\mathbb{R}((\mathcal{A},\mathcal{B}))$.


(4) For any $(\mathcal{L},\mathcal{R})\in\mathfrak{CP}(\widetilde{\mathscr{C}})$, we have
\[ \mathbb{R}\circ\mathbb{I}((\mathcal{L},\mathcal{R}))=(\ell(\mathcal{U}\cap\ell^{-1}(\mathcal{L})),\ell(\mathcal{T}\cap\ell^{-1}(\mathcal{R}))), \]
which obviously satisfies $\ell(\mathcal{U}\cap\ell^{-1}(\mathcal{L}))\subseteq\mathcal{L}$ and $\ell(\mathcal{T}\cap\ell^{-1}(\mathcal{R}))\subseteq\mathcal{R}$. Similarly as in (3), it follows that $(\mathcal{L},\mathcal{R})=\mathbb{R}\circ\mathbb{I}((\mathcal{L},\mathcal{R}))$.
\end{proof}

\begin{claim}\label{ClaimConeMut}
For any $(\mathcal{A},\mathcal{B})\in\mathfrak{CP}(\mathscr{C})$ satisfying $\mathcal{S}\subseteq\mathcal{A}\subseteq\mathcal{U}$ (or equivalently $\mathcal{V}\subseteq\mathcal{B}\subseteq\mathcal{T}$), we have
\begin{eqnarray*}
&\mathcal{U}\cap\ell^{-1}\ell(\mathcal{A})=\mathcal{U}\cap\mathrm{CoCone}(\mathcal{A},\mathcal{S}),&\\
&\mathcal{T}\cap\ell^{-1}\ell(\mathcal{B})=\mathcal{T}\cap\mathrm{Cone}(\mathcal{V},\mathcal{B}).&
\end{eqnarray*}
\end{claim}
\begin{proof}
We only show $\mathcal{U}\cap\ell^{-1}\ell(\mathcal{A})=\mathcal{U}\cap\mathrm{CoCone}(\mathcal{A},\mathcal{S})$.

$\mathcal{U}\cap\mathrm{CoCone}(\mathcal{A},\mathcal{S})\subseteq\mathcal{U}\cap\ell^{-1}\ell(\mathcal{A})$ is obvious. For the converse, let $U\in\mathcal{U}$ be any object satisfying $\ell(U)\cong \ell(A)$ in $\widetilde{\mathscr{C}}$ for some $A\in\mathcal{A}$. Resolve $U$ by an $\mathbb{E}$-triangle
\[ U\to Z\to S\dashrightarrow\quad(Z\in\mathcal{Z},\, S\in\mathcal{S}). \]
Then $\ell(A)\cong \ell(Z)$ holds in $\widetilde{\mathscr{C}}$. Since $A\in\mathcal{U},Z\in\mathcal{T}$, there is a morphism $f\in\mathscr{C}(A,Z)$ which gives the isomorphism $\ell(f)\colon\ell(A)\to\ell(Z)$ by Remark~\ref{RemUTMono}. Factorize this $f\in\mathbb{W}$ as
\[ f=h\circ g\quad(g\in w\mathit{Cof},\, h\in w\mathit{Fib}). \]
By definition, we have $\mathbb{E}$-triangles
\begin{eqnarray*}
V_0\to E\overset{h}{\longrightarrow}Z\dashrightarrow&&(V_0\in\mathcal{V}),\\
A\overset{g}{\longrightarrow}E\to S_0\dashrightarrow&&(S_0\in\mathcal{S}).
\end{eqnarray*}
Since $\mathbb{E}(Z,V_0)=0$, it follows that $V_0\oplus Z\cong E\in\mathcal{A}$, which implies $Z\in\mathcal{A}$.
\end{proof}

The class $\mathfrak{M}_{\mathcal{P}}$ can be rewritten as follows.
\begin{corollary}\label{CorCondExt}
Let $\mathcal{P}$ be a Hovey twin cotorsion pair on $\mathscr{C}$. For any $(\mathcal{A},\mathcal{B})\in\mathfrak{CP}(\mathscr{C})$ satisfying $\mathcal{S}\subseteq\mathcal{A}\subseteq\mathcal{U}$ (or equivalently $\mathcal{V}\subseteq\mathcal{B}\subseteq\mathcal{T}$), the following are equivalent.
\begin{enumerate}
\item[{\rm (1)}] $(\mathcal{A},\mathcal{B})\in\mathfrak{M}_{\mathcal{P}}$ i.e., it satisfies $\mathrm{Ext}^1_{\widetilde{\mathscr{C}}}(\ell(\mathcal{A}),\ell(\mathcal{B}))=0$
\item[{\rm (2)}] $\mathcal{U}\cap\ell^{-1}\ell(\mathcal{A})=\mathcal{A}$.
\item[{\rm (2)$^{\prime}$}] $\mathcal{U}\cap\ell^{-1}\ell(\mathcal{A})\subseteq\mathcal{A}$.
\item[{\rm (3)}] $\mathcal{T}\cap\ell^{-1}\ell(\mathcal{B})=\mathcal{B}$.
\item[{\rm (3)$^{\prime}$}] $\mathcal{T}\cap\ell^{-1}\ell(\mathcal{B})\subseteq\mathcal{B}$.
\end{enumerate}
Thus by Claim~\ref{ClaimConeMut}, we have
\begin{eqnarray*}
\mathfrak{M}_{\mathcal{P}}&=&\{(\mathcal{A},\mathcal{B})\in\mathfrak{CP}(\mathscr{C})\mid \mathcal{S}\subseteq\mathcal{A}\subseteq\mathcal{U},\ \mathcal{U}\cap\mathrm{CoCone}(\mathcal{A},\mathcal{S})\subseteq\mathcal{A}\}\\
&=&\{(\mathcal{A},\mathcal{B})\in\mathfrak{CP}(\mathscr{C})\mid \mathcal{V}\subseteq\mathcal{B}\subseteq\mathcal{T},\ \mathcal{T}\cap\mathrm{Cone}(\mathcal{V},\mathcal{B})\subseteq\mathcal{B}\}.
\end{eqnarray*}
\end{corollary}
\begin{proof}
We only show $(1)\Leftrightarrow(2)\Leftrightarrow(2)^{\prime}$. 

$(1)\Rightarrow (2)$ follows from Theorem~\ref{ThmRecoll}.

$(2)\Rightarrow (2)^{\prime}$ is obvious.
It remains to show $(2)^{\prime}\Rightarrow (1)$.

Suppose $(2)^{\prime}$ is satisfied. Let us show $\mathrm{Ext}^1_{\widetilde{\mathscr{C}}}(\ell(A),\ell(B))=0$ for any pair of objects $A\in\mathcal{A},\, B\in\mathcal{B}$.
Resolve $A$ by an $\mathbb{E}$-triangle
\[ T^A\overset{t}{\longrightarrow}S^A\to A\dashrightarrow\quad(T^A\in\mathcal{T},\, S^A\in\mathcal{S}) \]
and $T_A$ by
\[ V^T\to Z^T\overset{z}{\longrightarrow}T^A\dashrightarrow\quad(Z^T\in\mathcal{Z},\, V^T\in\mathcal{V}). \]
Then we have $\ell(A)[-1]\cong\ell(T^A)\cong\ell(Z^T)$, and thus
\[ \mathrm{Ext}^1_{\widetilde{\mathscr{C}}}(\ell(A),\ell(B))\cong\widetilde{\mathscr{C}}(\ell(Z^T),\ell(B))\cong(\mathscr{C}/\mathcal{I})(Z^T,B) \]by Remark~\ref{RemUTMono}. Let us show $(\mathscr{C}/\mathcal{I})(Z^T,B)=0$.
Factorize $t\circ z$ as
\[ t\circ z=h\circ g\quad(g\in\mathit{Cof},\, h\in w\mathit{Fib}), \]
to obtain a diagram
\[
\xy
(-7,15)*+{V^T}="-12";
(7,15)*+{{}^{\exists}V_0}="-14";
(-7,2)*+{Z^T}="2";
(7,2)*+{{}^{\exists}E}="4";
(21,2)*+{{}^{\exists}U_0}="6";
(-7,-12)*+{T^A}="12";
(7,-12)*+{S^A}="14";
(21,-12)*+{A}="16";
{\ar_{} "-12";"2"};
{\ar^{} "-14";"4"};
{\ar^{g} "2";"4"};
{\ar^{} "4";"6"};
{\ar_{z} "2";"12"};
{\ar^{h} "4";"14"};
{\ar_{t} "12";"14"};
{\ar_{} "14";"16"};
{\ar@{}|\circlearrowright "2";"14"};
\endxy\qquad(U_0\in\mathcal{U},\, V_0\in\mathcal{V})
\]
made of conflations.
Since $\mathbb{E}(S^A,V_0)=0$, we have $E\cong S^A\oplus V_0$. Besides, by the extension-closedness of $\mathcal{U}\subseteq\mathscr{C}$, we have $E\in\mathcal{U}$, which shows $V_0\in\mathcal{V}\cap\mathcal{U}=\mathcal{I}$. Thus it follows that $E\in\mathcal{S}$.

By Lemma~\ref{LemNine}, we obtain $X\in\mathscr{C}$ and conflations
\[ V^T\to V_0\to X\ \ \text{and}\ \ X\to U_0\overset{u}{\longrightarrow} A. \]
By Lemma~\ref{LemConcentric1} {\rm (1)}, we have $X\in\mathcal{N}$, and thus $u\in\mathbb{W}$ by Claim~\ref{ClaimNW}. This shows $\ell(U_0)\cong\ell(A)$, which means $U_0\in\mathcal{U}\cap\ell^{-1}\ell(\mathcal{A})\subseteq\mathcal{A}$ by assumption. Thus we obtain exact sequence
\[ \mathscr{C}(E,B)\overset{\mathscr{C}(g,B)}{\longrightarrow}\mathscr{C}(Z^T,B)\to\mathbb{E}(U_0,B)=0. \]
This shows that any morphism $f\in\mathscr{C}(Z^T,B)$ factors through $E\in\mathcal{S}$, and thus $\overline{f}=0$ holds in $(\mathscr{C}/\mathcal{I})(Z^T,B)$.
\end{proof}

The above arguments allows us to define the mutation of cotorsion pairs as follows.
\begin{definition}
Let $\mathcal{P}=((\mathcal{S},\mathcal{T}),(\mathcal{U},\mathcal{V}))$ be a Hovey twin cotorsion pair on $\mathscr{C}$. We define {\it mutation with respect to} $(\mathcal{S},\mathcal{V})$ as a $\mathbb{Z}$-action on $\mathfrak{M}_{\mathcal{P}}$ given by $\mu_n=\mathbb{I}\circ [n]\circ \mathbb{R}$, i.e.,
\[ \mu_n\colon\mathfrak{M}_{\mathcal{P}}\to\mathfrak{M}_{\mathcal{P}}\ ;\ (\mathcal{A},\mathcal{B})\mapsto \big(\mathcal{U}\cap\ell^{-1}(\ell(\mathcal{A})[n]),\mathcal{T}\cap\ell^{-1}(\ell(\mathcal{B})[n])\big) \]
for any $n\in\mathbb{Z}$.
\end{definition}


\begin{thebibliography}{BBD}


\bibitem[AN]{AN} Abe, N.; Nakaoka, H.: \emph{General heart construction on a triangulated category (II): associated homological functor}, Appl. Categ. Structures, \textbf{20} (2012) no.2, 161--174.

\bibitem[BBD]{BBD} Be\u{\i}linson, A. A.; Bernstein, J.; Deligne, P.: \emph{Faisceaux pervers} (French) [Perverse sheaves] Analysis and topology on singular spaces, I (Luminy, 1981), 5--171, Ast\'{e}risque, \textbf{100}, Soc. Math. France, Paris, 1982.

\bibitem[Bo]{Bo}Borceux, F.: \emph{Handbook of categorical algebra 1, Basic category theory}, Encyclopedia of Mathematics and its Applications, \textbf{50}. Cambridge University Press, Cambridge, 1994. xvi+345 pp.

\bibitem[Bu]{Bu} B\"uhler, T.: \emph{Exact Categories}, Expo. Math. \textbf{28} (2010) 1--69.

\bibitem[DI]{DI} Demonet, L.; Iyama, O.: \emph{Lifting preprojective algebras to orders and categorifying partial flag varieties}, Algebra Number Theory \textbf{10} (2016) no. 7, 1527--1580.

\bibitem[G]{G} Gillespie, J.: \emph{Model structures on exact categories}. J. Pure Appl. Algebra \textbf{215} (2011) no. 12, 2892--2902.

\bibitem[Ha]{Ha} Happel, D.: \emph{Triangulated categories in the representation theory of finite-dimensional algebras}, London Mathematical Society Lecture Note Series, \textbf{119}. Cambridge University Press, Cambridge, 1988. x+208 pp.

\bibitem[Ho1]{Ho1} Hovey, M.: \emph{Cotorsion pairs, model category structures, and representation theory}, Math. Z. \textbf{241} (2002) no. 3, 553--592.

\bibitem[Ho2]{Ho2} Hovey, M.: \emph{Cotorsion pairs and model categories. Interactions between homotopy theory and algebra}, 277--296, Contemp. Math., \textbf{436}, Amer. Math. Soc., Providence, RI, 2007.

\bibitem[IYa]{IYa} Iyama, O.; Yang, D.: \emph{Silting reduction and Calabi--Yau reduction of triangulated categories}, arXiv:1408.2678.

\bibitem[IYo]{IYo} Iyama, O.; Yoshino, Y.: \emph{Mutation in triangulated categories and rigid Cohen-Macaulay modules}, Invent. Math. \textbf{172} (2008) no. 1, 117--168. 

\bibitem[Joy]{Joy} Joyal, A.: \emph{The theory of quasi-categories and its applications}, https://web.archive.org/web/20110706214636/http://www.crm.cat/HigherCategories/hc2.pdf.

\bibitem[Ke]{Ke} Keller, B.: \emph{Derived categories and their uses}. Handbook of algebra, Vol. 1, 671--701, North-Holland, Amsterdam, 1996. 

\bibitem[KR]{KR} Keller, B.; Reiten, I.: \emph{Cluster-tilted algebras are Gorenstein and stably Calabi-Yau}, Adv. Math. \textbf{211} (2007) no. 1, 123--151.

\bibitem[KZ]{KZ} Koenig, S.; Zhu, B.: \emph{From triangulated categories to abelian categories: cluster tilting in a general framework}, Math. Z. \textbf{258} (2008) no. 1, 143--160.

\bibitem[Li]{Li} Li, Z.: \emph{The model structure of Iyama-Yoshino's subfactor triangulated categories}, arXiv:1510.02258.

\bibitem[Liu]{Liu} Liu, Y.: \emph{Hearts of twin cotorsion pairs on exact categories}, J. Algebra, \textbf{394} (2013) 245--284.

\bibitem[LN]{LN} Liu, Y; Nakaoka, H.: \emph{Hearts of twin Cotorsion pairs on extriangulated categories}. arXiv:1702.00244.

\bibitem[MP]{MP} Marsh, R.; Palu, Y.: \emph{Nearly Morita equivalences and rigid objects}, Nagoya Math. J. \textbf{225} (2017) 64--99.

\bibitem[Na1]{Na1} Nakaoka, H.: \emph{General heart construction on a triangulated category (I): unifying $t$-structures and cluster tilting subcategories}, Appl. Categ. Structures, \textbf{19} (2011) no.6, 879--899.

\bibitem[Na2]{Na2} Nakaoka, H.: \emph{General heart construction for twin torsion pairs on triangulated categories}, J. Algebra, \textbf{374} (2013) 195--215.

\bibitem[Na3]{Na3} Nakaoka, H.: \emph{Equivalence of hearts of twin cotorsion pairs on triangulated categories}, Comm. Algebra \textbf{44} (2016) no. 10, 4302--4326.

\bibitem[Na4]{Na4} Nakaoka, H.: \emph{A simultaneous generalization of mutation and recollement on a triangulated category}, to appear in Appl. Categ. Structures, \verb+doi:10.1007/s10485-017-9501-3+.

\bibitem[Ne]{Ne} Neeman, A.: \emph{Triangulated categories}. Annals of Mathematics Studies, 148. Princeton University Press, Princeton, NJ, 2001. viii+449 pp. 

\bibitem[Pal]{Pal} Palu, Y.: \emph{From triangulated categories to module categories via homotopical algebra}, arXiv:1412.7289.

\bibitem[Pau]{Pau} Pauksztello, D.: \emph{Compact corigid objects in triangulated categories and co-$t$-structures}. Cent. Eur. J. Math. \textbf{6} (2008) no. 1, 25--42.

\bibitem[Sal]{Sal} Salce, L.: \emph{Cotorsion theories for abelian groups} in Symposia Mathematica, Vol. XXIII (Conf. Abelian Groups and their Relationship to the Theory of Modules, INDAM, Rome, 1977) (1979), 11--32.

\bibitem[S]{S} \v{S}\v{t}ov\'{\i}\v{c}ek, J.: \emph{Exact model categories, approximation theory, and cohomology of quasi-coherent sheaves}. Advances in representation theory of algebras, 297--367, EMS Ser. Congr. Rep., Eur. Math. Soc., Z\"{u}rich, 2013. 

\bibitem[Y]{Y} Yang, X.: \emph{Model structures on triangulated categories}. Glasg. Math. J. \textbf{57} (2015) no. 2, 263--284. 

\end{thebibliography}
\end{document}